\documentclass[11pt,letterpaper]{amsart}

\usepackage{amscd,amsmath,amssymb,euscript}
\usepackage[frame,cmtip,curve,arrow,matrix,line,graph]{xy}
\usepackage{tikz-cd}
\usepackage{bigints}
\usetikzlibrary{matrix}
\usepackage{hyperref}
\usepackage{stmaryrd}
\usepackage[bottom]{footmisc}
\usepackage{enumitem}
\usepackage{graphicx}
\usepackage{flafter}

\title[Categorical and K-theoretic DT of $\mathbb{C}^3$ (part I)]{Categorical and K-theoretic Donaldson-Thomas theory of $\mathbb{C}^3$ (part I)}
\author{Tudor P\u adurariu and Yukinobu Toda}

\makeatletter
\@namedef{subjclassname@2020}{%
  \textup{2020} Mathematics Subject Classification}
\makeatother

\newtheorem{thm}{Theorem}[section]
\newtheorem{cor}[thm]{Corollary}

\newtheorem{prop}[thm]{Proposition}
\newtheorem{lemma}[thm]{Lemma}

\theoremstyle{definition}
\newtheorem{defn}[thm]{Definition}

\newtheorem{thm*}[thm]{Theorem$^*$}
\newtheorem{remark}[thm]{Remark}

\newtheorem{example}[thm]{Example}

\newtheorem{step}{Step}

\newcommand{\comment}[1]{}

\renewcommand{\leq}{\leqslant}
\renewcommand{\geq}{\geqslant}

\newcommand{\OO}{\mathcal{O}}

\newcommand{\X}{\mathcal{X}}

\newcommand{\Hilb}{\operatorname{Hilb}}

\newcommand{\id}{\operatorname{id}}

\newcommand{\Hom}{\operatorname{Hom}}

\newcommand{\Spec}{\operatorname{Spec}}

\newcommand{\C}{\mathbb{C}}

\newcommand{\ee}{\underline{e}}
\newcommand{\dd}{\underline{d}}

\newcommand{\relmiddle}[1]{\mathrel{}\middle#1\mathrel{}}

\newcommand{\Tr}{\mathop{\mathrm{Tr}}\nolimits}

\makeatletter

\@addtoreset{equation}{section}	
\makeatother

\usetikzlibrary{positioning,fit,calc}
\tikzstyle{block}=[draw=black, width=1cm, minimum height=2cm, align=center] 
\tikzstyle{block2}=[draw=black, text width=2cm, minimum height=1cm, align=center] 
\tikzstyle{block3}=[draw=black, text width=2cm, minimum height=1cm, align=center] 

\begin{document}
\maketitle

\begin{abstract}
    We begin the study of categorifications of Donaldson-Thomas 
    invariants 
    associated with Hilbert schemes of points on the three-dimensional affine space, which 
    we call DT categories. 
    The DT category is defined to be the category of matrix factorizations
    on the non-commutative Hilbert scheme with a super-potential
    whose critical locus is the Hilbert scheme of points.
    
    The first main result in this paper is the construction of semiorthogonal decompositions
    of DT categories,
    which can be regarded as categorical wall-crossing formulae
    of the framed triple loop quiver. 
    Each summand
    is given by the categorical Hall product of some subcategories of 
    matrix factorizations, called quasi-BPS 
    categories. 
    They are categories of matrix factorizations on twisted versions of noncommutative resolutions of singularities considered by \v{S}penko--Van den Bergh,
    and were used by the first author to
    prove a PBW theorem for K-theoretic Hall algebras.
    
We next construct explicit objects of quasi-BPS categories via Koszul duality 
equivalences, and show that they form a basis in the torus localized K-theory.
These computations may be regarded as a numerical K-theoretic analogue 
in dimension three of the McKay correspondence for 
Hilbert schemes of points. 
In particular, the torus localized K-theory of DT categories
has a basis whose cardinality is the number of plane partitions, giving a K-theoretic
analogue of MacMahon's formula. 
\end{abstract}

\section{Introduction}

\renewcommand{\thefootnote}{\fnsymbol{footnote}} 
\footnotetext{\emph{$2020$ Mathematics Subject Classification.} Primary: 14N35, 18N25. Secondary: 19E08.}     
\renewcommand{\thefootnote}{\arabic{footnote}}

\subsection{Categorical DT theory for Hilbert schemes of points}
For a variety $X$, the Hilbert scheme of points
\begin{align*}
	\Hilb(X, d)
	\end{align*}
parametrizes zero-dimensional closed subschemes in 
$X$ with length $d$. 
When $\dim X \leq 2$, 
$\Hilb(X, d)$ is a smooth variety 
 and an important geometric object 
in classical algebraic geometry, representation theory, and 
mathematical physics (cf.~\cite{NLecture}).
The derived category of coherent sheaves on 
$\Hilb(X, d)$ when $\dim X \leq 2$ has been extensively studied in the literature, see Subsection~\ref{subsec:catHilb}. 
The purpose of this paper is to study categories of coherent sheaves associated to 
$\Hilb(X, d)$ for $X=\mathbb{C}^3$ from the viewpoint of 
categorification of Donaldson-Thomas (DT) invariants and 
their wall-crossing formulae. 

Hilbert schemes in higher dimensions are very singular 
 in general (e.g. not irreducible, non-reduced), 
 and the usual derived categories are not well-behaved. 
Instead, we use the following critical locus description of 
$\Hilb(\mathbb{C}^3, d)$, which is often used in DT theory. 
Let
$\mathrm{NHilb}(d)$ be the non-singular quasi-projective 
variety, called the \textit{non-commutative 
Hilbert scheme}, defined by 
\begin{align}\label{intro:NHilb}
	\mathrm{NHilb}(d):=\left(V\oplus \text{Hom}(V, V)^{\oplus 3}\right)^{\text{ss}}/GL(V),
	\end{align}
where $V$ is a $d$-dimensional vector space and 
$(-)^{\text{ss}}$ is the GIT semistable locus 
consisting of $(v, X, Y, Z)$ such that $V$ is generated by $v$ under the 
action of $(X, Y, Z)$. 
 Then $\mathrm{Hilb}(\mathbb{C}^3, d)$ is the critical 
 locus of the regular function 
\begin{align}\label{intro:trace}
	\Tr W \colon  \mathrm{NHilb}(d)\to\mathbb{C}, \ 
	(v, X, Y, Z) \mapsto \Tr Z[X, Y]. 
	\end{align}

We define the dg-category
\begin{align}\label{intro:DTcat}
	\mathcal{DT}(d) :=\mathrm{MF}(\mathrm{NHilb}(d), \Tr W),
	\end{align}
where the right hand side is the category of matrix factorizations 
on $\mathrm{NHilb}(d)$ with respect to the function $\Tr W$. 
It has objects
 $(\alpha \colon \mathcal{F}\rightleftarrows \mathcal{G} \colon \beta)$,
where $\mathcal{F}$ and $\mathcal{G}$ are coherent 
sheaves on $\mathrm{NHilb}(d)$ and the compositions $\alpha\circ\beta$ and $\beta\circ\alpha$ are multiplications by $\Tr W$.
One can also consider graded categories of matrix factorizations,
and equivariant (graded or not) matrix factorizations 
with respect to the action of 
the subtorus $T \subset (\mathbb{C}^{\ast})^3$ preserving the 
function $\Tr W$. 

The above category (\ref{intro:DTcat}) and its variants
are called \textit{DT categories}, 
as they recover the DT invariant by taking the Euler characteristic
of their periodic cyclic homologies. 
The purpose of this paper is to investigate
the structure of DT categories. 
This paper contains the first results on the categorical study of Hilbert schemes of points in dimension greater than $2$.    
 
\subsection{Semiorthogonal decompositions of DT categories}
The Hilbert scheme of points naturally fits into the following diagram: 
\begin{align}\label{dia:Hilb}
    \xymatrix{
    \mathrm{Hilb}(X, d) \ar[rd] & & X^{\times d}/\mathfrak{S}_d \ar[ld] \\
    & \mathrm{Sym}^d(X). &
    }
\end{align}
When $X=\mathbb{C}^2$, as we recall in Subsection~\ref{subsec:catHilb},
both the Hilbert scheme 
$\mathrm{Hilb}(X, d)$ and the quotient 
stack $X^{\times d}/\mathfrak{S}_d$
are smooth and they 
are derived equivalent. 
For $X=\mathbb{C}^3$ and $d=2$, 
both sides of (\ref{dia:Hilb}) are still smooth.
However, they are not derived equivalent,
and $D^b(\mathrm{Hilb}(X, 2))$ has a nontrivial semiorthogonal 
decomposition (see~Example~\ref{exam:d2}). 
Our first main result gives an analogue of such 
semiorthogonal decompositions of DT categories for arbitrary $d$: 
\begin{thm}\emph{(Theorem~\ref{MacMahonthm})}\label{thm:intro1}
There is a semiorthogonal decomposition 
	\[\mathcal{DT}(d)=
	\left\langle \boxtimes_{i=1}^k \mathbb{S}(d_i)_{v_i+d_i\left(\sum_{i>j}d_j-\sum_{j>i}d_j\right)}\right\rangle,\] where the right hand side consists of all tuples $A=(d_i, v_i)_{i=1}^k$ with $\sum_{i=1}^k d_i=d$ such that  \begin{equation}\label{ineq:intro1}
		0 \leq \frac{v_1}{d_1}<\cdots<\frac{v_k}{d_k} <1.\end{equation}
Moreover the fully-faithful functor 
	\begin{align}\label{intro:thm:hall}
		\boxtimes_{i=1}^k \mathbb{S}(d_i)_{v_i+d_i\left(\sum_{i>j}d_j-\sum_{j>i}d_j\right)}
		\to \mathcal{DT}(d)
	\end{align}
	is induced by the categorical Hall product for quivers 
	with super-potentials.
	\end{thm}

For $(d, w) \in \mathbb{N} \times \mathbb{Z}$, 
the category $\mathbb{S}(d)_w$ 
in Theorem~\ref{thm:intro1} is defined 
to be a certain dg-subcategory
\begin{align}\label{intro:S}
	\mathbb{S}(d)_w:=\mathrm{MF}(\mathbb{M}(d)_w, \Tr W)
	\subset \mathrm{MF}(\mathcal{X}(d), \Tr W),
	\end{align}
where $\mathcal{X}(d)$ is
the moduli stack of $d$-dimensional representations of the triple loop quiver
and $\Tr W$ is given as in (\ref{intro:trace}). 
It consists of matrix factorizations with factors coherent sheaves on $\X(d)$
whose weights with respect to the maximal torus $T(d) \subset GL(d)$
are contained in a certain polytope
in the weight lattice, and weight $w$ with respect to
the diagonal cocharacter of $T(d)$. 
The categorical Hall product 
in (\ref{intro:thm:hall})
is introduced and studied by the first author~\cite{P0}. 
We remark that each semiorthogona summand is equivalent to 
$\boxtimes_{i=1}^k\mathbb{S}(d_i)_{v_i}$, but 
we use the weight 
$v_i+d_i(\sum_{i>j}d_j-\sum_{j>i}d_j)$
in order to specify the fully-faithful functor
as a categorical Hall product.

The categories $\mathbb{M}(d)_w$ from \eqref{intro:S}
first appeared in 
the work of \v{S}penko--Van den Bergh~\cite{SVdB} on constructions of (twisted)
non-commutative resolutions of quotient singularities
of quasi-symmetric representations of reductive groups. 
They were later used by Halpern-Leistner--Sam~\cite{hls}  
to prove the magic window theorem and settle several cases of the D/K equivalence 
conjecture. 
The first author also used the subcategories (\ref{intro:S})
to prove PBW theorem for K-theoretic Hall algebras
for quivers with super-potentials~\cite{P}. 
We refer to the subcategory (\ref{intro:S}) as a \textit{quasi-BPS category}
since it is closely related (but not completely analogous) 
to the 
BPS sheaf by Davison--Meinhardt in cohomological DT theory \cite{DM}.  
We expect that the subcategories (\ref{intro:S}) cannot be further decomposed into 
non-trivial semiorthogonal decompositions, so that they are
the smallest 
building blocks
of the DT category $\mathcal{DT}(d)$.

Note that Theorem~\ref{thm:intro1} in particular implies the existence of a fully-faithful 
functor 
\begin{align}\label{intro:FF}
\mathbb{S}(d)_0 \hookrightarrow \mathcal{DT}(d).
\end{align}
Our point of view is that $\mathbb{S}(d)_0$ may be
regarded as a categorical 
DT-analogue of $D^b\left((\mathbb{C}^3)^{\times d}/\mathfrak{S}_d\right)$, 
and the above fully-faithful functor (\ref{intro:FF}) as 
a derived McKay correspondence for the permutation action of $\mathfrak{S}_d$
on $(\mathbb{C}^3)^{\times d}$.  
We will discuss this point of view in Subsection~\ref{subsection8}. 

\subsection{Generators of quasi-BPS categories via Koszul duality}
From Theorem~\ref{thm:intro1}, it is important to investigate the 
structure of $\mathbb{S}(d)_w$. 
However, it is not clear whether 
$\mathbb{S}(d)_w$ is empty or not since it is not clear whether a
given vector bundle on $ \mathcal{X}(d)$ satisfying some condition 
on $T(d)$-weights fits into a matrix factorization of $\Tr W$. 

The next purpose of this paper is to construct explicit objects in $\mathbb{S}(d)_w$ via Koszul duality equivalences, 
and show that these objects generate $\mathbb{S}(d)_w$ in 
torus localized K-theory. Let $\mathcal{C}(d)$ be the derived 
stack of commuting matrices of size $d$, or, equivalently, 
the derived moduli stack of zero-dimensional 
coherent sheaves on $\mathbb{C}^2$ with length $d$. 
The Koszul duality equivalence \cite{I, Hirano}, also called dimensional reduction in the literature, says that there is an 
equivalence 
\begin{align*}
	D^b(\mathcal{C}(d)) \stackrel{\sim}{\to}
	\mathrm{MF}^{\mathrm{gr}}(\mathcal{X}(d), \Tr W).
	\end{align*}
Then there is a subcategory $\mathbb{T}(d)_w \subset D^b(\mathcal{C}(d))$ 
corresponding to $\mathbb{S}^{\mathrm{gr}}(d)_w$ under the above equivalence. 
For each $(d, v) \in \mathbb{N} \times \mathbb{Z}$, we construct 
an object $\mathcal{E}_{d, v} \in \mathbb{T}(d)_v$
explicitly, using the derived stack of filtrations of zero-dimensional 
coherent sheaves on $\mathbb{C}^2$. The complexes $\mathcal{E}_{d, v}$ have also been considered by Gorsky--Negu\c{t} \cite[Section 1.4]{GoNeg}.
It turns out that, 
up to a constant in $\mathbb{K}:=K(BT)=\mathbb{Z}[q_1^{\pm 1}, q_2^{\pm 1}]$, 
the images of the above objects under a natural map 
\begin{align*}
	\iota_\ast \colon G_T(\mathcal{C}(d)) \to \mathbb{K}[z_1^{\pm 1}, \ldots, z_d^{\pm 1}]^{\mathfrak{S}_d}
\end{align*}
are part of a basis of the elliptic Hall algebra considered by Negu\c{t} in \cite{N}. Using \cite[Equation~(2.11)]{N}, we show in Subsection \ref{subsection:generators} that:
\begin{thm}\emph{(Theorem~\ref{prop:gen})}\label{thm:intro2}
	Let $(d, v)\in \mathbb{N}\times\mathbb{Z}$ be coprime integers, let $n\in\mathbb{N}$,
	and let $\mathbb{F}$ be the fraction field of $\mathbb{K}$. 
	The $\mathbb{F}$-vector
	space $K_T\left(\mathbb{T}(nd)_{nv}\right)\otimes_{\mathbb{K}}\mathbb{F}$ has a basis $\mathcal{E}_{n_1d, n_1v}\ast \cdots \ast \mathcal{E}_{n_kd, n_kv}$, where $n_1, \ldots, n_k\geq 1$ and $\sum_{i=1}^kn_i=n$. 
	Here $\ast$ is the categorical Hall product on $D^b(\mathcal{C}(d))$. 
\end{thm}

As we will recall in Subsection~\ref{subsec:intro:DTreview}, the generating series of classical 
DT invariants for zero-dimensional subschemes in $\mathbb{C}^3$ 
is the MacMahon function, whose coefficients are 
the number of 3D partitions (also called plane partitions in the literature). 
By combining Theorem~\ref{thm:intro2} with Theorem~\ref{thm:intro1}, we obtain the following 
K-theoretic analogue of MacMahon's formula: 
\begin{cor}\emph{(Corollary~\ref{thm:dim})}\label{intro:cor1}
	Let $p_3(d)$ be the number of 3D partitions of $d$. 
	We have the identity 
	\begin{align*}
		\dim_{\mathbb{F}}K_{T}(\mathcal{DT}(d))\otimes_{\mathbb{K}} \mathbb{F}
		=p_3(d). 
		\end{align*}
	\end{cor}
	Indeed, we have the explicitly constructed basis of 
	$K_T(\mathcal{DT}(d)) \otimes_{\mathbb{K}} \mathbb{F}$
	whose cardinality is $p_3(d)$, 
	see~Corollary~\ref{thm:dim} for a more precise statement.

\subsection{Classical/motivic/cohomological/categorical DT theory}\label{subsec:intro:DTreview}
Here we review DT theory briefly, from its classical version to its categorical one. 

The classical DT invariant is a virtual count of sheaves 
on a Calabi-Yau 3-fold. 
When the sheaves considered are ideal sheaves of zero-dimensional 
subschemes, the DT invariant $\text{DT}_{X, d}$ for any $d\geq 0$ equals the virtual Euler characteristic (with respect to the Behrend function) of the Hilbert scheme $\text{Hilb}(X, d)$. 
The generating series of DT invariants satisfies 
\begin{equation}\label{thm:DTthreefold}
	\sum_{d \geq 0}\mathrm{DT}_{X, d}q^d=M(-q)^{\chi(X)},
\end{equation}
where $M(q)$ is the generating series of $3$D partitions, 
and $\chi(X)$ is the Euler characteristic of $X$, see \cite{BF}. 
The proof of \eqref{thm:DTthreefold} for general $X$ follows from the statement for $\mathbb{C}^3$, when \eqref{thm:DTthreefold} reduces to the equality 
\begin{equation}\label{equation:intro1}
	\text{DT}_d:=\text{DT}_{\mathbb{C}^3, d}=(-1)^dp_3(d).
\end{equation} 

There are several refinements of the classical DT theory which associate motives, vector spaces, or constructible sheaves to
a Calabi-Yau $3$-fold. 
One can also compute these refinements (for $\mathbb{C}^3$, and then for general $X$) for ideal 
sheaves of zero-dimensional subschemes, see \cite[Theorem 2.7]{BBS}, \cite[Lemma 4.1]{Dav}. 
For example, in cohomological DT theory, one associates a sheaf $\varphi_d$ on $\text{Hilb}(X, d)$ such that 
\begin{equation}\label{equation:intro2}
	\sum_{i\in\mathbb{Z}}(-1)^i \text{dim}_{\mathbb{Q}} H^i(\text{Hilb}(X, d), \varphi_d)=\text{DT}_{X, d}.
\end{equation}
For $X=\mathbb{C}^3$, the sheaf $\varphi_d$ is
 the vanishing cycle sheaf 
 $\varphi_{\Tr W}\mathbb{Q}_N[\dim N]$
 for $N=\mathrm{NHilb}(d)$
 associated to the function (\ref{intro:trace}). 

In \cite{T}, the second author started the study of DT categories of local surfaces, 
i.e. threefolds $X=\text{Tot}_S\left(\omega_S\right)$ for $S$ a smooth surface. 
The DT categories recover a $\mathbb{Z}/2\mathbb{Z}$-periodization of the cohomological DT invariants, and thus the classical DT invariants. The K-theory of these categories provides a version of K-theoretic DT invariants for local surfaces. 
The category (\ref{intro:DTcat}) is equivalent to the 
DT category constructed in~\cite{T, T4} for $S=\mathbb{C}^2$.

\subsection{Relation with wall-crossing formulae of DT invariants}
We regard
the semiorthogonal decompositions in Theorem~\ref{thm:intro1} as 
categorifications of wall-crossing formulae of DT invariants. 
The non-commutative Hilbert scheme (\ref{intro:NHilb}) is 
the moduli space of stable representations of 
the triple loop quiver with a framing
of dimension vector $(1, d)$. 
If we change a stability condition by crossing a wall, 
then the semistable locus is empty unless $d=0$, and 
the difference of the corresponding DT invariants is described 
in terms of Joyce-Song generalized DT invariants~\cite{JS} for representations of the unframed 
triple loop quiver. 
The corresponding 
wall-crossing formula is given by
\begin{align}\label{intro:wcf}
	\sum_{d \geq 0} \mathrm{DT}_{d}q^d=
	\exp\left(\sum_{n\geq 1} (-1)^{n-1} n N_n q^n   \right)
	=\prod_{d\geq 1}(1-(-q)^d)^{d \Omega_d},
	\end{align}
	see~\cite[Section~6.3]{JS}, \cite[Remark~5.14]{T5}.
The invariant $N_{n} \in \mathbb{Q}$ is a generalized DT invariant 
counting zero-dimensional sheaves on $\mathbb{C}^3$ with length $n$, 
and the invariant $\Omega_{d} \in \mathbb{Z}$ (called \textit{BPS invariant})
is defined by the multiple cover formula:
\begin{align*}
	N_n=\sum_{k\geq 1, k|n}\frac{1}{k^2}\Omega_{n/k}. 
	\end{align*}
In fact, one computes that $\Omega_d=-1$ for all $d\geq 1$, so one obtains (\ref{thm:DTthreefold}). 
Davison--Meinhardt~\cite{DM} constructed a perverse sheaf on $\mathrm{Sym}^d(\mathbb{C}^3)$, 
called the \textit{BPS sheaf}, 
whose Euler characteristic recovers the BPS invariant. 

We may view $\mathbb{S}(d)_w$ as a categorification of $\Omega_d$
and regard the semiorthogonal decomposition in Theorem~\ref{thm:intro1} as 
a categorification of the formula (\ref{intro:wcf}).
However, this is not precise, as the number of generators of 
$\mathbb{S}(d)_w$ is the number of partitions of $\gcd(d, w)$
(at least in torus localized K-theory), 
which is bigger than $\lvert \Omega_d \rvert=1$
if $\gcd(d, w)>1$. When $\gcd(d, w)=1$, the category $\mathbb{S}(d)_w$ should be regarded 
as a categorification of $\Omega_d$ as we explain in the follow-up paper \cite{PT1}. 
When $\gcd(d, w)>1$, we expect that the primitive part
of $K(\mathbb{S}(d)_w)$ with respect to its coproduct structure 
has rank one, and 
gives a K-theoretic version of $\Omega_d$. 
We pursue this problem in \cite{PT1}.

The categorification problem of wall-crossing formulae of DT invariants 
has been pursued by the second author in several situations~\cite{T, T3, TodDK}. 
We remark that the combinatorics of torus weights required in this paper is much harder
than that involved in the mentioned previous works, and the argument from the proof of the PBW theorem for Hall algebras 
by the first author~\cite{P} is essential to overcome this difficulty.

\subsection{Categorical and K-theoretic study of Hilbert schemes of points}\label{subsec:catHilb}
We now review the categorical study of Hilbert schemes of points 
$\mathrm{Hilb}(X, d)$ when $\dim X \leq 2$, and of non-commutative Hilbert 
schemes. 

When $X$ is a smooth projective curve of genus $g$, $\mathrm{Hilb}(X, d)$ is the symmetric product
$\mathrm{Sym}^d(X)$ of $X$. 
Although $\mathrm{Sym}^d(X)$ is a classical 
variety, the structure of its derived category of coherent 
sheaves has been investigated 
rather recently. In~\cite{TodDK}, the second author 
constructed semiorthogonal decompositions
of the derived category of $\mathrm{Sym}^d(X)$
for $d>g-1$
based on the wall-crossing formula of Pandharipande-Thomas invariants, see also~\cite{Jiangproj} for a different proof.  
For $d\leq g-1$, the semiorthogonal indecomposability
of 
$D^b(\mathrm{Sym}^d(X))$ is studied in~\cite{BTK}. 
 
When $X$ is a K3 surface, 
$\Hilb(X, d)$ is a holomorphic symplectic manifold. 
The main result of Halpern-Leistner~\cite{HalpK32} implies the
D/K equivalence conjecture by Bondal--Orlov~\cite{B-O2}, Kawamata~\cite{Ka1} for $\Hilb(X, d)$, 
i.e. any holomorphic symplectic manifold $M$ 
which admits a birational map $M \dashrightarrow \Hilb(X, d)$
has $D^b(M)$ equivalent to $D^b(\Hilb(X, d))$. 
For an arbitrary smooth projective surface $X$
and its blow-up $\widehat{X} \to X$, 
Koseki~\cite{Kosekiup} proved a categorical blow-up formula for Hilbert 
schemes of points, using the second author's proof~\cite{Toquot} of Jiang's 
conjecture~\cite{JiangQuot} on derived categories of some Quot schemes. 

For $X=\mathbb{C}^2$, 
the celebrated derived McKay correspondence by 
Bridgeland--King--Reid \cite{BKR} and Haiman \cite{Ha} says that there is a derived equivalence 
\begin{align}\label{intro:McKay}
	D^b(\mathrm{Hilb}(\mathbb{C}^2, d)) \cong D^b((\mathbb{C}^2)^{\times d}/\mathfrak{S}_d).
	\end{align}
	In \cite{SV}, Schiffmann--Vasserot constructed an action of the elliptic Hall algebra on the equivariant K-theory of $\mathrm{Hilb}(\mathbb{C}^2, d)$, which extends the action of the Heisenberg algebra, due to Nakajima and Grojnowski, on the cohomology of $\mathrm{Hilb}(\mathbb{C}^2, d)$.
	The categorical Heisenberg actions on derived categories of 
	$\mathrm{Hilb}(\mathbb{C}^2, d)$ were studied in~\cite{CauLi, Krug}
	using the equivalence (\ref{intro:McKay}), also see~\cite{Zhao2}.

Recently, Lunts--{\v S}penko--{V}an den Bergh~\cite{LSV} studied derived categories of coherent sheaves on non-commutative Hilbert schemes (defined from framed quivers with 
arbitrary number of loops), 
and constructed semiorthogonal decompositions of these categories, 
see~\cite[Proposition~1.9]{LSV}.
Their semiorthogonal decomposition for the framed triple loop quiver 
corresponds to the semiorthogonal decomposition in (\ref{sod:E}). 
In the proof of Theorem~\ref{thm:intro1}, we 
give its refinement, see (\ref{SODnhilb}). 


\subsection{A K-theoretic Bridgeland--King--Reid--Haiman isomorphism in dimension three}\label{subsection8}
We propose investigating analogous statements to (\ref{intro:McKay})
in dimension three. The Hilbert scheme of points on $\mathbb{C}^3$ is very singular, and we replace it with the summands $\mathbb{S}(d)_w$ of the category $\mathcal{DT}(d)$ in our pursuit for a McKay correspondence 
with respect to the permutation action of $\mathfrak{S}_d$ on $(\mathbb{C}^{3})^{\times d}$.

	Let $(d, v)\in \mathbb{N}\times\mathbb{Z}$ be
	coprime numbers and take an integer $n\geq 1$. 
	The result of Theorem~\ref{thm:intro2}
	suggests a similarity between (the K-theory of) the categories $\mathbb{S}(nd)_{nv}$ and $\mathrm{MF}\left(\left(\mathbb{C}^3\right)^{\times n}/\mathfrak{S}_n, 0\right)$, and between their graded and/or $T$-equivariant versions. 
	Considering the graded $T$-equivariant situation, 
	we expect an isomorphism of $\mathbb{K}$-modules
	\begin{align}\label{isom:conj}
	K_T\left((\mathbb{C}^3)^{\times n}/\mathfrak{S}_n\right)\xrightarrow{\sim} K_T\left(\mathbb{T}(nd)_{nv}\right), 
	\end{align}
	or, equivalently, that $K_T\left(\mathbb{T}(nd)_{nv}\right)$ is a free $\mathbb{K}$-module 
	of rank $p_2(n)$. 
We regard Theorem \ref{thm:intro2} as evidence towards (\ref{isom:conj}), 
i.e. it is true after 
taking $(-)\otimes_{\mathbb{K}}\mathbb{F}$. 
We provide further evidence towards the expectation (\ref{isom:conj}) in \cite{PT1}. 
In loc.~cit.~, we define a $\mathbb{K}$-bialgebra structure on \[\mathcal{A}:=\bigoplus_{n\geq 0}K_T(\mathbb{T}(nd)_{nv})\] such that $\mathcal{A}\otimes_{\mathbb{K}}\mathbb{F}$ is the MacDonald symmetric algebra over $\mathbb{F}$. We further show in loc.~cit.~that
$K_T\left(\mathbb{T}(d)_v\right)$ modulo $\mathbb{K}\text{-torsion}$ is a free $\mathbb{K}$-module with generator $\mathcal{E}_{d, v}$, in particular the isomorphism~\eqref{isom:conj} holds for $n=1$ up to 
$\mathbb{K}$-torsion. 

It is natural to formulate analogues of (\ref{isom:conj}) for non-equivariant K-theory or even for categories, but we do not have much evidence towards them.

\subsection{Further directions}
In \cite{PT1}, we continue the study of the categories $\mathbb{S}(d)_w$. Besides the results referenced in Subsection~\ref{subsection8}, we study the support of objects in $\mathbb{S}(d)_w$. Davison~\cite[Lemma~4.1]{Dav} proved that the BPS sheaf on $\mathbb{C}^3$ is supported over the small diagonal $\mathbb{C}^3\hookrightarrow \text{Sym}^d\left(\mathbb{C}^3\right)$. In loc.~cit. we prove categorical and K-theoretic analogues of Davison's support lemma.

The computation \eqref{thm:DTthreefold} is not only a starting point in computations of DT invariants, but it also appears in the statement of the Maulik--Okounkov--Nekrasov--Pandharipande conjecture \cite{MNOP} and in the 
DT/ PT correspondence \cite{Br}, \cite{T5}.
In \cite{PT2}, we prove a version of Theorem \ref{thm:intro1} for threefolds $X=\text{Tot}_S(\omega_S)$, where $S$ is a smooth surface. We then use it to prove categorical versions of the DT/PT correspondence for local surfaces.

\subsection{Structure of the paper}
In Section \ref{Section:preliminaries}, we list some of the main notations used in the rest of the paper, especially related to weight spaces of $T(d)\subset GL(d)$, and we review and provide references for some of the constructions used, such as matrix factorizations. 
In Section \ref{section:MacMahon}, we define the quasi-BPS categories $\mathbb{S}(d)_w$ in Subsection \ref{ss2} and
prove Theorem \ref{thm:intro1}. 
In Section \ref{Section:KDT}, we introduce the objects $\mathcal{E}_{d, v}$ of $\mathbb{T}(d)_v$ in 
Definition~\ref{definition:Edw} and prove Theorem \ref{thm:intro2}.

\subsection{Acknowledgements}
We thank Tasuki Kinjo, Davesh Maulik, Andrei Negu\c{t}, 
Raphaël Rouquier, Olivier Schiffmann, and Eric Vasserot for discussions related to this work. We thank {\v S}pela {\v S}penko for telling us about her joint work \cite{LSV} and for comments on a previous version of the article. We thank the referees for numerous useful suggestions.
	Y.~T.~is supported by World Premier International Research Center
	Initiative (WPI initiative), MEXT, Japan, and Grant-in Aid for Scientific
	Research grant (No.~19H01779) from MEXT, Japan.

\section{Preliminaries}\label{Section:preliminaries}

\subsection{Notations}
The spaces considered in this paper are defined 
over the complex field $\mathbb{C}$ and they are quotient stacks $\X=A/G$, where $A$ is a dg scheme, the derived zero locus of a section $s$ of a finite rank bundle vector bundle $\mathcal{E}$ on a finite type separated scheme $X$ over $\mathbb{C}$, and $G$ is a reductive group. For such a dg scheme $A$, let $\dim A:=\dim X-\text{rank}(\mathcal{E})$,
let $A^{\mathrm{cl}}:=Z(s)\subset X$ be the (classical) zero locus, and let $\X^{\mathrm{cl}}:=A^{\mathrm{cl}}/G$. 

For a (derived) stack $\mathcal{X}$ with an action of a torus $T$, we denote by $D^b_T(\mathcal{X})$
the bounded derived category of $T$-equivariant coherent sheaves on $\mathcal{X}$ and by $G_T(\X)$ its Grothendieck group. We have a natural isomorphism $G_T(\X)\cong G_T(\X^{\text{cl}})$ \cite[Equation (2.4)]{VaVa}. When $T$ is trivial, we drop it from the notation.



\subsection{Table}
We list the most frequent notations used in the paper:

\begin{figure}
	\centering
\scalebox{0.7}{
	\begin{tabular}{|l|l|l|}
		\hline
	Notation & Description & Location defined \\\hline
$Q$ & three loop quiver & Subsection \ref{sss1}
\\ \hline
$Q^f$ & framed three loop quiver &
Subsection \ref{sss1}
\\ \hline
$W$ and $\mathrm{Tr}\,W$ & potential and induced regular function & Equation \eqref{def:Wd}
\\ \hline
$\X(d)$ & stack of representations of $Q$ &
Subsection \ref{sss1}
\\ \hline
$\X^f(d)$ & stack of representations of $Q^f$ &
Subsection \ref{sss1}
\\ \hline
$\mathrm{NHilb}(d)$ & noncommutative Hilbert scheme & 
Subsection \ref{subsec:catC3}
\\ \hline
$\mathcal{Y}(d)$ & stack of representations of the two loop quiver &
Subsection \ref{subsection:Hallcomm}
\\ \hline
$\mathcal{C}(d)$ & commuting stack & Subsection \ref{subsection:Hallcomm}
\\ \hline
$M(d)$, $M$, $M(d)_\mathbb{R}$, $M_\mathbb{R}$ & weight spaces & Subsection \ref{sss2}
\\ \hline
$\mathcal{W}$ and $\mathcal{W}^f$ & sets of weights associated to $\X(d)$ and $\X^f(d)$ &
Subsection \ref{def:setW}
\\ \hline
$\rho$ & half the sum of positive roots &
Subsection \ref{def:setW}
\\ \hline
$\tau_d$ & Weyl-invariant weight with sum of coefficients $1$ & Subsection \ref{def:setW}
\\ \hline
$1_d$ & diagonal cocharacter &
Subsection \ref{def:setW}
\\ \hline
$\mathcal{T}$ & tree used to define paths of partitions &
Subsection \ref{tree}
\\ \hline
$\mu\succeq \lambda$  & comparison of cocharacters &
Subsection \ref{comparisoncocharacters}
\\ \hline
$\ee\geq\dd$ & comparison of partitions &
Subsection \ref{compa}
\\ \hline
$\textbf{W}(d)$ and $\textbf{W}(d)_w$ & polytopes used to define quasi-BPS categories
& Equations \ref{W} and \ref{W0}
\\ \hline
$\textbf{V}(d)$ & polytope used to study $\mathrm{NHilb}(d)$ & Equation \ref{V}
\\ \hline
$F_r(\lambda)$ and $F_r(\lambda)^{\mathrm{int}}$ & Face (and its interior) of the polytope $\textbf{W}(d)$ & Subsection \ref{ss1}
\\ \hline
$\mathrm{MF}$, $\mathrm{MF}^{\mathrm{gr}}$ & Categories of matrix factorizations &
Subsection \ref{MF}
\\ \hline
$D_\mathrm{sg}$, $D_\mathrm{sg}^{\mathrm{gr}}$ & Categories of singularities &
Subsection \ref{MF}
\\ \hline
$T$ & two-dimensional torus &
Equation \eqref{torus:T}
\\ \hline
$\mathbb{K}$ & Grothendieck ring of $BT$ & Subsection \ref{subsection:shuffle}
\\ \hline
$\mathbb{F}$ & Fraction field of $\mathbb{K}$ &
Subsection \ref{subsection:generators}
\\ \hline
$\mathcal{DT}^\bullet_*(d)$ & DT category (graded or not, equivariant or not) &
Definition \ref{def:catDT2}
\\ \hline
$\mathbb{M}(d)$ and $\mathbb{M}(d)_w$ & Magic categories for $\X(d)$
&Subsection \ref{ss:Ddelta}
\\ \hline
$\mathbb{D}(d; \delta)$ and $\mathbb{F}(d;\delta)$ & Magic categories used to study $\mathrm{NHilb}(d)$
&Subsection \ref{ss:Ddelta}
\\ \hline
$\mathbb{S}^\bullet_*(d)_w$ & quasi-BPS categories
& Subsection \ref{gradingMF}
\\ \hline
$\mathbb{T}(d)_w$ and $\widetilde{\mathbb{T}}(d)_w$ & quasi-BPS categories for the commuting stack
& Subsection \ref{subsection:Tdv}
\\ \hline
$(v_i)_{i=1}^k$ & integer weights corresponding to a partition $(d_i)_{i=1}^k$ &
Equation \eqref{w:prime}
\\ \hline
$S^d_w$ & set of partitions
&Subsection \ref{dectree2}
\\ \hline
$T^d_w$ & set of partitions
&Subsection \ref{subsec333}
\\ \hline
$\mathcal{E}_{d,v}$ & complex in $\mathbb{T}(d)_v$ 
&Definition \ref{definition:Edw}
\\ \hline
$\mathcal{F}_{d,w}$ & complex in $\mathbb{S}^\bullet(d)_w$
& Definition \ref{def:objF}
\\ \hline
$\Phi$ & the Koszul equivalence
& Equation \eqref{equiv:Phi}
\\ \hline
	\end{tabular}
}
	\vspace{.5cm}
	\caption{Notation introduced in the paper}
	\label{table:notation}
\end{figure}

\subsection{Weights and partitions} 

\subsubsection{}\label{sss1}
Let $Q=(I, E)$ be the quiver with vertex set $I=\{1\}$ and edge set $E=\{x, y, z\}$. Let $Q^f$ be the quiver with vertex set $I^f=\{0, 1\}$ and edge set $E^f=e\sqcup E$, where $e$ is an edge from $0$ to $1$.

For $d\in \mathbb{N}$, let $V$ be a $\mathbb{C}$-vector space of dimension $d$. 
We often write $GL(d)$ as $GL(V)$. 
Its Lie algebra is 
denoted by $\mathfrak{gl}(d)=\mathfrak{gl}(V):=\text{End}(V)$. 
When the dimension is clear from the context, we drop $d$ from its notation
and write it as $\mathfrak{g}$. 
Denote by 
\begin{align}\label{def:Xd}
\X(d)&:=R(d)/GL(d)=\mathfrak{gl}(V)^{\oplus 3}/GL(V), \\
\notag \X^f(d)&:=R^f(d)/GL(d)=\left(V\oplus \mathfrak{gl}(V)^{\oplus 3}\right)/GL(V)
\end{align}
 the moduli stacks of representations of $Q$ of dimension $d$, and of
 representations of $Q^f$ of dimension vector $(1, d)$, 
 respectively. 
 
 \subsubsection{}\label{sss2}
 Fix 
 $T(d)\subset GL(d)$ the maximal torus consisting of diagonal matrices. 
Denote by $M(d)$ the weight space of $T(d)$ and let $M(d)_{\mathbb{R}}:=M(d)\otimes_{\mathbb{Z}}\mathbb{R}$. Let $\beta_1,\ldots, \beta_d$ be the weights of the standard representation of $GL(d)$. 
A weight $\chi=\sum_{i=1}^d c_i\beta_i$ is dominant 
(resp. strictly dominant) 
if 
\begin{align*}
c_1\leq\cdots\leq c_d, \ 
(\mbox{resp.~}
c_1<\cdots<c_d).
\end{align*}
We denote by $M^+\subset M$ and $M^+_{\mathbb{R}}\subset M_{\mathbb{R}}$ the dominant chambers. When we want to emphasize the dimension vector, we write $M(d)$ etc. Denote by $N$ the coweight lattice of $T(d)$ and by $N_{\mathbb{R}}:=N\otimes_{\mathbb{Z}}\mathbb{R}$. Let $\langle\,,\,\rangle$ be the natural pairing between $N_{\mathbb{R}}$ and $M_{\mathbb{R}}$. 

Let $\mathfrak{S}_d$ be the Weyl group of $GL(d)$. For $\chi\in M(d)^+$, let $\Gamma_{GL(d)}(\chi)$ be the irreducible 
representation of $GL(d)$ of highest weight $\chi$. We drop $GL(d)$ from the notation if the dimension vector $d$ is clear from the context. 

\subsubsection{}\label{def:setW}
Denote by $\mathcal{W}$ the multiset of $T(d)$-weights of $R(d)$ and by $\mathcal{W}^f$ the multiset of $T(d)$-weights of $R^f(d)$.
Namely we have 
\begin{align*}
\mathcal{W}=\{(\beta_i-\beta_j)^{\times 3} \mid 1\leq i, j \leq d\}, \
\mathcal{W}^f =\mathcal{W} \cup \{\beta_i \mid 1\leq i\leq d\}. 
\end{align*}
For $\lambda$ a cocharacter of $T(d)$
and a $T(d)$-representation $N$, we denote by $N^{\lambda>0}$ the sum of 
weights $\beta$ in $N$ such that $\langle \lambda, \beta\rangle>0$.

We denote by 
$\rho$ half the sum of positive roots of $GL(d)$. 
In our convention of the dominant chamber, 
it is given by 
\begin{align*}
    \rho=\frac{1}{2}\mathfrak{g}^{\lambda<0}=\frac{1}{2}\sum_{j<i}(\beta_i-\beta_j),
\end{align*}
where $\lambda$ is the antidominant cocharacter $\lambda(t)=(t^d, t^{d-1}, \ldots, t)$. 
We denote by $1_d:=z\cdot\text{Id}$ the diagonal cocharacter of $T(d)$. 
Define the real weight
\begin{align}\label{def:taud}
\tau_d:=\frac{1}{d}\sum_{j=1}^d\beta_j
\in M_{\mathbb{R}}. 
\end{align}


\subsubsection{}\label{ss13} Let $G=GL(d)$, let $X$ be a $G$-representation, and let
$\X=X/G$ be the corresponding quotient stack. Let $\mathcal{V}$ be the multiset of $T(d)$-weights of $X$.
For a cocharacter $\lambda$ of $T(d)$,
let $X^\lambda\subset X$ be the subspace generated by weights $\beta\in \mathcal{V}$ such that $\langle \lambda, \beta\rangle=0$, let $X^{\lambda\geq 0}\subset X$ be the subspace generated by weights $\beta\in \mathcal{V}$ such that $\langle \lambda, \beta\rangle\geq 0$, and let $G^\lambda$ and $G^{\lambda\geq 0}$ be the Levi and parabolic groups associated to $\lambda$.
Consider the fixed and attracting stacks
\begin{align}\notag
    \X^\lambda:=X^\lambda/ G^\lambda,\
    \X^{\lambda\geq 0}:=X^{\lambda\geq 0}/G^{\lambda\geq 0}
\end{align}
with maps
\begin{align}\label{e}
 \X^\lambda\xleftarrow{q_\lambda}\X^{\lambda\geq 0}\xrightarrow{p_\lambda}\X.
 \end{align}
 
We abuse notation and denote by $\X^{\lambda\geq 0}$ the class 
\begin{align*}
\det(X^{\lambda \geq 0}) \otimes \det(\mathfrak{g}^{\lambda>0})^{\vee}
\in \mathrm{Pic}(BT(d))=M
\end{align*}	
and by $\langle \lambda, \X^{\lambda\geq 0}\rangle$ the corresponding integer. We use the similar notation for $\X^{\lambda>0}$, $\X^{\lambda\leq 0}$, etc.

\subsubsection{}\label{paco} 
Let $d\in \mathbb{N}$ and recall the definition of $\X(d)$ from (\ref{def:Xd}).
For a cocharacter $\lambda \colon \C^*\to T(d)$, consider the maps of fixed and attracting loci
(\ref{e}). 
We say that two cocharacters $\lambda$ and $\lambda'$ are equivalent and write $\lambda\sim\lambda'$ if $\lambda$ and $\lambda'$ have the same fixed and attracting stacks as above. 
 We call $\dd:=(d_i)_{i=1}^k$ a partition of $d$ if $d_i\in\mathbb{N}$ are all non-zero and $\sum_{i=1}^k d_i=d$. We similarly define partitions of $(d,w)\in\mathbb{N}\times\mathbb{Z}$.
For a cocharacter $\lambda$ of $T(d)$, there is an associated partition $(d_i)_{i=1}^k$ such that
$\X(d)^\lambda\cong\times_{i=1}^k\X(d_i)$. 
Equivalence classes of antidominant cocharacters are in bijection with ordered partitions $(d_i)_{i=1}^k$ of $d$.
For an ordered partition $\dd=(d_i)_{i=1}^k$ of $d$, fix a corresponding antidominant cocharacter $\lambda=\lambda_{\dd}$ of $T(d)$ which induces the maps
\[\X(d)^\lambda\cong\times_{i=1}^k\X(d_i)
\xleftarrow{q_\lambda}\X(d)^{\lambda\geq 0}
\xrightarrow{p_\lambda}\X.\] 
We also use the notations $p_\lambda=p_{\dd}$, $q_\lambda=q_{\dd}$.
The stack $\mathcal{X}(d)^{\lambda \ge 0}$ is isomorphic to the 
moduli stack of filtrations of $Q$-representations 
\begin{align*}
    0=R_0 \subset R_1 \subset \cdots \subset R_k
\end{align*}
such that $R_i/R_{i-1}$ has dimension $d_i$. 
The morphism $q_{\lambda}$ sends the above filtration to its 
associated graded, and $p_{\lambda}$ sends it to $R_k$. 
The categorical Hall algebra is given 
by the functor 
$p_{\lambda*}q_\lambda^*=p_{\dd*}q_{\dd}^*$ and denoted by
\begin{align}\label{prel:hall}
    \ast 
     \colon D^b(\mathcal{X}(d_1)) \boxtimes
     \cdots \boxtimes D^b(\mathcal{X}(d_k))
     \to D^b(\mathcal{X}(d)). 
\end{align}
We may drop the subscript $\lambda$ or $\dd$ in the functors $p_*$ and $q^*$ when the cocharacter $\lambda$ or the partition $\dd$ is clear.

\subsubsection{}\label{id} Let $(d_i)_{i=1}^k$ be a partition of $d$. There is an identification \[\bigoplus_{i=1}^k M(d_i)\cong M(d),\] where 
$\beta_1, \ldots, \beta_{d_1}$ correspond
to the 
the weights
of standard representation of $GL(d_1)$ in $M(d_1)$, etc.

\subsubsection{}\label{comparisoncocharacters}\label{subsubsub:order}
For cocharacters $\lambda,\mu \colon \C^*\to SL(d)\cap T(d)$, 
we write $\mu\succeq \lambda$ 
if for every weight $\beta \in \mathcal{W}$ with $\langle \lambda, \beta \rangle> 0$, we have that $\langle \mu, \beta\rangle > 0$. If both $\mu\succeq \lambda$ and $\lambda\succeq \mu$, write $\mu\sim\lambda$.

\subsubsection{}\label{compa}
Let $\underline{e}=(e_i)_{i=1}^l$ and $\dd=(d_i)_{i=1}^k$ be two partitions of $d\in \mathbb{N}$. We write $\ee\geq\dd$
if there exist integers \[a_0=0< a_1<\cdots<a_{k-1}\leq a_k=l\] such that for any $0\leq j\leq k-1$, we have
\[\sum_{i=a_{j}+1}^{a_{j+1}} e_i=d_{j+1}.\]
There is a similarly defined order on pairs $(d,w)\in\mathbb{N}\times\mathbb{Z}$.


\subsubsection{}\label{tree}
We define 
$\mathcal{T}$ to be the unique (oriented) tree such that: 
\begin{enumerate}
\item each vertex is indexed by a partition $(d_1, \ldots, d_k)$ of 
	some $d \in \mathbb{N}$, 
 \item for each $d \in \mathbb{N}$, there is a unique vertex 
	indexed by the partition $(d)$ of size one, 
  \item if $\bullet$ is a vertex indexed by $(d_1, \ldots, d_k)$ 
	and $d_m=(e_1, \ldots, e_s)$ is a partition of $d_m$ for some $1\leq m \leq k$, then there is a unique vertex $\bullet'$ indexed by 
	$(d_1, \ldots, d_{m-1}, e_1, \ldots, e_s, d_{m+1}, \ldots, d_k)$ and
	with an edge from $\bullet$ to $\bullet'$, and 
 \item all edges in $\mathcal{T}$ are as in (3).
	\end{enumerate}
Note that  
 each partition $(d_1, \ldots, d_k)$ of 
	some $d \in \mathbb{N}$ 
 gives an index of some (not necessary unique) vertex.  
A subtree $T \subset \mathcal{T}$ is called a \textit{path of partitions} 
if it is connected, contains a vertex indexed by $(d)$ for some 
$d \in \mathbb{N}$ and 
a unique end vertex $\bullet$. 
The partition $(d_1, \ldots, d_k)$ at the end vertex $\bullet$
is called the associated partition of $T$. 
We define the Levi group associated to $T$ to be 
\begin{align*}
	L(T):=\times_{i=1}^k GL(d_i). 
	\end{align*}
 Note that each edge $\ell\in \mathcal{T}$ corresponds to a partition of some natural number: if $\ell$ connects $(d_1, \ldots, d_k)$ and $(d_1, \ldots, d_{m-1}, e_1, \ldots, e_s, d_{m+1}, \ldots, d_k)$ as in (3), then $\ell$ corresponds to the partition $(e_1,\ldots, e_s)$ of $d_m$.


The tree $\mathcal{T}$ decomposes into the disjoint
union of $\mathcal{T}_d$ for $d \in \mathbb{N}$, 
where $\mathcal{T}_d$
contains the vertex indexed by $(d)$. 
The picture of $\mathcal{T}_3$ is depicted below. 

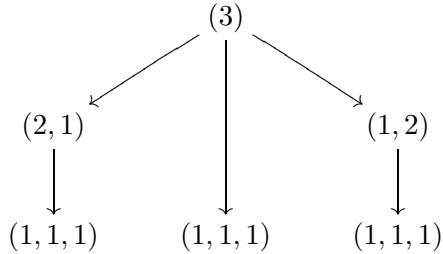
\begin{align*}
\xymatrix{
& (3) \ar[ld] \ar[dd] \ar[rd]  & \\
(2, 1) \ar[d] & & (1, 2) \ar[d] \\
(1, 1, 1) & (1, 1, 1) & (1, 1, 1)
}
\end{align*}


\subsection{Polytopes}\label{ss1}
We consider the following polytopes in $M(d)_{\mathbb{R}}$. 
The polytope $\textbf{W}(d)$ is defined as
\begin{equation}\label{W}
    \textbf{W}(d):=\frac{3}{2}\text{sum}[0, \beta_i-\beta_j]+\mathbb{R}\tau_d\subset M(d)_{\mathbb{R}},
    \end{equation}
where the Minkowski sum is after all $1\leq i, j\leq d$ and where $\tau_d$ is given by (\ref{def:taud}).  Consider the hyperplane
\begin{equation}\label{W0}
    \textbf{W}(d)_w:=\frac{3}{2}\text{sum}[0, \beta_i-\beta_j]+w\tau_d\subset \textbf{W}(d).
    \end{equation}
    For $r>0$ and $\lambda$ an antidominant cocharacter of $SL(d) \cap T(d)$, let $F_r(\lambda)$ be the face of the polytope $2r\textbf{W}(d)$ corresponding to the cocharacter $\lambda$, so the set of weights $\chi$ in $M(d)_\mathbb{R}$ such that 
    \begin{align*}
        \chi&\in 2r \textbf{W}(d), \ 
        \langle \lambda, \chi\rangle=-r\langle \lambda, R(d)^{\lambda>0}\rangle.
    \end{align*}
    Note that the second condition implies that 
    $\chi \in 2r \partial \textbf{W}(d)$. 
    For $\lambda$ a cocharacter, let $\Lambda$ be the set of cocharacters $\mu$ such that $R(d)^\mu\subsetneq R(d)^\lambda$.
    Denote by \[F_r(\lambda)^{\text{int}}:=F_r(\lambda)\setminus\bigcup_{\mu\in\Lambda}F_r(\mu).\]
    When $r=\frac{1}{2}$, we use the notations $F(\lambda)$ and $F(\lambda)^{\text{int}}$.
    For $\chi\in M(d)_{\mathbb{R}}$, its $r$-invariant $r(\chi)$ is the smallest non-negative real number $r$ such that
    \[\chi\in 2r\textbf{W}(d).
    \] 
    Recall that $\mathcal{W}$ is the set of weights of $R(d)$ counted with multiplicity. 
If $\chi\in M_\mathbb{R}$ has $r(\chi)=r$, then there are coefficients $-r\leq c_{\beta}\leq 0$ for $\beta\in\mathcal{W}$ and $c\in\mathbb{R}$ such that
\begin{equation}\label{chibeta}
    \chi=\sum_{\beta\in\mathcal{W}} c_{\beta}\beta+c\tau_d.
    \end{equation}
The $p$-invariant for $\chi$ with $r(\chi)=r$ is defined as the smallest possible number of coefficients $c_\beta$ for $\beta\in\mathcal{W}$ equal to $-r$ in a formula \eqref{chibeta}.
    
The polytope $\textbf{V}(d)$ is defined as
\begin{equation}\label{V}
    \textbf{V}(d):=\frac{3}{2}\text{sum}[0, \beta_i-\beta_j]+\text{sum}[0, \beta_k]\subset M(d)_{\mathbb{R}},
\end{equation}
where the sum is after all $1\leq i, j, k\leq d$. 
\\

\subsection{A corollary of the Borel-Weil-Bott theorem}

We continue with the notations from Subsection \ref{ss13}. We assume that $X$ is a symmetric $G$-representation, meaning that for any weight $\beta$ of the maximal torus 
$T \subset G$, the weights $\beta$ and $-\beta$ appear with the same multiplicity in $X$. 
For a weight $\chi \in M$, 
let $\chi^+$ be the dominant Weyl-shifted conjugate of $\chi$ if it exists, and let $\chi^+=0$ otherwise.
Let $\mathcal{V}$ be the multiset of 
weights in $X$. 
For a subset $J\subset \mathcal{V}$, let
\[\sigma_J:=\sum_{\beta\in J}\beta.\]
For a weight $\chi \in M$, let 
$w$ be the element of the Weyl group such that $w*(\chi-\sigma_J)$ is dominant or zero. It has length $\ell(w)=:\ell(J)$. 
The following proposition is based on computations in \cite[Section 3.2]{hls} and \cite[Subsection 11.2]{SVdB}.

\begin{prop}\label{bbw}
Let $G$ be a reductive group, let $X$ be a symmetric $G$-representation, and
let $\lambda$ be a cocharacter of the maximal torus $T\subset G$. 
Recall the fixed and attracting stacks and the corresponding maps
\[X^\lambda/G^\lambda\xleftarrow{q_\lambda}X^{\lambda\geq 0}/G^{\lambda\geq 0}\xrightarrow{p_\lambda}X/G.\]
Let $\chi\in M$ be a dominant weight of $G^\lambda$. 
Then there is a quasi-isomorphism
\[\left(\bigoplus_{J}\mathcal{O}_{X}\otimes \Gamma_{G}\left((\chi-\sigma_J)^+\right)\left[|J|-\ell(J)\right], d\right)\xrightarrow{\sim}p_{\lambda*}q_{\lambda}^*\left(\mathcal{O}_{X^\lambda}\otimes\Gamma_{G^\lambda}(\chi)\right),\] where the complex on the left hand side has terms (shifted) vector bundles for all multisets $J\subset \{\beta\in\mathcal{V} \mid \langle \lambda, \beta\rangle<0\}$. 
\end{prop}

\begin{proof}
Factor the map $p_\lambda=\pi_\lambda i_\lambda$, where 
\[i_\lambda\colon X^{\lambda\geq 0}/G^{\lambda\geq 0}\hookrightarrow X/G^{\lambda\geq 0},\, \pi_\lambda\colon X/G^{\lambda\geq 0}\to X/G.\]
We thus need to find a quasi-isomorphism as in the statement of the proposition for $\pi_{\lambda*}i_{\lambda*}(\mathcal{O}_{X^{\lambda\geq 0}}\otimes \Gamma_{G^{\lambda\geq 0}}(\chi))$. The statement will follow from the Koszul resolution for the complex $i_{\lambda*}(\mathcal{O}_{X^{\lambda\geq 0}}\otimes \Gamma_{G^{\lambda\geq 0}}(\chi))$ and from the Borel-Weil-Bott theorem applied for each term of the Koszul resolution.

Let $\mathcal{N}:=\mathcal{O}_X\otimes (X^{\lambda<0})^\vee$ be a $G^{\lambda\geq 0}$-equivariant bundle on $X$. Consider the Koszul resolution 
$\left(\mathrm{Sym}^\bullet \left(\mathcal{N}[1]\right), d\right)\cong i_{\lambda *}\mathcal{O}_{X^{\lambda\geq 0}}$. 
Then \[i_{\lambda*}(\mathcal{O}_{X^{\lambda\geq 0}}\otimes \Gamma_{G^{\lambda\geq 0}}(\chi))\cong \left(\mathrm{Sym}^\bullet (\mathcal{N}[1])\otimes \Gamma_{G^{\lambda\geq 0}}(\chi), d\right).\]
Let $\pi^\circ\colon X/B\to X/G$ and $\tau_\lambda\colon X/B\to X/G^{\lambda\geq 0}.$ Then $\pi^\circ=\pi_\lambda \tau_\lambda$. By the projection formula and the Borel-Weil-Bott theorem for the group $G^\lambda$, there is an isomorphism \[\left(\mathrm{Sym}^\bullet (\mathcal{N}[1])\otimes \Gamma_{G^{\lambda\geq 0}}(\chi), d\right)\cong \tau_{\lambda *}\left(\tau_\lambda^*\mathrm{Sym}^\bullet(\mathcal{N}[1])\otimes \mathcal{O}(\chi), d\right).\]  
Write $\left(\tau^*_\lambda \mathrm{Sym}^\bullet(\mathcal{N}[1])\otimes\mathcal{O}(\chi), d\right)=\left( \bigoplus_{j\geq 0} A_j[j], d\right)$ for 
\[A_j:=\bigoplus_{|J|=j}\mathcal{O}_X(\chi-\sigma_J),\] where the right hand sum is after all subsets $J\subset \{\beta\in \mathcal{V}\mid \langle \lambda, \beta\rangle<0\}$ of cardinality $j$. 
Then $\pi_{\lambda*}i_{\lambda*}\left(\mathcal{O}_{X^{\lambda\geq 0}}\otimes \Gamma_{G^{\lambda\geq 0}}(\chi)\right)$ is the totalization of the bicomplex $\left(\bigoplus_{j\geq 0}\pi^\circ_*A_j[j], d\right)$. By the Borel-Weil-Bott theorem for the group $G$, the complex $\pi^\circ_*A_j$ is 
\[\pi^\circ_*A_j\cong \bigoplus_{|J|=j}\mathcal{O}_X\otimes \Gamma_{G}\left((\chi-\sigma_J)^+\right)[-\ell(J)],\] and the conclusion thus follows.
\end{proof}

We will mostly use the proposition above for the group $G=GL(d)$ and $X=R(d)$, in conjunction with a version of \cite[Proposition 3.6]{P}.
In the following, we denote by $(r, p)(\chi')<(r, p)(\chi')$
for weights $\chi$, $\chi'$ if 
$r(\chi')<r(\chi)$, or $r(\chi')=r(\chi)$ and $p(\chi')<p(\chi)$. 
Also recall the order of cocharacters in~\ref{subsubsub:order}. 

\begin{prop}\label{rgoesdown}
Let $\chi\in M(d)^+_{\mathbb{R}}$ with $r(\chi)=r>\frac{1}{2}$ and let $\lambda$ be the antidominant cocharacter of $T(d)$ with $\chi\in F_r(\lambda)^{\mathrm{int}}$.
Let
$I$ be a non-empty subset of  $\{\beta\in\mathcal{W} \mid \langle \lambda, \beta\rangle<0\}$.
Then $\chi-\sigma_I$ is in $2r\mathbf{W}(d)$. 
If $\chi-\sigma_I$ lies on $F_r(\mu)^{\mathrm{int}}$, then $\lambda\succ\mu$. In particular, we have that \[(r,p)(\chi-\sigma_I)<(r,p)(\chi).\]
\end{prop}

\subsection{Matrix factorizations and categories of singularities}\label{MF}

Reference for this subsection are \cite[Section 2.2]{T3}, \cite[Subsection~2.2]{T4}, \cite[Section~2.3]{BFK2}, \cite[Section~1]{PoVa3}.

\subsubsection{}\label{mf11} Let $G$ be a reductive group and let $Y$ be a smooth affine scheme with an action of $G$
and a trivial $\mathbb{Z}/2$-action.
Let 
$\mathcal{Y}=Y/G$ be the corresponding quotient stack
and let $f$ be
a regular function 
$f \colon \mathcal{Y}\to \mathbb{C}$. 
We define the dg-category of matrix factorizations $\mathrm{MF}(\mathcal{Y}, f)$. It has objects $(\mathbb{Z}/2\mathbb{Z})\times G$-equivariant factorizations $(P, d_P)$, where $P$ is a $G$-equivariant coherent sheaf on $Y$, 
$\langle 1\rangle$ is the twist corresponding to a non-trivial $\mathbb{Z}/2\mathbb{Z}$-character on $Y$, and 
$d_P \colon  P\to P\langle 1\rangle$ is a morphism satisfying $d_P\circ d_P=f$. Alternatively, the objects of $\mathrm{MF}(\mathcal{Y}, f)$ are tuples
\begin{align}\label{MF:obj}
\left(\alpha \colon E\rightleftarrows F \colon  \beta\right),
\end{align}
where $E$ and $F$ are $ G$-equivariant coherent sheaves
on $Y$, and $\alpha$ and $\beta$ are $G$-equivariant morphisms
such that $\alpha\circ\beta$ and $\beta\circ\alpha$ are multiplication by $f$.
Let $(P', d_{P'})$ be another 
matrix factorization which we also write as $\left(\alpha' \colon E'\rightleftarrows F' \colon  \beta'\right)$. The
category of tuples (\ref{MF:obj})
naturally forms a dg-category, denoted by
$\underline{\mathrm{MF}}(\mathcal{Y}, f)$. 
The dg-category $\mathrm{MF}(\mathcal{Y}, f)$
is defined to be the Drinfeld dg-quotient of 
$\underline{\mathrm{MF}}(\mathcal{Y}, f)$
by its subcategory of absolutely acyclic objects, 
see~\cite[Section~2.3]{BFK2}, \cite[Subsection~2.2]{T4}. It is weak equivalent to the dg-category consisting 
of objects (\ref{MF:obj}) such that $E, F$ are vector 
bundles,
see \cite[Definition 1.2]{PoVa3}.

The category of matrix factorizations is equivalent to the category of singularities of the derived zero locus $\mathcal{Y}_0$:
\[D_{\mathrm{sg}}(\mathcal{Y}_0):=D^b\mathrm{Coh}(\mathcal{Y}_0)/\mathrm{Perf}(\mathcal{Y}_0).\]
The equivalence \begin{equation}\label{MFSG}
    \mathrm{MF}(\mathcal{Y}, f)\stackrel{\sim}{\to} D_{\mathrm{sg}}(\mathcal{Y}_0)
\end{equation} is induced by the functor
$\left(\alpha \colon E\rightleftarrows F \colon  \beta\right)\mapsto \mathrm{Coker}(\alpha)$, see \cite[Theorem 3.10]{O}.

Consider a full dg-subcategory $\mathbb{M}\subset D^b(\mathcal{Y})$. Let $\mathbb{M}_0$ be the subcategory of $D^b\mathrm{Coh}(\mathcal{Y}_0)$ of complexes $F$ such that $i_*(F)\in \mathbb{M}$, where $i\colon \mathcal{Y}_0\to \mathcal{Y}$ is the natural map. Let $\mathbb{M}^p_0\subset \mathbb{M}$ be the full dg-subcategory of perfect complexes. Define 
\[D_{\mathrm{sg}}(\mathbb{M}):=\mathbb{M}_0/\mathbb{M}^p_0.\]
Further, we also define the subcategory of matrix factorizations
\[\mathrm{MF}(\mathbb{M}, f)\subset \mathrm{MF}(\mathcal{Y},f)\]
which is equivalent to $D_{\rm{sg}}(\mathbb{M})$ under 
the equivalence~\eqref{MFSG}. 

In this paper, we will consider $\mathbb{M}$ to be 
generated by some set of vector bundles $\{\mathcal{P}_i \mid i \in I\}$ on $\mathcal{Y}$. 
In this case, we can describe
$\mathrm{MF}(\mathcal{Y}, f)$
in terms of $\mathcal{P}_i$ as follows.
\begin{lemma}\label{lem:MF(M)}
    The subcategory $\mathrm{MF}(\mathbb{M}, f)$ consists of 
    matrix factorizations
    $\left(\alpha \colon E\rightleftarrows F \colon  \beta\right)$
    such that $E$, $F$ are direct sums of 
    vector bundles from $\{\mathcal{P}_i \mid i\in I\}$. 
\end{lemma}
\begin{proof}
    If $E$, $F$ are direct sums
    of vector bundles of $\mathcal{P}_i$, then 
    it corresponds to an object in $D_{\rm{sg}}(\mathbb{M})$
    under the equivalence (\ref{MFSG})
    by the definition of $D_{\rm{sg}}(\mathbb{M})$. 
    Conversely, we take an object $A \in D_{\rm{sg}}(\mathbb{M})$. Up to shift, it is represented by 
    $A \in \mathrm{Coh}(\mathcal{Y}_0)$ such that 
    $i_{\ast}A \in \mathbb{M}$, see~\cite[Lemma~1.10]{O}. 
Since $Y$ is affine, $G$ is reductive and $\mathcal{P}_i$ are vector 
bundles, there is a bounded complex $P^{\bullet}$ whose 
terms are direct sums from $\mathcal{P}_i$ such that 
$P^{\bullet}$ is quasi-isomorphic to $i_{\ast}A$. 
Then the resolution property of 
matrix factorizations in~\cite[Theorem~3.9]{MR3488035} shows that 
there is a matrix factorization
of the form 
$\left(\alpha \colon P^{\rm{odd}}\rightleftarrows P^{\rm{even}} \colon  \beta\right)$
and is isomorphic to $\left(0\rightleftarrows i_{\ast}A \right)$, hence corresponds to $A$
under the equivalence (\ref{MFSG}).
\end{proof}

\subsubsection{}\label{gradedMFdef}
Assume there exists an extra action of $\mathbb{C}^*$ on $Y$ which commutes with the action of $G$ on $Y$, 
and trivial on $\mathbb{Z}/2 \subset \mathbb{C}^{\ast}$.
Assume that $f$ is weight two with respect to 
the above $\mathbb{C}^{\ast}$-action. 
Denote by $(1)$ the twist by the character
$\mathrm{pr}_2 \colon G\times\mathbb{C}^*\to\mathbb{C}^*$. 
Consider the category of graded matrix factorizations $\mathrm{MF}^{\mathrm{gr}}(\mathcal{Y}, f)$. It has objects pairs $(P, d_P)$ with $P$ an equivariant $G\times\mathbb{C}^*$-sheaf on $Y$ and $d_P \colon P\to P(1)$ a $G\times\mathbb{C}^*$-equivariant morphism. 
Note that
as the $\mathbb{C}^{\ast}$-action is trivial on $\mathbb{Z}/2$, 
we have the induced action of
$\mathbb{C}^{\star}=\mathbb{C}^{\ast}/(\mathbb{Z}/2)$ on $Y$
and $f$ is weight one with respect to the above $\mathbb{C}^{\star}$-action. 
The objects of $\mathrm{MF}^{\mathrm{gr}}(\mathcal{Y}, f)$ can be alternatively described as tuples 
\begin{align}\label{tuplet:graded}
(E, F, \alpha \colon 
E\to F(1)', \beta \colon F\to E),
\end{align}
where $E$ and $F$ are $G\times\mathbb{C}^{\star}$-equivariant coherent sheaves
on $Y$, $(1)'$ is the twist by the character 
$G \times \mathbb{C}^{\star} \to \mathbb{C}^{\star}$, 
and $\alpha$ and $\beta$ are $\mathbb{C}^{\star}$-equivariant morphisms
such that $\alpha\circ\beta$ and $\beta\circ\alpha$ are multiplication by $f$.
The dg-category 
$\underline{\mathrm{MF}}^{\mathrm{gr}}(\mathcal{Y}, f)$
of tuples (\ref{tuplet:graded})
is defined similarly to 
$\underline{\mathrm{MF}}(\mathcal{Y}, f)$, 
and the dg-category 
$\mathrm{MF}^{\mathrm{gr}}(\mathcal{Y}, f)$
is defined to be its dg-quotient 
by the subcategory of absolutely acyclic objects. 

Functoriality of categories of (graded or ungraded) matrix factorizations for pullback and proper pushfoward is discussed in \cite{PoVa3}.

The category of graded matrix factorizations is equivalent to the category of singularities of $Y_0/G\times\mathbb{C}^*$:
\[D^{\mathrm{gr}}_{\mathrm{sg}}(\mathcal{Y}_0):=D^b\mathrm{Coh}_{\mathbb{C}^*}(\mathcal{Y}_0)/\mathrm{Perf}_{\mathbb{C}^*}(\mathcal{Y}_0).\]
The equivalence \begin{equation}\label{MFSG2}
    \mathrm{MF}^{\mathrm{gr}}(\mathcal{Y}, f)\stackrel{\sim}{\to} D^{\mathrm{gr}}_{\mathrm{sg}}(\mathcal{Y}_0)
\end{equation} is induced by the functor
$\left(\alpha \colon E\rightleftarrows F \colon  \beta\right)\mapsto \mathrm{Coker}(\alpha)$, see \cite[Theorem 3.9]{o2}.

Consider a full dg-subcategory $\mathbb{M}\subset D^b_{\mathbb{C}^*}(\mathcal{Y})$. 
Let $\mathbb{M}_0$ be the subcategory of $D^b\mathrm{Coh}_{\mathbb{C}^*}(\mathcal{Y}_0)$ of complexes $F$ such that $i_*(F)\in \mathbb{M}$, where $i\colon \mathcal{Y}_0\to \mathcal{Y}$ is the natural map. Let $\mathbb{M}^p_0\subset \mathbb{M}$ be the full dg-subcategory of perfect complexes. Define 
\[D^{\mathrm{gr}}_{\mathrm{sg}}(\mathbb{M}):=\mathbb{M}_0/\mathbb{M}^p_0.\]
Further, using the equivalence \eqref{MFSG}, we also define the subcategory of matrix factorizations
\[\mathrm{MF}^{\mathrm{gr}}(\mathbb{M}, f)\subset \mathrm{MF}^{\mathrm{gr}}(\mathcal{Y},f).\]
In the case that $\mathbb{M}$ is generated by $\mathbb{C}^{\ast}$-equivariant 
vector bundles $\{\mathcal{P}_i \mid i \in I\}$, 
we have the following graded version of Lemma \ref{lem:MF(M)}:

\begin{lemma}\label{lem:MF(M)gr}
    The subcategory $\mathrm{MF}^{\mathrm{gr}}(\mathbb{M}, f)$ consists of 
    matrix factorizations
    $\left(E, F, \alpha \colon E\to F(1)', \beta \colon F\to E\right)$
    such that $E$, $F$ are direct sums of 
    vector bundles from $\{\mathcal{P}_i \mid i\in I\}$. 
\end{lemma}

\subsection{Semiorthogonal decompositions}\label{sod}
Let $I$ be a set. Consider a set $O\subset I\times I$ such that for any $i, j\in I$ we have $(i,j)\in O$, or $(j,i)\in O$, or both $(i,j)\in O$ and $(j,i)\in O$. 

Let $\mathbb{T}$ be a triangulated category. We will construct semiorthogonal decompositions
\[\mathbb{T}=\langle \mathbb{A}_i \mid i \in I \rangle\] with summands 
triangulated subcategories $\mathbb{A}_i$ indexed by $i\in I$
such that for any $i,j\in I$ with $(i, j)\in O$ and objects $\mathcal{A}_i\in\mathbb{A}_i$, $\mathcal{A}_j\in\mathbb{A}_j$, we have 
\[\Hom_{\mathbb{T}}(\mathcal{A}_i,\mathcal{A}_j)=0.
\]

\subsection{Semiorthogonal decompositions for matrix factorizations}

We state a version of \cite[Lemma 3.9]{HLP}, \cite[Proposition 1.10]{OrLG} which constructs a semiorthogonal decomposition of a category of matrix factorizations from a semiorthogonal decomposition of the ambient stack.
Recall the notations from Subsections \ref{mf11} and \ref{gradedMFdef}.

\begin{prop}\label{propsodpotential}

Let $\mathcal{Y}=Y/G$ be a smooth quotient stack with $Y$ affine and let $f\colon \mathcal{Y}\to \mathbb{C}$ be a regular function. 

    (a) Assume $D^b(\mathcal{Y})$ has a $D^b(\mathbb{C})$-linear semiorthogonal decomposition $D^b(\mathcal{Y})=\langle \mathbb{A}_i\mid i\in I\rangle,$ for $I$ a set as in Subsection \ref{sod}. Then there is a semiorthogonal decomposition 
    \[\mathrm{MF}(\mathcal{Y}, f)=\langle \mathrm{MF}(\mathbb{A}_i, f)\mid i\in I\rangle.\]

    (b) Assume we are in the setting of Subsection \ref{gradedMFdef} and assume that $D^b_{\mathbb{C}^*}(\mathcal{Y})$ has a $D^b_{\mathbb{C}^*}(\mathbb{C})$-linear semiorthogonal decomposition $D^b_{\mathbb{C}^*}(\mathcal{Y})=\langle \mathbb{A}_i\mid i\in I\rangle.$
    Then there is a semiorthogonal decomposition
    \[\mathrm{MF}^{\mathrm{gr}}(\mathcal{Y}, f)=\langle \mathrm{MF}^{\mathrm{gr}}(\mathbb{A}_i, f)\mid i\in I\rangle.\]
\end{prop}

\begin{proof}
    We only prove (a), the proof of (b) is similar. We also prove instead that there is a semiorthogonal decomposition 
    $D_{\mathrm{sg}}(\mathcal{Y}_0)=\langle D_{\mathrm{sg}}(\mathbb{A}_i)\mid i\in I\rangle.$
    By \cite[Proposition 3.8]{HLP}, there is a semiorthogonal decomposition 
    \begin{equation}\label{SODHLP}
    \mathrm{Fun}_{D^b(\mathbb{C})}(D^b(\mathrm{pt}), D^b(\mathcal{Y}))=\langle \mathrm{Fun}_{D^b(\mathbb{C})}(D^b(\mathrm{pt}), \mathbb{A}_i)\mid i\in I\rangle.
     \end{equation}
    Denote by $\mathrm{FM}\colon D^b(\mathcal{Y}_0)\xrightarrow{\sim} \mathrm{Fun}_{D^b(\mathbb{C})}(D^b(\mathrm{pt}), D^b(\mathcal{Y}))$ the equivalence induced by Fourier-Mukai functors, see \cite[Theorem 1.1.2]{BZNP}.
     We claim that there is an equiavelence 
     \begin{equation}\label{equivFM}
         \mathrm{FM}\colon \mathbb{A}_{0}\xrightarrow{\sim} \mathrm{Fun}_{D^b(\mathbb{C})}(D^b(\mathrm{pt}), \mathbb{A})
     \end{equation} for a subcategory $\mathbb{A}\subset D^b(\mathcal{Y})$. It suffices to check that the inverse $\mathrm{FM}^{-1}$ sends $\mathrm{Fun}_{D^b(\mathbb{C})}(D^b(\mathrm{pt}), \mathbb{A})$ to $\mathbb{A}_0$.
    Indeed, consider the diagram:
    \begin{equation*}
        \begin{tikzcd}
        \mathrm{Fun}_{D^b(\mathbb{C})}(D^b(\mathrm{pt}), \mathbb{A}) \arrow[d]\arrow[r, hook]& \mathrm{Fun}_{D^b(\mathbb{C})}(D^b(\mathrm{pt}), D^b(\mathcal{Y}))\arrow[d, "\Phi"]\arrow[r, "\sim"]& D^b(\mathcal{Y}_0)\arrow[ddl, bend left, "\Psi"]\\
        \mathrm{Fun}_{D^b(\mathrm{pt})}(D^b(\mathrm{pt}), \mathbb{A}) \arrow[d, "\sim"]\arrow[r, hook]& \mathrm{Fun}_{D^b(\mathrm{pt})}(D^b(\mathrm{pt}), D^b(\mathcal{Y}))\arrow[d, "\sim"]& \\
        \mathbb{A}\arrow[r, hook]& D^b(\mathcal{Y})& 
        \end{tikzcd}
    \end{equation*}
    The equivalences above are induced by functor $\mathrm{FM}^{-1}$. Thus $\Psi=i_*$, where recall that $i\colon \mathcal{Y}_0\hookrightarrow\mathcal{Y}$ is the natural inclusion. Thus $\mathrm{FM}^{-1}\mathrm{Fun}_{D^b(\mathbb{C})}(D^b(\mathrm{pt}), \mathbb{A})$ is a subcategory of $D^b(\mathcal{Y}_0)$ of objects $F$ such that $i_*(F)$ is in $\mathbb{A}$, thus $\mathrm{FM}^{-1}\mathrm{Fun}_{D^b(\mathbb{C})}(D^b(\mathrm{pt}), \mathbb{A})$ is a subcategory of $\mathbb{A}_0$. Then the semiorthogonal decomposition \eqref{SODHLP} becomes 
    \[D^b(\mathcal{Y}_0)=\langle \mathbb{A}_{i,0}\mid i\in I\rangle.\] The claim then follows from \cite[Proposition 1.10]{OrLG}.
\end{proof}

\subsection{Partitions of natural numbers}\label{partitions}

Let $d\in \mathbb{N}$. Denote by $p_2(d)$ the number of partitions of $d$ and by $p_3(d)$ the number of $3$D/ plane partitions of $d$. The generating series corresponding to these two types of partitions are
given as follows: 
\begin{align}\label{macmahon:formula}
    \sum_{d\geq 0}p_2(d)q^d&=\prod_{n\geq 1}\frac{1}{1-q^n},\\
  \notag \sum_{d\geq 0}p_3(d)q^d&=M(q):=\prod_{n\geq 1}\frac{1}{\left(1-q^n\right)^n}.
\end{align}
The second formula is called the MacMahon's formula.

\subsection{DT categories for \texorpdfstring{$\mathbb{C}^3$}{C3}}\label{subsec:catC3}
Let $\mathcal{X}^f(d)$ be the stack from (\ref{def:Xd}). 
We define 
\begin{align*}
\mathrm{NHilb}(d):=\X^f(d)^{\text{ss}}\subset \X^f(d)
\end{align*}
to be the open substack
of stable points with respect to the linearization $\text{det}\colon GL(d)\to \mathbb{C}^*$. Let $V$ be a vector space of dimension $d$.
The points of $\X^f(d)^{\text{ss}}$ correspond to tuples \begin{align*}
(v, X, Y, Z)\in V \oplus \text{End}(V)^{\oplus 3}
\end{align*}
such that $v$ generates $V$ under the action of $X, Y,$ and $Z$.
The stack $\X^f(d)^{\text{ss}}$ is a smooth variety called the noncommutative Hilbert scheme of $d$ points on $\mathbb{C}^3$. 

Consider the following 
super-potential of $Q$
\begin{align*}
    W=z[x, y] \in \mathbb{C}[Q]/[\mathbb{C}[Q], \mathbb{C}[Q]], 
\end{align*}
where $x, y, z\in \mathbb{C}[Q]$ correspond to the three loops 
of the quiver $Q$. 
Consider the regular function 
\begin{align}\label{def:Wd}
\Tr W \colon  \X(d)\to \mathbb{C}, \ (X, Y, Z) \mapsto \Tr Z[X, Y]. 
\end{align}
It is well-known (and easy to see)
that the critical locus 
$\mathrm{Crit}(\Tr W)$
in $\X(d)$
is isomorphic to 
the moduli stack of zero-dimensional sheaves on 
$\mathbb{C}^3$ with length $d$. 
We also write $\Tr W$ as $\Tr W_d$
when it is necessary to remember the dimension. 

The super-potential $W$ 
also induces 
a super-potential for $Q^f$, 
hence an 
analogous regular function on $\X^f(d)$
which coincides with the 
pull-back of $\Tr W$
by the projection to $\mathcal{X}(d)$. 
In particular, we have the regular function  
\begin{equation}\label{equation:NHilbTrW}
    \Tr W \colon \text{NHilb}(d)\to \mathbb{C}.
\end{equation}
The critical locus of 
$\Tr W$ in $\text{NHilb}(d)$
is isomorphic to the Hilbert scheme $\text{Hilb}(\mathbb{C}^3, d)$ of $d$ points on $\mathbb{C}^3$, see~\cite[Proposition~3.1.1, Remark~3.1.2]{DimSz}.

We will also consider equivariant and graded versions of matrix factorizations for the regular function \eqref{equation:NHilbTrW}.
The group $(\mathbb{C}^{\ast})^3$
acts on the edges $(x, y, z)$
of the quiver $Q$ by scalar multiplication. 
Consider the two dimensional subtorus 
$T \subset (\mathbb{C}^{\ast})^3$
which preserves the super-potential $W$
\begin{align}\label{torus:T}
(\mathbb{C}^{\ast})^2
\stackrel{\cong}{\to}
    T:=\{(t_1, t_2, t_3) \in (\mathbb{C}^{\ast})^{\times 3} \mid t_1 t_2 t_3=1\}. 
\end{align}
The isomorphism above is given by $(t_1, t_2) \mapsto (t_1, t_2, t_1^{-1} t_2^{-1})$.
Then $T$ acts on $\mathcal{X}(d)$, 
$\mathcal{X}^f(d)$
and $\mathrm{NHilb}(d)$
preserving $\Tr W$. 
We will also consider the 
grading given by 
$\mathbb{C}^*$ scaling with weight $2$ the space $\mathfrak{gl}(V)$ for $V$ a vector space. For example, we can choose an edge $e\in \{x, y, z\}$ of $Q$ and let $\mathbb{C}^*$ scale with weight $2$ the linear map corresponding to $e$. 
Contrary to the $T$-action, the super-potential 
$W$ has weight $2$ with respect to such a grading. 

In the above setting, 
we define DT categories as follows: 

\begin{defn}\label{def:catDT2}
The DT category for $\mathbb{C}^3$ is defined by 
\begin{align}\label{cat:DT}
\mathcal{DT}(d):=\text{MF}(\text{NHilb}(d), \Tr W).
\end{align}
We also define some variants of 
it, the graded version, the $T$-equivariant version, and 
the graded $T$-equivariant version
\begin{align*}
  	\mathcal{DT}^{\bullet}_{\ast}(d)
	&:=\text{MF}^{\bullet}_{\ast}(\mathrm{NHilb}(d), \Tr W), \ 
	\ast \in \{\emptyset, T\}, \bullet\in \{\emptyset, \mathrm{gr}\}. 
\end{align*}
\end{defn}
The categories $\mathcal{DT}(d)$
are categorifications of the Donaldson-Thomas invariants of $\mathbb{C}^3$ in the following sense.
Let $\text{HP}_\bullet$ denote the periodic cyclic homology of a small dg-category. Then $\text{HP}_i$ are $\mathbb{C}(\!(u)\!)$-vector spaces for $i\in \mathbb{Z}/2$. The dimensions over $\mathbb{C}(\!(u)\!)$ of the periodic cyclic homologies of (\ref{cat:DT})
are finite and
\[
\sum_{i\in\{0,1\}} (-1)^i \dim_{\mathbb{C}(\!(u)\!)} \text{HP}_i\left(\text{MF}(\text{NHilb}(d), \Tr W)\right)=(-1)^d\text{DT}_d, 
\] where $\text{DT}_d$ is the Donaldson-Thomas invariant of $\mathbb{C}^3$ for zero-dimensional closed subschemes
with length $d$, see \cite[Theorem 1.1]{E} for the computation of $\text{HP}_\cdot$ in terms of vanishing cohomology and \cite[Subsection 2.4]{Sz} for the relation between DT invariants and vanishing cohomology, see also \cite[Subsection 3.3]{T} for a discussion of categorification of more general DT invariants. 

By \cite[Corollary 4.3]{BF}, we have the equality
\begin{equation}\label{BehFan}
    \text{DT}_d=(-1)^dp_3(d).
\end{equation}
In particular, the generating series of DT-invariants of $\mathbb{C}^3$ has the following product expansion
\begin{equation}\label{MacMahon}
    \sum_{d\geq 0}\text{DT}_d\,q^d=M(-q)=\prod_{n\geq 1}\frac{1}{\left(1-(-q)^n\right)^n}
\end{equation}
using MacMahon's formula from Subsection \ref{partitions}.

\section{A categorical analogue of MacMahon's formula}\label{section:MacMahon}

In this section, we prove Theorem~\ref{MacMahonthm}, a slightly more general and precise version
of Theorem~\ref{thm:intro1}. 
\subsection{Main theorem}
Let us recall 
the diagram (\ref{dia:Hilb}) from the introduction. 
As we mentioned, the two sides in (\ref{dia:Hilb}) do not give 
a derived equivalence for $X=\mathbb{C}^3$. 
We begin with an observation on the diagram (\ref{dia:Hilb})
for the Hilbert scheme of two points on $\mathbb{C}^3$:

\begin{example}\label{exam:d2}
    Let $X=\mathbb{C}^3$
    and $d=2$. 
    In this case, the scheme $\mathrm{Hilb}(\mathbb{C}^3, 2)$
    is smooth and the left arrow in (\ref{dia:Hilb}) 
    is the blow-up of the diagonal 
    $\mathbb{C}^3=\Delta \subset \mathrm{Sym}^2(\mathbb{C}^3)$. 
    So we have the diagram 
    \begin{align*}
        \xymatrix{
        E \ar@<-0.3ex>@{^{(}->}[r]^-{i} \ar[d]_-{p} & \mathrm{Hilb}(\mathbb{C}^3, 2) \ar[d]  \\
        \Delta \ar@<-0.3ex>@{^{(}->}[r] & \mathrm{Sym}^2(\mathbb{C}^3)
        }
    \end{align*}
    where $E$ is the exceptional locus of the blow-up of $\Delta$ and
    where $p$ is a $\mathbb{P}^2$-bundle. 
    Instead of an equivalence like (\ref{intro:McKay}), we have the 
    semiorthogonal decomposition (which easily follows from~\cite[Example~8.8 (1)]{Kawamata}):
    \begin{align}\label{sod:exam}
    D^b(\mathrm{Hilb}(\mathbb{C}^3, 2)) =
    \langle i_{\ast}(p^{\ast}D^b(\Delta) \otimes \mathcal{O}(-1)), 
    D^b((\mathbb{C}^3)^{\times 2}/\mathfrak{S}_2)). 
    \end{align}
    Note that the first summand has one generator and the second 
    summand has two generators, so $D^b(\mathrm{Hilb}(\mathbb{C}^3, 2))$
    has $p_3(2)=3$ generators. 
\end{example}

It is known that for $d\geq 4$, the Hilbert scheme 
of points $\mathrm{Hilb}(\mathbb{C}^3, d)$
is a highly singular scheme, so 
its derived category is not well-behaved. 
Instead, we consider the categories of matrix factorizations $\mathcal{DT}(d)$
as in Definition~\ref{def:catDT2}, 
and construct a semiorthogonal decomposition 
similar to (\ref{sod:exam}). 
The main result of this section is the following:
\begin{thm}\label{MacMahonthm}
Let 
$\mu \in \mathbb{R}$ with $2k \mu \notin \mathbb{Z}$ for $1\leq k \leq d$. Then
there is a semiorthogonal decomposition 
\[\mathcal{DT}(d)=
\left\langle \boxtimes_{i=1}^k \mathbb{S}(d_i)_{w_i}\right\rangle,\] where the right hand side consists of all tuples $A=(d_i, w_i)_{i=1}^k$ 
with $\sum_{i=1}^k d_i=d$
such that its associated $A'=(d_i, v_i)_{i=1}^k$, see Subsection \ref{prime}, satisfies \begin{equation}\label{i}
    -\mu\leq \frac{v_1}{d_1}<\cdots<\frac{v_k}{d_k}\leq 1-\mu.\end{equation}
    Consider the set $O$ from \eqref{def:setO} and assume that $A=(d_i, w_i)_{i=1}^k$ and $B=(e_i, u_i)_{i=1}^s$ are as above and $(A, B)\in O$. Let $\mathcal{A}\in \boxtimes_{i=1}^k \mathbb{S}(d_i)_{w_i}$ and $\mathcal{B}\in \boxtimes_{i=1}^s \mathbb{S}(e_i)_{u_i}$. Then $\Hom(\mathcal{A}, \mathcal{B})=0$.
    There are also analogous semiorthogonal decompositions of the variants of 
    DT categories 
    $\mathcal{DT}^{\mathrm{gr}}(d)$, 
    $\mathcal{DT}_T(d)$, $\mathcal{DT}_T^{\mathrm{gr}}(d)$
    defined in Definition~\ref{def:catDT2}. 
\end{thm}
The categories of generators 
$\mathbb{S}(d)_w$ will be 
defined in Subsection \ref{ss2}. In the ungraded (and non-equivariant) case, the product above is the product of dg-categories over $\mathbb{C}(\!(\beta)\!)$, where $\beta$ has homological degree $-2$, see \cite[Subsection 1.2, Theorem 4.1.3]{Preygel}. In the graded and non-equivariant) case, the product is the usual product of dg-categories over $\mathbb{C}$.
\begin{example}
    For $\mu=1-\varepsilon$ with $0<\varepsilon \ll 1$, 
    the semiorthogonal decompositions 
    in Theorem~\ref{MacMahonthm} for $2\leq d \leq 4$ 
    are given as follows: 
    \begin{align*}
        \mathcal{DT}(2)=&\langle \mathbb{S}(2)_{-1}, \mathbb{S}(2)_0 \rangle, \\
        \mathcal{DT}(3)=&\langle \mathbb{S}(3)_{-2}, \mathbb{S}(2)_{-3} \boxtimes 
        \mathbb{S}(1)_2, \mathbb{S}(3)_{-1}, \mathbb{S}(3)_0 \rangle, \\
        \mathcal{DT}(4)=&\langle \mathbb{S}(4)_{-3}, \mathbb{S}(3)_{-5} \boxtimes 
        \mathbb{S}(1)_{3}, \mathbb{S}(4)_{-2}, 
        \mathbb{S}(2)_{-5} \boxtimes \mathbb{S}(2)_4,  \\
        &  \quad  \mathbb{S}(3)_{-4} \boxtimes 
        \mathbb{S}(1)_3, \mathbb{S}(4)_{-1}, \mathbb{S}(4)_0 \rangle. 
    \end{align*}
    Compare the semiorthogonal decomposition for $\mathcal{DT}(2)$
    with (\ref{sod:exam}). 
\end{example}

We briefly explain an outline of the proof of Theorem \ref{MacMahonthm}.
Let $\delta=d\mu \tau_d \in M_{\mathbb{R}}$. 
Let $\mathbb{D}(d; \delta)$ 
be the dg-subcategory of $D^b(\X(d))$ generated by the vector bundles $\mathcal{O}_{\X(d)}\otimes\Gamma_{GL(d)}(\chi)$ for $\chi$ a dominant weight such that
\begin{align*}
\chi+\rho+\delta\in \textbf{V}(d),
\end{align*}
where $\textbf{V}(d)\subset M(d)_{\mathbb{R}}$ is defined in Subsection \ref{ss1}. We use 
the same techniques as in the study
of categorical Hall algebras of quivers with potential \cite{P} to prove in Proposition \ref{prop1} that
\[\mathbb{D}(d; \delta)=
\left\langle \ast_{i=1}^k \mathbb{M}(d_i)_{v_i+d_i\left(\sum_{i>j}d_j-\sum_{j>i}d_j\right)}\right\rangle,\] where $\ast$ denotes the product in the Hall algebra.
Let $a \colon \text{NHilb}(d)\to \X(d)$ be the projection. 
Using results of Halpern-Leistner \cite[Theorem 2.10]{halp} and Halpern-Leistner--Sam \cite[Theorem 3.2]{hls}, the category $D^b(\text{NHilb}(d))$ is generated by the complexes $a^*\mathcal{F}$ for $\mathcal{F}\in \mathbb{D}(d; \delta)$. We prove that there is a semiorthogonal decomposition
\[D^b(\text{NHilb}(d))=
\left\langle \boxtimes_{i=1}^k \mathbb{M}(d_i)_{v_i+d_i\left(\sum_{i>j}d_j-\sum_{j>i}d_j\right)}\right\rangle,\] see \eqref{SODnhilb}.
The semiorthogonal 
decomposition for matrix factorizations then follows formally, see for example \cite[Proposition 2.1]{P0}.

\subsection{Categories of generators}\label{ss2} 
We now introduce notation and terminology which appear in Theorem~\ref{MacMahonthm}. 
We will be using the stacks introduced in Subsection \ref{sss1} and the polytopes defined in Subsection \ref{ss1}.

\subsubsection{}\label{ss:Ddelta}
For $w \in \mathbb{Z}$, 
we denote by $D^b(\mathcal{X}(d))_w$
the subcategory of $D^b(\mathcal{X}(d))$
consisting of objects of
weight $w$ with respect to the diagonal 
cocharacter $1_d$ of $T(d)$.  
We have the direct sum decomposition 
\begin{align*}
    D^b(\mathcal{X}(d))=\bigoplus_{w\in \mathbb{Z}}
    D^b(\mathcal{X}(d))_w. 
\end{align*}
We define the dg-subcategories 
\begin{align*}
    \mathbb{M}(d) \subset D^b(\X(d)), \ 
    (\mbox{resp. }
    \mathbb{M}(d)_w \subset D^b(\X(d))_w)
\end{align*}
to be generated 
by the vector bundles $\OO_{\X(d)}\otimes \Gamma_{GL(d)}(\chi)$, where $\chi$ is a dominant weight of $T(d)$ such that
\begin{equation}\label{M}
    \chi+\rho\in \textbf{W}(d), \ 
    (\mbox{resp. } \chi+\rho \in \textbf{W}(d)_w). 
    \end{equation}
    Note that $\mathbb{M}(d)$
    decomposes into the direct sum of $\mathbb{M}(d)_w$
    for $w \in \mathbb{Z}$. 
    Moreover, taking the tensor product with respect to 
    the determinant character $\det \colon GL(d) \to \mathbb{C}^{\ast}$
    gives an equivalence
    \begin{align}\label{equiv:periodic}
    \otimes \det \colon 
    \mathbb{M}(d)_w \stackrel{\sim}{\to} \mathbb{M}(d)_{d+w}. 
    \end{align}

We fix $\mu \in \mathbb{R}$
and set $\delta := d \mu \tau_d$. 
We define the dg-subcategories
\begin{align}\label{subcat:D}
    \mathbb{D}(d; \delta)
    \subset D^b(\X(d)), \ 
    \mbox{(resp. } \mathbb{E}(d; \delta)
\subset D^b(\X^f(d)) )
    \end{align}
    to be generated by the vector bundles $\OO_{\X(d)}\otimes \Gamma_{GL(d)}(\chi)$
    (resp.~
    $\OO_{\X^f(d)}\otimes \Gamma_{GL(d)}(\chi)$),
    where $\chi$ is a dominant weight of $T(d)$ such that
\begin{equation}\label{D}
    \chi+\rho+\delta\in \textbf{V}(d).\end{equation}
 Note that for the projection $a \colon \X^f(d) \to \X(d)$,
 the pull-back $a^{\ast}$ restricts to the functor 
 \begin{align*}
     a^{\ast} \colon \mathbb{D}(d; \delta) \to \mathbb{E}(d; \delta)
 \end{align*}
 whose essential image generates $\mathbb{E}(d; \delta)$. 

\subsubsection{}\label{gradingMF}
Let $\Tr W$ be the regular function (\ref{def:Wd}). 
We define the subcategory
\[\mathbb{S}(d):=\text{MF}(\mathbb{M}(d), \Tr W)
\subset \mathrm{MF}(\mathcal{X}(d), \Tr W)
\] 
as in Subsection \ref{mf11}.
It decomposes into the direct sum of 
$\mathbb{S}(d)_w$ for $w \in \mathbb{Z}$, 
where $\mathbb{S}(d)_w$
is defined similarly to $\mathbb{S}(d)$
using $\mathbb{M}(d)_w$. 

In the setting of Definition~\ref{def:catDT2}, 
we also consider subcategories 
for $\ast \in \{\emptyset, T\}$, 
$\bullet \in \{\emptyset, \text{gr}\}$
defined as in Subsections \ref{mf11} or \ref{gradedMFdef}: 
\[\mathbb{S}^{\bullet}_{\ast}(d):=
\text{MF}^{\bullet}_{\ast}(\mathbb{M}(d), \Tr W)
\subset \text{MF}^{\bullet}_{\ast}(\X(d), \Tr W). 
\] 
The subcategory $\mathbb{S}^{\bullet}_{\ast}(d)_w$
is also defined in a similar way.  
As we mentioned in the introduction, we call 
the subcategories $\mathbb{S}^{\bullet}_{\ast}(d)_w$ \textit{quasi-BPS categories} because of their similarity
to BPS invariants (BPS sheaves) in (cohomological) DT theory. 
As in (\ref{equiv:periodic}), 
there is an equivalence 
\begin{align}\label{equiv:periodic2}
    \otimes \det \colon 
    \mathbb{S}_{\ast}^{\bullet}(d)_w \stackrel{\sim}{\to} \mathbb{S}_{\ast}^{\bullet}(d)_{d+w}. 
    \end{align}


\subsection{Preliminaries on weights}

In this subsection, we recall a decompositions of weights from \cite{P2}. We also introduce the sets of partitions of $(d,w)$ which are used in the semiorthogonal decomposition of $D^b(\X(d))$ from loc. cit.

\subsubsection{}\label{prime} 

Let $A$ be a partition $(d_i, w_i)_{i=1}^k$ of $(d, w)$ and consider its corresponding antidominant cocharacter $\lambda$. Define the weights
\begin{align*}
    \chi_A:=\sum_{i=1}^k w_i\tau_{d_i},\
    \chi'_A:=\chi_A+\mathfrak{g}^{\lambda>0}.
\end{align*}
Consider weights $\chi'_i\in M(d_i)_{\mathbb{R}}$ such that
$\chi'_A=\sum_{i=1}^k \chi'_i$. 
 Let $v_i$ be the sum of coefficients of $\chi'_i$ for $1\leq i\leq k$; alternatively, $v_i:=\langle 1_{d_i}, \chi'_i\rangle$. We denote the above transformation by
  \begin{align}\label{trans:A}
     A\mapsto A', \ 
     (d_i, w_i)_{i=1}^k\mapsto (d_i, v_i)_{i=1}^k.
 \end{align}
Explicitly, the weights $v_i$ for $1\leq i\leq k$ are given by 
\begin{align}\label{w:prime}
    v_i=w_i+d_i\left(\sum_{j>i}d_j -\sum_{j<i}d_j \right). 
\end{align}

\subsubsection{}\label{dectree2}

We will use a decomposition of dominant weights from \cite[Subsection 3.2.8]{P2}. 

Before stating it, we introduce some notations. For $d_a$ a summand of a partition of $d$, denote by $M(d_a)\subset M(d)$ the subspace as in the decomposition from Subsection~\ref{id} and let $A\subset \{1,\ldots, d\}$ be the set of indices of weights of standard representation corresponding to $M(d_a)\subset M(d)$. 
Assume that $\ell$ is a partition of a dimension $d_a\in \mathbb{N}$, alternatively, $\ell$ is an edge of the tree $\mathcal{T}$ introduced in Subsection \ref{tree}.
Let $\lambda_{\ell}$ be the corresponding 
antidominant cocharacter of $T(d_a)$. Let \[\mathcal{W}_{\ell}\subset \{(\beta_i-\beta_j)^{\times 3} \mid i, j \in A\}\] be the multiset of weights
with $\langle \lambda_{\ell}, \beta\rangle>0$. Define \[\mathfrak{g}_\ell:=\sum_{\beta\in\mathcal{W}_{\ell}}\beta.\]
The following is 
\cite[Subsection 3.2.8]{P2}, \cite[Subsection 3.1.2]{P0}:

\begin{prop}\label{prop:decompchi}
Let $\chi$ be a dominant weight in $M(d)_\mathbb{R}$ and let $w=\langle 1_d, \chi\rangle$. 
There exists:
\begin{enumerate}
    \item a path of partitions $T$, see Subsection \ref{tree}, with decomposition 
$(d_i)_{i=1}^k$ at the end vertex,
\item coefficients $r_\ell$ for $\ell\in T$ such that $r_{\ell}>1/2$ 
if $\ell$ corresponds to a partition with length $>1$, 
and $r_{\ell}=0$ otherwise; further, if $\ell, \ell'\in T$ are vertices 
corresponding to partitions with length $>1$, and 
with a path from $\ell$ to $\ell'$, then $r_{\ell}> r_{\ell'}> \frac{1}{2}$, and
\item dominant weights $\psi_i\in\mathbf{W}(d_i)_0$ for $1\leq i\leq k$
such that:
\end{enumerate}
\begin{equation}\label{decompchi}
    \chi+\rho+\delta=-\sum_{\ell\in T}r_\ell(3\mathfrak{g}_\ell)+\sum_{i=1}^k\psi_i+(w+d\mu)\tau_d.
\end{equation}
\end{prop}

We briefly explain how to write $\chi$ in the form \eqref{decompchi}. If $\chi+\rho+\delta$ has $r$-invariant $\leq \frac{1}{2}$, then $T$ consists 
of one vertex corresponding to the partition $d$ of length $1$. Otherwise, we want to find the smallest homothetic polygon to $\textbf{W}(d)$ which contains $\chi+\rho+\delta$ on its boundary. This polygon is $2r\textbf{W}(d)$, where $r$ is the $r$-invariant of $\chi+\rho+\delta$. Next, we choose the face of $2r\textbf{W}(d)$ which contains $\chi+\rho+\delta$ in its interior: let $\lambda$ be an antidominant cocharacter of $SL(d)\cap T(d)$ such that $\chi+\rho+\delta\in F_r(\lambda)^{\text{int}}$. Assume $\lambda$ has associated partition $(e_i)_{i=1}^n$.
The weights $\chi+\rho+\delta$ on $F_r(\lambda)^{\text{int}}$ can be written as 
\begin{equation}\label{firststepdecomp}
\chi+\rho+\delta=-r(3\mathfrak{g}^{\lambda>0})+\psi,
\end{equation}
where $r(\psi)=s<r$, $\psi$ is $GL(d)^\lambda$-dominant, and \[\psi\in 2s\mathbf{W}\left(GL(d)^\lambda\right):=\left(\bigoplus_{1\leq i\leq n}\,2s\mathbf{W}(e_i)_0\right)\oplus \mathbb{R}\tau_d\subset \bigoplus_{1\leq i\leq n}\,M(e_i)_{\mathbb{R}}\cong M(d)_{\mathbb{R}},\]
see \cite[Corollary 3.4]{P}.
The partition $\ell=(e_i)_{i=1}^n$ is the maximal partition in $T$ associated to $\chi$ in \eqref{decompchi}. If the weight $\psi$ has $r$-invariant $\leq \frac{1}{2}$, we stop. If not, we continue to use the process above and the decomposition \eqref{firststepdecomp} to further decompose the weight $\psi$, and thus the weight $\chi+\rho+\delta$, until we reach \eqref{decompchi}.
\medskip

Let $M(d)^+_w$ be the subset of $M(d)$ of integral dominant weights $\chi$ with $\langle 1_d, \chi\rangle=w$. Define $M(d)^+_{w,\mathbb{R}}$ analogously.
We denote by $S^d_w$ the set of all paths of partitions $T$ with coefficients $r_\ell$ for $\ell\in T$ satisfying (2) from the statement of Proposition \ref{prop:decompchi} for a dominant integral weight $\chi\in M(d)^+_w$. There is a function:
\[\Theta\colon M(d)^+_w\to S^d_w, \,\, \Theta(\chi)=(T, r_\ell).\]
Let $S'^d_w$ be the set of all paths of partitions $T$ with coefficients $r_\ell$ for $\ell\in T$ satisfying (2) from the statement of Proposition \ref{prop:decompchi} for a dominant real weight $\chi\in M(d)_{w,\mathbb{R}}$. Thus $S^d_w\subset S'^d_w$. We also denote by $\Theta$ the induced function 
\[\Theta\colon M(d)^+_{w,\mathbb{R}}\to S'^d_w.\]

\subsubsection{}\label{subsec333} We continue the discussion from the previous subsection. In this subsection, we give a more explicit description of the set $S^d_w$.
Write 
\begin{align*}
    \chi=\sum_{i=1}^k \chi_i, \ 
    \chi_i \in M(d_i), \ w_i :=\langle 1_{d_i}, 
    \chi_i\rangle. 
\end{align*}
Let $\lambda$ be the antidominant cocharacter 
which corresponds to the decomposition 
$(d_i)_{i=1}^k$. Let $\rho_i$ be half the sum of positive roots of $\mathfrak{gl}(d_i)$ for $1\leq i\leq k$.
Using that
\begin{align}\label{rho:id}
    \rho-\sum_{i=1}^k \rho_i=
    \frac{1}{2}\mathfrak{g}^{\lambda<0} 
    =-\frac{1}{2}\mathfrak{g}^{\lambda>0},
\end{align}
we can rewrite (\ref{decompchi}) as 
\begin{align}\label{decompchibis}
    \chi=-\sum_{\ell\in T}r_l(3\mathfrak{g}_l)+\frac{1}{2}
    \mathfrak{g}^{\lambda>0}+w\tau_d+\sum_{i=1}^k 
    (\psi_i-\rho_i). 
\end{align}
Since the first three terms are of the form 
$\sum_{i=1}^k x_i \tau_{d_i}$
for some $x_i\in \mathbb{R}$ for $1\leq i\leq k$
and $\langle 1_{d_i}, \psi_i-\rho_i \rangle=0$, 
we can write 
\begin{align}\label{chi:i}
\chi=\sum_{i=1}^k w_i \tau_{d_i}+\sum_{i=1}^k(\psi_i-\rho_i).  
\end{align}
In particular, we have 
$\chi_i+\rho_i=\psi_i+w_i\tau_{d_i} \in \textbf{W}(d_i)_{w_i}$. 

Define $v_i$ for $1\leq i\leq k$ by the formula 
\begin{equation}\label{wprime}
    \sum_{i=1}^k v_i\tau_{d_i}=-\sum_{\ell\in T}3\left(r_\ell-\frac{1}{2}\right)\mathfrak{g}_\ell+w\tau_d
\end{equation} and let 
\begin{equation}\label{Aprime232}
A=(d_i, w_i)_{i=1}^k, \ 
    A'=(d_i, v_i)_{i=1}^k. 
\end{equation}
By (\ref{decompchibis}) and (\ref{wprime}),
we have the identity
\begin{align*}
    \chi+\mathfrak{g}^{\lambda>0}=\sum_{i=1}^k v_i
    \tau_{d_i}+\sum_{i=1}^k (\psi_i -\rho_i). 
\end{align*}
It follows that  
the partitions $A$ and $A'$ are related by the 
transformation (\ref{trans:A}).
\begin{prop}\label{prop:dominanttree}
	Let $\chi \in M$ be a dominant weight
	and write $\chi+\rho+\delta$ as in (\ref{decompchi}). 
	Let $v_i$ for $1\leq i\leq k$ be defined by (\ref{wprime}). 
	Then we have that \begin{equation}\label{ineq}
    \frac{v_1}{d_1}<\cdots<\frac{v_k}{d_k}.
\end{equation}
\end{prop}

\begin{proof}
Let $s_i:=\frac{v_i}{d_i}$ for $1\leq i\leq k$.
Let $\phi$ 
be the maximal (antidominant) cocharacter realizing the $r$-invariant of $\chi+\rho+\delta$ and assume that the partition associated to $\phi$ is \[(d_1+\cdots+d_{a_1}, d_{a_1+1}+\cdots+d_{a_2},\ldots, d_{a_b+1}+\cdots+d_k)\] for $1\leq a_1<a_2<\cdots<a_b<k$ for some $b\in\mathbb{N}$. Using induction on $d$, it suffices to check that $s_{a_i}<s_{a_i+1}$ for $1\leq i\leq b$. 
We assume for simplicity that $b=1$, and let $a:=a_1$, $r:=r(\chi+\rho+\delta)$,  $\sigma:=d_1+\cdots+d_a$, and $\tau:=d-\sigma$. 
We have that
\[\chi+\rho+\delta=-3r\mathfrak{g}^{\phi>0}+\chi_1+\chi_2,\] with $\chi_1 \in 2r' \mathbf{W}(\sigma)$ and $\chi_2 \in 2r'\mathbf{W}(\tau)$ such that $r'<r$.
We need to check that $s_{a}<s_{a+1}$. The coefficient of $\beta_{\sigma}$ in \eqref{wprime} is at most the one computed by setting $r_{\ell}=r$ for all $\ell$. Therefore we have the inequality
\begin{equation}\label{wprimea}
s_{a}\leq -3\left(r-\frac{1}{2}\right)\tau+3\left(r-\frac{1}{2}\right)(\sigma-1)
+\frac{w}{d}.
\end{equation}
Similarly, the 
coefficient of $\beta_{\sigma+1}$ is at least
\begin{equation}\label{wprimea2}
s_{a+1}\geq 3\left(r-\frac{1}{2}\right)\sigma-3\left(r-\frac{1}{2}\right)(\tau-1)
+\frac{w}{d},
\end{equation}
and thus $s_{a}<s_{a+1}$.
\end{proof}

Let $T^d_w$ be the set of partitions $A=(d_i, w_i)_{i=1}^k$ of $(d, w)$ such that, for its corresponding $A'=(d_i, v_i)_{i=1}^k$ in \eqref{Aprime232}, the inequality \eqref{ineq} holds.
By Proposition \ref{prop:dominanttree}, there is a function 
\[\varphi\colon S^d_w\to T^d_w.\] 

In the remaining of this subsection, we show that the function $\varphi$ is an isomorphism.
For $A=(d_i, w_i)_{i=1}^k$ in $T^d_w$, consider the weight $\chi_A:=\sum_{i=1}^k w_i\tau_{d_i}$. Then $\chi_A\in M(d)^+_{w,\mathbb{R}}$ is a dominant real weight, so it has a corresponding 
$\Theta(\chi_A)=(T', r'_{\ell'})\in S'^d_w$.
Thus there is a function
\[\psi\colon T^d_w\to S'^d_w,\,\, \psi(A)=\Theta(\chi_A)=(T',r'_{\ell'}).\]

\begin{prop}\label{propdecomrprime}
    We have that 
\[\chi_A=-\sum_{\ell'\in T'}r'_{\ell'}(3\mathfrak{g}_{\ell'})+w\tau_d.\]
\end{prop}

\begin{proof}
We use induction on $k$.
Let the $r$-invariant of $\chi_A$ be $r'$, let $\phi$ be an antidominant cocharacter of $SL(d)\cap T(d)$ such that
that $\chi+\rho\in F_{r'}(\phi)^{\mathrm{int}}$, and let $\underline{e}=(e_i)_{i=1}^s$ be the partition corresponding to $\phi$.
Then there exists $r''<r'$ such that
\[\chi_A+\rho=-3r'\mathfrak{g}^{\phi>0}+\sum_{i=1}^s (\chi_i+\rho_i)+w\tau_d\] with $\chi_i+\rho_{i}\in 2r''\textbf{W}(e_i)$ for all $1\leq i\leq s$, see \cite[Corollary 3.4]{P}. By an argument similar to the one used to prove Proposition \ref{prop:dominanttree}, we have that $\dd\geq \underline{e}$. Further, we may use the induction hypothesis on each $\chi_i$ for $1\leq i\leq s$, and then the conclusion follows. 
\end{proof}

\begin{prop}\label{functionpsi}
    The image of the function $\psi$ lies in $S^d_w$.
\end{prop}

\begin{proof}
    For $1\leq i\leq k$, there are dominant weights $\theta_i\in M(d_i)_{0, \mathbb{R}}$ which are linear combinations with coefficients in $[0,1]$ of weights in $\mathfrak{gl}(d_i)$ such that $\chi':=\chi_A+\sum_{i=1}^k \theta_i$ is an integral dominant weight. Note that 
    \[\chi'+\rho=-\sum_{\ell'\in T'}r'_{\ell'}(3\mathfrak{g}_{\ell'})+\sum_{i=1}^k (\theta_i+\rho_i)+w\tau_d\] and $\theta_i+\rho_i\in \textbf{W}(d_i)_0$.
    By the uniqueness of the decomposition \ref{decompchi} for $\chi'$, we see that
    \[\Theta(\chi_A)=(T', r'_{\ell'})=\Theta(\chi')\in S^d_w.\]
\end{proof}

\begin{prop}\label{ST}
    The functions $\varphi$ and $\psi$ are inverses to each other. There is thus an isomorphism 
    \[\varphi\colon S^d_w\xrightarrow{\sim} T^d_w.\]
\end{prop}

\begin{proof}

Let $\chi\in M(d)^+_w$ and let $\Theta(\chi)=(T, r_\ell)$. From the decomposition \eqref{decompchi} for $\chi$ and \eqref{chi:i}, we have that 
    \[-\sum_{\ell\in T}r_\ell(3\mathfrak{g}_\ell)+\sum_{i=1}^k \rho_i+w\tau_d=\sum_{i=1}^k w_i\tau_{d_i}+\rho.\]
    By the uniqueness of the decomposition \eqref{decompchi} for the dominant real weight $\chi_A=\sum_{i=1}^k w_i\tau_{d_i}$, we see that $\psi(A')=(T, r_\ell)$, so $\psi\varphi$ is the identity.  

    The composition $\varphi\psi$ is the identity from Propositions \ref{propdecomrprime} and \ref{functionpsi}.
\end{proof}

\subsection{Decompositions of weights}

In this subsection, we further restrict the set $S^d_w$ to a set of partitions of $(d,w)$ which will index the summands of $\mathcal{DT}(d)$ from Theorem \ref{MacMahonthm}. The main result is the existence of decomposition of weights as in Proposition \ref{dominantij}.
\medskip

Consider the sets 
\begin{align}\label{def:CDE}
C=\{(i, j)\mid \,1\leq i,j\leq d\}, \ 
D=\{(i, j) \mid \,1\leq j<i\leq d\}, \ 
E=\{1,\ldots, d\}.
\end{align}

\begin{prop}\label{dominantij}
Let $\chi\in M(d)_\mathbb{R}$ be a strictly dominant weight such that  
\[\chi\in \frac{3}{2}\mathrm{sum}_C[0, \beta_i-\beta_j]+\mathrm{sum}_E[0, \beta_i].\] 
Then \[\chi\in \frac{3}{2}\mathrm{sum}_D[0, \beta_i-\beta_j]+\mathrm{sum}_E[0, \beta_i].\] 
\end{prop}

\begin{proof}
Let $\chi$ be a strictly dominant weight and write it as 
\begin{equation}\label{chicij}
D: \chi=\sum_{1\leq j<i\leq d}c_{ij}(\beta_i-\beta_j)+\sum_{1\leq i\leq d} d_i\beta_i,
\end{equation}
with coefficients 
\begin{equation}\label{conditionsd}
    -\frac{3}{2}\leq c_{ij}\leq \frac{3}{2} \text{ for }j<i\text{ and }0\leq d_i\leq 1.
\end{equation}
Let $p_D:=(i'j')$ be the largest pair (in lexicographic order) such that $j'<i'$ and $c_{i'j'}\leq 0$. If there is no such pair, let $p_D=\emptyset$ and declare that $\emptyset<(ij)$ for any $1\leq j<i\leq d$. 
Consider the $L^2$ metric on the vector space  $M(d)_{\mathbb{R}}$.
The limit along a convergent sequence of decompositions $D_n$ with $p_{D_n}=(ij)$ is a decomposition $\widetilde{D}$ with $p_{\widetilde{D}}\geq (ij)$. 

Among all possible decompositions $D$ in \eqref{chicij} satisfying \eqref{conditionsd}, choose one with the smallest possible (in lexicographic order) associated pair $p_D$. Assume that $p_D$ is not $\emptyset$, otherwise the statement to be proved is true for $\chi$. Write $p_D=(ij)$.
Assume that $m$ is the supremum of $c_{ij}$ for all decompositions $D$ with $p_D=(ij)$. By taking the limit along a sequence $D_n$ with $p_{D_n}=(ij)$ and $\text{lim}_n\,c_{ij, n}=m$, 
there exists a decomposition $D$ with $p_{D}=(ij)$ and $c_{ij}=m$. We will be working with this decomposition $D$ from now on.
The weight $\chi$ is strictly dominant, so 
\begin{align}\label{ineq:c}
c_{ij}+\sum_{i<l}(-c_{li})+\sum_{i>l\neq j}c_{il}+d_i> -c_{ij}+\sum_{j<l\neq i}(-c_{lj})+\sum_{j>l}c_{jl}+d_j.
\end{align}
At least one of the following holds:
\begin{enumerate}
    \item there exists $i<l$ such that $-c_{li}>-c_{lj}$,
    \item there exists $j<l<i$ such that $c_{il}>-c_{lj}$,
    \item there exists $l<j$ such that $c_{il}>c_{jl}$,
    \item $d_i>d_j$.
\end{enumerate}
We explain that in each of these scenarios, there exists a decomposition $D'$ such that $p_{D'}<p_D$ or $p_{D'}=p_D$ and $c'_{ij}>c_{ij}$, where $c'_{ab}$ and $d'_a$ are the coefficients appearing in the decomposition \eqref{chicij} for $D'$.
The existence of $D'$ contradicts the choice of $D$.

In Case (1), we have that $c_{li}<\frac{3}{2}$ and $c_{lj}>-\frac{3}{2}$, so there exists $\varepsilon>0$ such that 
\begin{align}\label{equ:c'}
    -\frac{3}{2}\leq c'_{ij}, c'_{li}, c'_{lj}\leq\frac{3}{2},
    \end{align}
    where
\begin{align}\label{equ:c'2}
    c'_{ij}=c_{ij}+\varepsilon,\
    c'_{li}=c_{li}+\varepsilon,\
    c'_{lj}=c_{lj}-\varepsilon.
\end{align}
Consider the decomposition $D'$ with coefficients $c'_{ij}$, $c'_{li}$, $c'_{lj}$, $c'_{ab}=c_{ab}$ for all other pairs $1\leq b<a\leq d$, and $d'_a=d_a$ for all $1\leq a\leq d$. Then $D'$ is a decomposition of $\chi$ because 
\begin{multline}\label{equalityprime}
    c_{ij}(\beta_i-\beta_j)+c_{li}(\beta_l-\beta_i)+c_{lj}(\beta_l-\beta_j)=\\(c_{ij}+\varepsilon)(\beta_i-\beta_j)+(c_{li}+\varepsilon)(\beta_l-\beta_i)+(c_{lj}-\varepsilon)(\beta_l-\beta_j).
\end{multline}
By the choice of $(ij)$, we have that $c_{lj}, c_{li}>0$. We can thus choose $\varepsilon>0$ such that $c'_{li}, c'_{lj}>0$.
Then either $p_{D'}=p_D$ and $c'_{ij}>c_{ij}$ or $p_{D'}<p_D$.

In Case (2), we have that $c_{il}>-\frac{3}{2}$ and $c_{lj}>-\frac{3}{2}$, so there exists $\varepsilon>0$ such that 
\begin{align}\label{equ:c'3}
-\frac{3}{2}\leq c'_{ij}, c'_{il}, c'_{lj}\leq\frac{3}{2},
\end{align}
where 
\begin{align}\label{equ:c'4}
c'_{ij}=c_{ij}+\varepsilon, \ c'_{il}=c_{il}-\varepsilon, \  c'_{lj}=c_{lj}-\varepsilon.
\end{align}
Consider the decomposition $D'$ with coefficients $c'_{ij}$, $c'_{il}$, $c'_{lj}$, $c'_{ab}=c_{ab}$ for all other $1\leq b<a\leq d$ and $d'_a=d_a$ for $1\leq a\leq d$. By the choice of $(ij)$, we have that $c_{il}>0$. We can thus choose $\varepsilon>0$ such that $c'_{il}>0$. Then either $p_{D'}=p_D$ and $c'_{ij}>c_{ij}$ or $p_{D'}<p_D$.

In Case (3), we have that $c_{il}>-\frac{3}{2}$ and $c_{jl}<\frac{3}{2}$, so there exists $\varepsilon>0$ such that \[-\frac{3}{2}\leq c'_{ij}, c'_{il}, c'_{jl}\leq\frac{3}{2},\] 
where $c'_{ij}=c_{ij}+\varepsilon$, $c'_{il}=c_{il}-\varepsilon$, and $c'_{jl}=c_{jl}+\varepsilon$.
Consider the decomposition $D'$ with coefficients $c'_{ij}$, $c'_{il}$, $c'_{jl}$, and $c'_{ab}=c_{ab}$ for $1\leq b<a\leq d$ and $d'_a=d_a$ for $1\leq a\leq d$. Then either $p_{D'}=p_D$ and $c'_{ij}>c_{ij}$ or $p_{D'}<p_D$.

In Case (4), we have that $d_i>0$ and $d_j<1$,
so there exists $\varepsilon>0$ such that \[0\leq d'_i, d'_j\leq 1,\] where $d'_i=d_i-\varepsilon$, $d'_j=d_j+\varepsilon$.
Consider the decomposition $D'$ with coefficients $c'_{ij}=c_{ij}+\varepsilon$, $d'_i$, $d'_j$ and all the other coefficients are the same as those of $D$. 
\end{proof}

For a partition $(d_i)_{i=1}^k$ of $d$, denote by $\sigma_a:=\sum_{i=1}^a d_i$ for $1\leq a\leq k$. Define the set 
\begin{align}\label{def:Da}
    D_a :=\{(i, j) \mid 
    \,\sigma_a+1\leq j<i\leq \sigma_{a+1}\}. 
    \end{align}

\begin{prop}\label{10}
Let $\mu \in \mathbb{R}$ and 
$\delta=d\mu \tau_d$. 
Let $\chi\in M(d)_{\mathbb{R}}$ be a dominant weight such that $\chi+\rho+\delta\in \mathbf{V}(d)$ with associated sets $A=(d_i, w_i)_{i=1}^k$ and $A'=(d_i, v_i)_{i=1}^k$ as in \eqref{Aprime232}. Then 
    \begin{equation}\label{ineqq}
    -\mu \leq \frac{v_1}{d_1}<\cdots<
    \frac{v_k}{d_k}\leq 1-\mu.\end{equation}
\end{prop}
\begin{proof}
We use the decomposition \eqref{decompchibis}. 
By Proposition \ref{prop:dominanttree}, 
we have that \begin{equation*}
    \frac{v_1}{d_1}<\cdots<\frac{v_k}{d_k}.
\end{equation*}
Let $\lambda$ be the antidominant cocharacter associated to the decomposition $(d_i)_{i=1}^k$.
From the decompositions
 (\ref{decompchi}) and (\ref{wprime}), 
we have that
\begin{equation}\label{decompchi2}
\chi+\rho+\delta=
\frac{3}{2}\mathfrak{g}^{\lambda<0}+\sum_{i=1}^k v_i\tau_{d_i}+d\mu \tau_d+\sum_{i=1}^k \psi_i
\in \textbf{V}(d).
\end{equation}
The weight $\chi+\rho+\delta$ 
is strictly dominant, 
and the first three terms in (\ref{decompchi2}) are of the form 
$\sum_{i=1}^k x_i \tau_{d_i}$ for some $x_i \in \mathbb{R}$, 
hence 
each $\psi_i$ is also strictly dominant. 
Therefore, by 
Proposition \ref{dominantij} we have that 
\begin{align*}
\chi+\rho+\delta&\in \frac{3}{2}\text{sum}_D[0, \beta_i-\beta_j]+\text{sum}_E[0, \beta_i],\\
\psi_a&\in \frac{3}{2}\text{sum}_{D_a}[0, \beta_i-\beta_j]
\end{align*}
for all $1\leq a\leq k$. 
We can thus substract the weights $\psi_i$ in \eqref{decompchi2} and still have that
\begin{equation}\label{decompchi3}
    \frac{3}{2}\mathfrak{g}^{\lambda<0}+\sum_{i=1}^k v_i\tau_{d_i}+d\mu\tau_d\in \textbf{V}(d).
\end{equation}
Let $\nu \colon \mathbb{C}^{\ast} \to T(d)$ be the cocharacter acting with weight $-1$ on $\beta_i$ for $1\leq i\leq d_1$ and with weight zero on $\beta_i$ for $d_1<i\leq d$. 
The weight $\frac{3}{2}\mathfrak{g}^{\lambda<0}$ has maximal $\nu$-weight among weights in $\textbf{V}(d)$, so \eqref{decompchi3} implies that
\[\left\langle \nu, \frac{3}{2}\mathfrak{g}^{\lambda<0}+
\sum_{i=1}^k v_i\tau_{d_i}+
d\mu\tau_d\right\rangle=
\left\langle \nu, \frac{3}{2}\mathfrak{g}^{\lambda<0}+v_1\tau_{d_1}+d_1\mu\tau_{d_1}\right\rangle \leq \left\langle \nu, \frac{3}{2}\mathfrak{g}^{\lambda<0}\right\rangle,\] and thus we obtain that 
\begin{equation}\label{less1}
   \frac{v_1}{d_1}+\mu\geq 0. 
\end{equation}
Let $\nu' \colon \mathbb{C}^{\ast} \to T(d)$ 
be the cocharacter acting with weight $0$ on $\beta_i$ for $1\leq i\leq d-d_k$ and with weight $1$ on $\beta_i$ for $d-d_k<i\leq d$. 
The weight $\frac{3}{2}\mathfrak{g}^{\lambda<0}+d_k\tau_{d_k}$ has maximal $\nu'$-weight among weights in $\textbf{V}(d)$, so
from \eqref{decompchi3} we also have that 
\begin{align*}
\left\langle \nu', \frac{3}{2}\mathfrak{g}^{\lambda<0}+
\sum_{i=1}^k v_i\tau_{d_i}+
d\mu\tau_d\right\rangle &=
\left\langle \nu', \frac{3}{2}\mathfrak{g}^{\lambda<0}+v_k\tau_{d_k}+d_k\mu\tau_{d_k}\right\rangle \\
&\leq \left\langle \nu', \frac{3}{2}\mathfrak{g}^{\lambda<0}+d_k\tau_{d_k}\right\rangle,
\end{align*}
and thus we obtain that 
\begin{equation}\label{gre0}
   \frac{v_k}{d_k}+\mu\leq 1. 
\end{equation}
From \eqref{ineq}, \eqref{less1}, and \eqref{gre0}, we obtain the desired conclusion.
\end{proof}

We next discuss a converse to Proposition \ref{10}.

\begin{prop}\label{prop2}
    Let $(d_i, v_i)_{i=1}^k$ satisfy \eqref{ineqq} and let $\lambda$ be an antidominant cocharacter associated to $(d_i)_{i=1}^k$. Then
    \[\sum_{i=1}^k v_i\tau_{d_i}+d\mu\tau_d\in \mathrm{sum}_E[0, \beta_i].\] 
\end{prop}
\begin{proof}
For each $1\leq i\leq k$, the condition $(1-\mu) d_i\geq v_i\geq -\mu d_i$ implies that 
\begin{equation}\label{inclusioni1k}
v_i\tau_{d_i}+d_i\mu\tau_{d_i}\in \mathrm{sum}[0, \beta_i].
\end{equation}
We obtain the conclusion by summing \eqref{inclusioni1k} for $1\leq i\leq k$.
\end{proof}

\subsection{Comparison of partitions}\label{comppartitions}
In this subsection, we define an order on the set of partitions $S^d_w$. This induces an order $O$ of the summands appearing in the semiorthogonal decomposition of $\mathcal{DT}(d)$
from Theorem~\ref{MacMahonthm}. 
\medskip

Fix $d\in\mathbb{N}$. 
We define an order on the tuples $A=(d_i, w_i)_{i=1}^k$ such that $\sum_{i=1}^kd_i=d$. Let $w_A:=\sum_{i=1}^kw_i$. Recall the sets \[S^d_w\cong T^d_w\] of partitions of $(d,w)$, see Proposition \ref{ST}, and let $\mathcal{S}:=\bigcup_{w\in\mathbb{Z}}S^d_w$. We will define a set $O\subset \mathcal{S}\times\mathcal{S}$. 
For $w>w'$, let $O_{w, w'}$ be the set of all pairs $(A, B)$ with $A, B\in \mathcal{S}$ such that $w_A=w$ and $w_B=w'$. For $w<w'$, let $O_{w, w'}$ be the empty set.

We explain how to compare two partitions $A, B\in S^d_w$. 
The general procedure for comparing two such partitions for an arbitrary symmetric quiver is described in \cite[Subsection 3.3.4]{P}. Consider the path of partitions $T_A$ with coefficients $r_{\ell, A}$ as in \eqref{decompchi} corresponding to $A\in S^d_w$. 
Order the coefficients $r_{\ell, A}$ in decreasing order $r'_{1,A}>r'_{2,A}>\cdots>r'_{f(A),A}.$ Each $r'_{i, A}$ for $1\leq i\leq f(A)$ corresponds to a partition $\pi_{i,A}$.  
Similarly, consider the path of partitions $T_B$ with coefficients $r_{\ell, B}$ corresponding to $B\in S^d_w$. Define similarly $r'_{1,B}>\cdots>r'_{f(B),B}$ and $\pi_{i,B}$ for $1\leq i\leq f(B)$. 

Define the set $R\subset \mathcal{S}_w\times\mathcal{S}_w$ which contains pairs $(A,B)$ such that

\begin{itemize}
    \item there exists $n\geq 1$ such that $r'_{n, A}>r'_{n, B}$ and $r'_{i, A}=r'_{i, B}$ for $i<n$, or
    \item there exists $n\geq 1$ such that $r'_{i, A}=r'_{i, B}$ for $i\leq n$, $\pi_{i, A}=\pi_{i, B}$ for $i<n$, and $\pi_{n, B}\geq \pi_{n, A}$, see Subsection \ref{compa}, or
    \item are of the form $(A, A)$.
\end{itemize}
We then let $O_{w,w}:=\mathcal{S}_w\times\mathcal{S}_w\setminus R$ and 
\begin{equation}\label{def:setO}
O:=\bigcup_{w,w'\in\mathbb{Z}}O_{w, w'}.\end{equation}

We will only use that such an order exists in the current paper. Further, in order to make the above process more accessible, we explain how to compute $r'_{1,A}$ and $\pi_{1,A}$.
Assume that $A=(d_i, w_i)_{i=1}^k\in S^d_w$. 
By \cite[Subsection 3.1.1]{P}, 
for a dominant weight $\theta$, 
the $r$-invariant is equal to
\begin{equation}\label{def:rtheta}
r(\theta)=\text{max}_\lambda \frac{\langle \lambda, \theta\rangle}{\Big\langle \lambda, \left(\text{Hom}(V, V)^{\oplus 3}\right)^{\lambda>0}\Big\rangle},
\end{equation}
where $V$ is a $\mathbb{C}$-vector space of dimension $d$ and the maximum is taken over all dominant cocharacters $\lambda$ of $GL(d)$.
Let $\lambda$ be a dominant cocharacter attaining the maximum above and
assume the associated partition of $\lambda$ is $(e_i)_{i=1}^s$. Then the maximum in \eqref{def:rtheta} is also attained for the cocharacter $\lambda_a$ with associated partition $\left(\sum_{i\leq a} e_i, \sum_{i>a} e_i\right)$, see \cite[Proposition 3.2]{P}. 

Let $\theta=\chi_A+\rho:=\sum_{i=1}^k w_i\tau_{d_i}+\rho$.
We have that $\underline{d}\geq \underline{e}$, see Subsection \ref{compa} for the notation and \eqref{decompchi}.
Let $\lambda_a$ be a cocharacter which acts with weight $\sum_{i\leq a} d_i$ on $\beta_j$ for $j>a$ and weight $-\sum_{i>a} d_i$ on $\beta_j$ for $j\leq a$. 
We have that
\[r(\chi_A+\rho)=\text{max}_a \frac{\langle \lambda_a, \chi_A+\rho\rangle}{\Big\langle \lambda_a, \left(\text{Hom}(V, V)^{\oplus 3}\right)^{\lambda>0}\Big\rangle},\] where the maximum is taken after all $1\leq a <k$.
We compute
\begin{multline}
\Big\langle \lambda_a, \text{Hom}(V, V)^{\lambda_a>0}\Big\rangle=
\Big\langle \lambda_a, \sum_{j>a\geq i}\beta_j-\beta_i\Big\rangle=\\
\left(\sum_{i>a} d_i\right)^2\left(\sum_{i\leq a}d_i\right)+\left(\sum_{i\leq a} d_i\right)^2\left(\sum_{i>a}d_i\right)
=d\left(\sum_{i>a} d_i\right)\left(\sum_{i\leq a}d_i\right).
\end{multline}
Then
\[r(\chi_A+\rho)=r'_{1,A}=\text{max}_a\left(\frac{
\left(\sum_{i\leq a}d_i\right)\left(\sum_{i>a}w_i\right)-
\left(\sum_{i> a}d_i\right)\left(\sum_{i\leq a}w_i\right)}{3d\left(\sum_{i>a} d_i\right)\left(\sum_{i\leq a}d_i\right)}+\frac{1}{6}\right),\]
where the maximum is taken after $1\leq a\leq k$.
The $r$-invariant can then we expressed as a linear function (with positive coefficients) of
\begin{equation}\label{def:rho}
   \rho_A:=\text{max}_a\left(\frac{\sum_{i>a} w_i}{\sum_{i>a} d_i}-\frac{\sum_{i\leq a} w_i}{\sum_{i\leq a} d_i}\right). 
\end{equation}
The partition $\pi_{1,A}$ can be reconstructed from all $1\leq a<k$ for which $\lambda_a$ attains the maximum of \eqref{def:rtheta}, alternatively from all $1\leq a <k$ which attain the maximum in \eqref{def:rho}.
Assume the set of all such $1\leq a < k$ is $1\leq a_2<\cdots<a_s< k$. Then $\pi_{1,A}=(e_i)_{i=1}^s$ is the partition of $d$ with terms:
\[(d_1+\ldots+d_{a_2},\ldots, d_{a_s+1}+\ldots+d_k).\]

\subsection{Semiorthogonal decompositions}

We discuss some preliminary results
needed in the proof of Theorem~\ref{MacMahonthm}:

\begin{prop}\label{prop1}
Let $\mu\in\mathbb{R}$ and let $\delta:=d\mu\tau_d\in M(d)_{\mathbb{R}}$.
Consider 
$A=(d_i, w_i)_{i=1}^k$ 
with $\sum_{i=1}^k d_i=d$ and
such that its associated $A'=(d_i, v_i)_{i=1}^k$, see Subsection \ref{prime}, satisfies \begin{equation}\label{i0}
   -\mu\leq \frac{v_1}{d_1}<\cdots<\frac{v_k}{d_k}\leq 1-\mu. 
   \end{equation}
Then the categorical Hall product (\ref{prel:hall}) restricts to 
the fully-faithful functor
\begin{align}\label{ast:MD}
   \ast \colon 
   \mathbb{M}(d_1)_{w_1} \boxtimes \cdots 
    \boxtimes \mathbb{M}(d_k)_{w_k} \to 
    \mathbb{D}(d; \delta)
    \end{align}
such that 
there is a semiorthogonal decomposition 
\begin{equation}\label{sodD}
\mathbb{D}(d; \delta)=\Big\langle \boxtimes_{i=1}^k \mathbb{M}(d_i)_{w_i}\Big\rangle.
\end{equation}
Here
the right hand side consists of all tuples $A=(d_i, w_i)_{i=1}^k$ with $\sum_{i=1}^k d_i=d$
satisfying \eqref{i0}. 
The order of the semiorthogonal decomposition is given the set $O$ 
in (\ref{def:setO}) as in Theorem~\ref{MacMahonthm}. 
\end{prop}
We divide the proof into three steps: 
\begin{step}
The categorical Hall product (\ref{prel:hall})
restricts to the functor (\ref{ast:MD}). 
\end{step}
\begin{proof}
Let $A=(d_i, w_i)_{i=1}^k$ be as above and let $\lambda$ be the corresponding antidominant cocharacter of $(d_i)_{i=1}^k$.
For $1\leq i\leq k$, consider a weight $\chi_i\in M(d_i)$ such that 
$\chi_i+\rho_i\in \textbf{W}(d_i)_{w_i}$. 
Write $\chi_i+\rho_i=\psi_i+w_i\tau_{d_i}$ with
$\psi_i \in \textbf{W}(d_i)_{0}$
for $1\leq i\leq k$. 
In order to show that 
the categorical Hall product (\ref{prel:hall})
restricts to the functor (\ref{ast:MD}), 
it suffices to check that
\begin{align*}
    \ast_{i=1}^k\left( \mathcal{O}_{\mathcal{X}(d_i)} \otimes \Gamma_{GL(d_i)}(\chi_i)\right)
    =p_{\lambda *}q_\lambda^*\left(\mathcal{O}_{\X(d)^\lambda}\otimes\Gamma_{L}\left(\sum_{i=1}^k \chi_i\right)\right)
    \in \mathbb{D}(d; \delta),
\end{align*}
where $L:=\times_{i=1}^k GL(d_i)$ and the 
first identity follows from the definition of the
categorical Hall product. 
For a set 
\begin{align}\label{set:J}
J\subset \{\beta\text{ wt of }R(d) \mid 
\langle \lambda, \beta\rangle<0\},
\end{align}
let $\sigma_J:=\sum_J \beta$.
By Proposition \ref{bbw}, it suffices to show that for all subsets $J$ as above, we have that
\[\left(\sum_{i=1}^k\chi_i-\sigma_J\right)^++\rho+d\mu\tau_d\in \textbf{V}(d).\]
The weight $d\mu\tau_d$ is Weyl invariant, so it suffices to check that
\begin{equation}\label{incl}
    \sum_{i=1}^k\chi_i-\sigma_J+\rho+d\mu\tau_d\in \textbf{V}(d).
\end{equation}
Write 
\begin{align*}
    \sum_{i=1}^k\chi_i+\rho+d\mu \tau_d &=
\sum_{i=1}^k w_i\tau_{d_i}+\sum_{i=1}^k(\psi_i
-\rho_i)+\rho+d\mu \tau_d \\
&=
-\frac{3}{2}\mathfrak{g}^{\lambda>0}+\sum_{i=1}^k v_i\tau_{d_i}+\sum_{i=1}^k(\psi_i
+d_i \mu \tau_{d_i}).
\end{align*}
By Proposition \ref{prop2}, we have that
\[\sum_{i=1}^k v_i\tau_{d_i}+d\mu\tau_d\in 
\text{sum}_E[0, \beta_i].\]
Further, write
$-\mathfrak{g}^{\lambda>0}=\sum_{D'}(\beta_i-\beta_j)$
where $D'=D\setminus \bigcup_{a=1}^k D_a$, 
see (\ref{def:CDE}), (\ref{def:Da}) for the definition of $D$ and $D_a$.
Then for any partial sum $\sigma_J$ as above, we have that
\[-\frac{3}{2}\mathfrak{g}^{\lambda>0}+\sum_{i=1}^k v_i\tau_{d_i}+d\mu\tau_d-\sigma_J\in \frac{3}{2}\text{sum}_{D'}[\beta_j-\beta_i, \beta_i-\beta_j]+\text{sum}_E[0, \beta_i].\]
Then 
\[-\frac{3}{2}\mathfrak{g}^{\lambda>0}+\sum_{i=1}^k v_i\tau_{d_i}+d\mu\tau_d-\sigma_J+\sum_{i=1}^k \psi_i\in \textbf{V}(d),\]
and thus \eqref{incl} holds. 
\end{proof}

\begin{step}
The functor (\ref{ast:MD}) is fully-faithful 
and the essential images of such functors 
for partitions satisfying (\ref{i0})
are 
semiorthogonal. 
\end{step}
\begin{proof}
The semiorthogonality and fully-faithfulness 
follow from \cite[Proposition 4.3]{P}, \cite[Corollary 3.3]{P2}. 
\end{proof}

\begin{step}
The essential images of (\ref{ast:MD}) 
for partitions satisfying (\ref{i0})
generate $\mathbb{D}(d;\mu)$. 
\end{step}
\begin{proof}
Let $\chi$ be a dominant weight in $M(d)$ such that $\chi+\rho+\delta\in \textbf{V}(d)$. It suffices to show that 
the object 
\begin{align}\label{obj:D}
\mathcal{O}_{\X(d)}\otimes \Gamma_{GL(d)}(\chi)
\in \mathbb{D}(d; \delta)
\end{align}
is generated by the categories on the right hand side
of (\ref{sodD}). 
Following the proof in \cite[Proposition 4.1]{P} (which goes back to the proof of \cite[Theorem 1.1.2]{SVdB2}), 
we use induction on the $(r,p)$-invariant of $\chi+\rho+\delta$ with respect to the polytope $\textbf{W}(d)$. 
The base of the induction is $r\leq 1/2$, where 
(\ref{obj:D}) is an object of 
$\mathbb{M}(d)_{w}$ with $-\mu\leq w/d \leq 1-\mu$, in particular
generated by objects from the right hand side of (\ref{sodD}). 

By the discussion in Subsection \ref{dectree2}, there exists a path $T$ of partitions with associated decomposition $(d_i)_{i=1}^k$ and antidominant cocharacter $\lambda$
such that $r_\ell>\frac{1}{2}$ for all $\ell\in T$ and 
\[\chi+\rho+\delta=-\sum_{\ell\in T}r_\ell\left(3\mathfrak{g}_\ell\right)+\sum_{i=1}^k\psi_i+(w+d\mu)\tau_d,\]
where $w=\langle 1_{d}, \chi \rangle$. 
Write \begin{equation}\label{decompositionchi}
    \chi=\sum_{i=1}^k \chi_i\text{ with } 
    \chi_a \in M(d_a)\text{ for }1\leq a\leq k,
\end{equation} 
where $M(d_a)$ is 
generated by the weights $\beta_i$ for $\sigma_a<i\leq \sigma_{a+1}$. Let $w_i=\langle 1_{d_i}, \chi_i\rangle$ and consider the partition $A=(d_i, w_i)_{i=1}^k$ of $(d, w)$. 
By (\ref{chi:i}),
 we have $\chi_i+\rho_i \in \textbf{W}(d_i)_{w_i}$. 
 Therefore we have the object
 \begin{align*}
    \mathcal{O}_{\X(d)^\lambda}\otimes \Gamma_{L}\left(\sum_{i=1}^k\chi_i\right)\in \boxtimes_{i=1}^k \mathbb{M}(d_i)_{w_i}
    \end{align*}
    where $L:=\times_{i=1}^k GL(d_i)$. There is a 
    natural morphism 
\begin{align}\label{mor:C}
\mathcal{O}_{\X(d)}\otimes \Gamma_{GL(d)}(\chi)\to p_{\lambda*}q_\lambda^*\left(\mathcal{O}_{\X(d)^\lambda}\otimes \Gamma_{L}\left(\sum_{i=1}^k\chi_i\right)\right).\end{align}
Let $C$ be its cone.
By 
 Proposition~\ref{10}, 
the associated partition $A'=(d_i, v_i)$
in (\ref{trans:A}) satisfies (\ref{i0}). 
Therefore the right hand side in (\ref{mor:C}) 
is an object in $\mathbb{D}(d; \delta)_w$
by Step~1, 
hence
the cone $C$ is also an 
object 
in $\mathbb{D}(d; \delta)_w$. Further,
the argument of Step~1 
shows that 
the cone $C$ is generated by the vector bundles   $\mathcal{O}_{\X(d)}\otimes
\Gamma_{GL(d)}
\left(\left(\chi-\sigma_J\right)^+\right)$,
where $\sigma_J$ is a sum of weights in a non-empty subset
$J$ in (\ref{set:J}). 
By Proposition \ref{rgoesdown}, the $(r,p)$-invariants of 
the weights 
$(\chi-\sigma_J)^{+}$ of these vector bundles is strictly less than the $(r,p)$-invariant of $\chi$. Repeating this process, we obtain that indeed $\mathcal{O}_{\X(d)}\otimes \Gamma_{GL(d)}(\chi)$ is generated by the categories on the right hand side of (\ref{sodD}).
\end{proof}

Let $\mathcal{X}^f(d)$
be the moduli stack for the framed quiver 
$Q^f$
as in (\ref{def:Xd}), and consider the projection
$a \colon \mathcal{X}^f(d) \to \mathcal{X}(d)$. 
We have the following proposition: 

\begin{prop}\label{fullyfaithful}
    Let $(d, w)\in \mathbb{N}\times\mathbb{Z}$.
    The functor $a^{\ast} \colon D^b(\X(d))_w\to D^b\left(\X^f(d)\right)$ is fully faithful.
    Moreover, for $\mathcal{E}_i \in D^b(\mathcal{X}(d))_{w_i}$
    with $i=1, 2$ and $w_1>w_2$, we have 
    $\Hom(a^{\ast}\mathcal{E}_1, a^{\ast}\mathcal{E}_2)=0$.
\end{prop}

\begin{proof}
    There is a push-forward 
    functor $a_{\ast} \colon D^b\left(\X^f(d)\right)\to D_{\text{qcoh}}(\X(d))$. 
    We take
    $\mathcal{E}_i \in D^b(\mathcal{X}(d))_{w_i}$
    for $i=1, 2$ and $w_1 \geq w_2$.
    Using adjunction and the projection formula, we have that
        \[\Hom_{\X^f(d)}(a^*\mathcal{E}_1, a^*\mathcal{E}_2)=\Hom_{\X(d)}(
        \mathcal{E}_1, a_{\ast}a^{\ast}\mathcal{E}_2)=\Hom_{\X(d)}\left(\mathcal{E}_1, 
        \mathcal{E}_2 \otimes a_{\ast}\mathcal{O}_{\X^f(d)}\right).\]
    The complex $a_*\mathcal{O}_{\X^f(d)}$ splits as 
    \[a_*\mathcal{O}_{\X^f(d)}=
    \mathrm{Sym}^{\bullet}(\mathcal{V}^{\vee})=
    \mathcal{O}_{\X(d)}\oplus B, \] where
    $\mathcal{V} \to \mathcal{X}(d)$
    is the vector bundle associated with
    the $GL(V)$-representation $V$, 
    and 
    $B:=\mathrm{Sym}^{>0}(\mathcal{V}^{\vee})$ is a quasi-coherent sheaf on $\X(d)$ on which 
    the diagonal 
    cocharacter $1_d$ of $T(d)$ acts with strictly negative weights. Then 
    \[\Hom_{\X(d)}
    \left(\mathcal{E}_1, 
        \mathcal{E}_2 \otimes a_{\ast}\mathcal{O}_{\X^f(d)}\right)=
        \Hom_{\X(d)}\left(\mathcal{E}_1, \mathcal{E}_2\right),\] and thus the 
        both conclusions follow.
\end{proof}

Recall the subcategory 
$\mathbb{E}(d; \delta) \subset D^b\left(\mathcal{X}^f(d)\right)$
defined in (\ref{subcat:D}). We have the following corollary: 

\begin{cor}\label{cor1}
    Let $\mu\in\mathbb{R}$ and let $\delta:=d\mu\tau_d\in M(d)_{\mathbb{R}}$.
There is a semiorthogonal decomposition 
\[\mathbb{E}(d; \delta)=\Big\langle a^*\left(\boxtimes_{i=1}^k \mathbb{M}(d_i)_{w_i}\right)\Big\rangle,\] where the right hand side consists of all tuples $A=(d_i, w_i)_{i=1}^k$ with $\sum_{i=1}^k d_i=d$ such that its associated $A'=(d_i, v_i)_{i=1}^k$ satisfies \begin{equation}\label{i1}
    -\mu\leq \frac{v_1}{d_1}<\cdots<\frac{v_k}{d_k}\leq 1-\mu. 
   \end{equation} Further, there exist equivalences $\boxtimes_{i=1}^k \mathbb{M}(d_i)_{w_i}
   \stackrel{\sim}{\to} a^*\left(\boxtimes_{i=1}^k \mathbb{M}(d_i)_{w_i}\right)$.
\end{cor}

\begin{proof}
For $w \in \mathbb{Z}$, let 
$\mathbb{D}(d; \delta)_{w} \subset D^b(\mathcal{X}(d))_w$
be the subcategory defined similarly to $\mathbb{D}(d; \delta)$
for a fixed $w$. 
Note that $\mathbb{D}(d; \delta)$
is a direct sum of $\mathbb{D}(d; \delta)_{w}$
for $w \in [-d\mu, d-d\mu] \cap \mathbb{Z}$. 
From Proposition~\ref{fullyfaithful},
we have the semiorthogonal decomposition 
\begin{align}\label{sod:E}
    \mathbb{E}(d; \delta)=\langle a^{\ast}\mathbb{D}(d; \delta)_{p}, 
    a^{\ast}\mathbb{D}(d; \delta)_{p+1}, \cdots, a^{\ast}\mathbb{D}(d; \delta)_q \rangle
\end{align}
where $[-d\mu, d-d\mu] \cap \mathbb{Z}=\{p, p+1, \ldots, q\}$, 
and $a^{\ast} \colon 
\mathbb{D}(d; \delta)_w \stackrel{\sim}{\to} a^{\ast}\mathbb{D}(d; \delta)_w$
is an equivalence. 
By Proposition~\ref{prop1}, we have the semiorthogonal decomposition 
\begin{align*}
    \mathbb{D}(d; \delta)_w=\left\langle 
    \boxtimes_{i=1}^k \mathbb{M}(d_i)_{w_i} \right\rangle
\end{align*}
where the partitions $(d_i, w_i)_{i=1}^k$
of $(d, w)$ satisfy (\ref{i0}). 
Therefore from (\ref{sod:E}), we obtain the corollary.
\end{proof}

\subsection{Proof of Theorem~\ref{MacMahonthm}}
In this subsection, we give a proof of Theorem~\ref{MacMahonthm}. 
We first prepare some notation. 
We denote by $Q'$
the quiver obtained by adding 
an edge $\overline{e}$
to the framed quiver $Q^f$, 
where $\overline{e}$ is the opposite 
edge to $e$. 
Let $V$ be a $d$-dimensional vector space.
The moduli stack of representations of $Q'$ with dimension 
vector $(1, d)$ is given by 
\begin{align}\label{def:Xprime}
\X'(d) :=R'(d)/GL(d)=(V \oplus V^{\vee} \oplus \mathfrak{gl}(V)^{\oplus 3})/GL(V). 
\end{align}
Note that $R'(d)$ is a symmetric $GL(d)$-representation. For a cocharacter $\lambda \colon \mathbb{C}^{\ast} \to T(d)$, 
we set 
\begin{align}\label{def:vlambda}
    \eta_{\lambda}:=\langle \lambda, 
    (\mathcal{X}'(d)^{\vee})^{\lambda>0} \rangle \in \mathbb{Z}.
\end{align}
We denote by $\nabla \subset M_{\mathbb{R}}$
the polytope defined by 
\begin{align*}
    \nabla :=\left\{\chi \in M_{\mathbb{R}} \relmiddle|
    -\frac{1}{2} \eta_{\lambda} \leq 
    \langle \lambda, \chi \rangle \leq \frac{1}{2} \eta_{\lambda} \mbox{ for all }
    \lambda \colon \mathbb{C}^{\ast} \to T(d) \right\}. 
\end{align*}
\begin{lemma}\label{lem:nabla}
For a dominant character $\chi \in M^+$, 
we have $\chi+\rho+\delta \in \mathbf{V}(d)$
if and only if for any $T(d)$-weight $\chi'$
of $\Gamma_{GL(d)}(\chi)$, we have 
\begin{align*}
\chi'+\delta -\frac{1}{2}
d\tau_d \in \nabla. 
\end{align*}
\end{lemma}
\begin{proof}
The condition $\chi+\rho+\delta \in \mathbf{V}(d)$
is equivalent to 
\begin{align*}
    \chi+\rho+\delta-\frac{1}{2}d\tau_d
    \in \frac{3}{2} \mathrm{sum}[0, \beta_i-\beta_j]
    +\frac{1}{2}\mathrm{sum}[-\beta_k, \beta_k]
\end{align*}
where the sum is after all $1\leq i, j, k \leq d$. 
The right hand side is half of the convex hull of the
$T(d)$-weights of $\wedge^{\bullet}(R'(d))$, 
so the lemma follows from~\cite[Lemma~2.9]{hls}. 
\end{proof}

The above lemma is used 
to show the following proposition, 
which claims that 
the category $\mathbb{E}(d; \delta)$
defined in (\ref{subcat:D}) gives a window 
subcategory for 
non-commutative Hilbert schemes: 
\begin{prop}\label{magic}
    Denote by $\iota \colon \mathrm{NHilb}(d)= \X^f(d)^{\text{ss}}\hookrightarrow 
    \X^f(d)$ the open immersion. 
    Let 
    $\mu \in \mathbb{R}$
    with $2\mu k \notin \mathbb{Z}$
    for $1\leq k \leq d$
    and let $\delta:=\mu d\tau_d$.
    Then the following composition functor is 
    an equivalence: 
    \begin{align}\label{iota:ast}
    \iota^{\ast}\colon 
    \mathbb{E}(d; \delta)
    \hookrightarrow D^b(\mathcal{X}^f(d)) \stackrel{\iota^{\ast}}{\twoheadrightarrow} 
    D^b(\mathrm{NHilb}(d)).
    \end{align}
\end{prop}
\begin{proof}
The proposition 
is proved using the magic 
window Theorem~\cite{hls}
by reducing the problem to the case of 
symmetric representations (see
the proof of~\cite[Proposition~2.6]{KoTo}
for a similar argument). 
Fix an isomorphism between the weight and coweight spaces $M_\mathbb{R}\cong N_\mathbb{R}$ and consider the norm on $M_\mathbb{R}\cong N_\mathbb{R}$ defined by
\begin{equation}\label{definition:norm}
    \left|\sum_{i=1}^d c_i\beta_i\right|:=\sqrt{\sum_{i=1}^d c_i^2}.
\end{equation}
Let 
\begin{align}\label{KN:S}
\X^f(d)=\bigsqcup_{i\in I} \mathcal{S}_i \sqcup \X^f(d)^{\text{ss}}
\end{align}
be the Kempf-Ness stratification of 
$\X^f(d)$
with respect to the determinant 
character $\det \colon GL(d) \to \mathbb{C}^{\ast}$
and the above choice of norm on $N_{\mathbb{R}}$
(see~\cite[Section~2.1]{halp}). 
Let $\lambda_i \colon \mathbb{C}^{\ast} \to T(d)$
be the cocharacter associated with $\mathcal{S}_i$. 
In~\cite[Lemma~5.1.9]{T}, 
it is proved that
each 
$\lambda_i$
is of the form 
\begin{align}\label{lambdai:form}
\lambda_i(t)=
    (\overbrace{t^{-1}, \ldots, t^{-1}}^k, 1, \ldots, 1)
    \end{align}
    for some $1\leq k \leq d$. 

For each $i\in I$, consider the fixed stack $\mathcal{Z}_i:=\mathcal{S}_i^{\lambda_i}$ and the integer
\begin{align*}
    \eta_i :=\langle \lambda_i, (\mathcal{X}^f(d)^{\vee})^{\lambda_i>0} \rangle=
    -\langle \lambda_i, \mathcal{X}^f(d)^{\lambda_i<0} \rangle=k+3k(d-k).  
\end{align*}
Let $\mathbb{G}_{\eta}\subset D^b\left(\X^f(d)\right)$ be the subcategory of complexes $\mathcal{F}$ such that \begin{align}\label{wt:condG}
    \mathrm{wt}_{\lambda_i}(\mathcal{F}|_{\mathcal{Z}_i})
    +\left\langle \lambda_i, \delta-\frac{1}{2}d\tau_d \right\rangle \subset 
    \left[-\frac{1}{2}\eta_i, \frac{1}{2}\eta_i \right)
\end{align}
for all $i\in I$.
By \cite[Theorem 2.10]{halp}, the restriction $\iota^*$ induces an equivalence
\begin{equation}\label{iotaGv}
\iota^* \colon \mathbb{G}_{\eta}\xrightarrow{\sim} D^b(\text{NHilb}(d)).
\end{equation}
On the other hand, 
as $\lambda_i$ is of the form (\ref{lambdai:form}), 
we have that
$\X^f(d)^{\lambda_i<0}=\X'(d)^{\lambda_i<0}$,
where $\X'(d)$
is the stack (\ref{def:Xprime}). 
Therefore we have $\eta_i=\eta_{\lambda_i}$, 
where the latter is defined in (\ref{def:vlambda}). 
Moreover, by the assumption $2k\mu \notin \mathbb{Z}$
for $1\leq k\leq d$
and $\langle \lambda_i, \delta \rangle=-\mu k$, 
$\langle \lambda_i, d\tau_d\rangle=-k$, 
the condition (\ref{wt:condG}) is equivalent to 
\begin{align}\notag
    \mathrm{wt}_{\lambda_i}(\mathcal{F}|_{\mathcal{Z}_i})
    +\left\langle \lambda_i, \delta-\frac{1}{2}d\tau_d \right\rangle \subset 
    \left[-\frac{1}{2}\eta_i, \frac{1}{2}\eta_i \right]. 
\end{align}
By Lemma~\ref{lem:nabla}, 
it follows that 
$\mathbb{E}(d; \delta)\subset \mathbb{G}_{\eta}$.
It suffices to show that 
the composition functor (\ref{iota:ast})
is essentially surjective. 

Consider the open subset
\begin{align*}
j \colon \X'(d)^{\text{ss}}
\hookrightarrow \mathcal{X}'(d)
\end{align*} of stable points in $\X'(d)$ with respect to the linearization $\det \colon GL(d)\to \mathbb{C}^*$. The stack $\X'(d)^{\text{ss}}$ is a smooth quasi-projective variety, 
which parametrizes tuples
\begin{align*}
(v, v', X, Y, Z), \ X, Y, Z \in \mathrm{End}(V), \ v \in V, \ v'  \in V^{\vee}   
\end{align*}
where $V$ is generated by $v$ under the action of $(X, Y, Z)$. 
In particular,
there is no constraint on $v'$, and so 
the projection $\mathcal{X}'(d) \to \mathcal{X}^f(d)$
restricts to 
the morphism 
$b \colon \mathcal{X}'(d)^{\text{ss}} \to \mathcal{X}^f(d)^{\text{ss}}$
which is an affine space bundle with fiber $V^{\vee}$. 
Let $\iota$ be the zero-section of the above affine space bundle. 
Then the pull-back 
\begin{align}\label{iota:pull}
    \iota^{\ast} \colon D^b(\mathcal{X}'(d)) \to D^b(\mathcal{X}^f(d)) 
\end{align}
is essentially surjective, as any object 
$E \in D^b(\mathcal{X}^f(d))$
is isomorphic to $\iota^{\ast}b^{\ast}E$. 

Similarly to (\ref{subcat:D}), 
let $\mathbb{F}(d; \delta) \subset D^b(\X'(d))$
be the subcategory generated by $\mathcal{O}_{\X'(d)} \otimes \Gamma_{GL(d)}(\chi)$
for a dominant weight $\chi$ satisfying $\chi+\rho+\delta \in \textbf{V}(d)$. 
As $Q'$ is symmetric,
by \cite[Theorem 3.2]{hls} there exists an equivalence via the restriction map:
\begin{equation}\label{magicdouble}
j^{\ast}\colon \mathbb{F}(d; \delta)\xrightarrow{\sim} D^b\left(\X'(d)^{\text{ss}}\right).
\end{equation}
The claim that (\ref{iota:ast}) is essentially 
surjective follows from the equivalence \eqref{magicdouble} and 
the essential surjectivity of (\ref{iota:pull}). 
\end{proof}

The last ingredient needed to prove Theorem \ref{MacMahonthm} is a Thom-Sebastiani theorem for quasi-BPS categories. For $d\in \mathbb{N}$, denote by $\X(d)_0$ the zero locus of $\mathrm{Tr}\,W\colon \X(d)\to\mathbb{C}$.

\begin{prop}\label{ThomSebastiani}
    Let $(d_j)_{j=1}^s\in \mathbb{N}^s$ and let $(w_j)_{j=1}^s\in\mathbb{Z}^s$.
    Let $\bullet\in\{\emptyset, \mathrm{gr}\}$.
    Then the inclusion $\iota\colon\times_{j=1}^s \X(d_j)_0\hookrightarrow \X(d)_0$ induces an equivalence
    \[\iota_*\colon \boxtimes_{j=1}^s\mathbb{S}^\bullet(d_j)_{w_j}\xrightarrow{\sim} \mathrm{MF}^\bullet\left(\boxtimes_{j=1}^s \mathbb{M}(d_j)_{w_j}, \mathrm{Tr}\,W\right).\]
    The product on the left hand side is the product of dg-categories over $\mathbb{C}(\!(\beta)\!)$ for $\beta$ of homological degree $-2$ in the ungraded case, and is the product of dg-categories over $\mathbb{C}$ in the graded case. 
\end{prop}

\begin{proof}
There is a semiorthogonal decomposition
\begin{equation}\label{SODprop3141}
D^b(\X(d))=\big\langle \boxtimes_{i=1}^k \mathbb{M}(d_i)_{w_i}\big\rangle,
\end{equation}
where the right hand side is after all partitions $(d_i)_{i=1}^k$ of $d$ and all $(w_i)_{i=1}^k\in\mathbb{Z}^k$ such that, for $v_i$ as in \eqref{w:prime}, we have 
\[\frac{v_1}{d_1}<\ldots<\frac{v_k}{d_k},\] and where the product is the product of dg-categories over $\mathbb{C}$, see \cite[Theorem 1.1]{P2}, \cite[Corollary 3.3]{P0}. 
By Proposition \ref{propsodpotential}, there is a semiorthogonal decomposition:
\begin{equation}\label{SODprop3142}
\mathrm{MF}^\bullet(\X(d), \mathrm{Tr}\,W)=\Big\langle \mathrm{MF}^\bullet\left(\boxtimes_{i=1}^k \mathbb{M}(d_i)_{w_i}, \mathrm{Tr}\,W\right)\Big\rangle.
\end{equation}
There is a Thom-Sebastiani equivalence 
    \[\iota_*\colon \boxtimes_{i=1}^k \mathrm{MF}^\bullet(\X(d_i), \mathrm{Tr}\,W_{d_i})\xrightarrow{\sim} \mathrm{MF}^\bullet(\X(d), \mathrm{Tr}\,W),\] 
    where the product on the left hand side is the product of dg-categories over $\mathbb{C}(\!(\beta)\!)$ for $\beta$ of homological degree $-2$ in the ungraded case,
    see \cite[Theorem 4.1.3]{Preygel}, or the product of dg-categories over $\mathbb{C}$ in the graded case, see \cite[Corollary 5.18]{MR3270588} (alternatively in the graded case, one can use the Koszul equivalence \eqref{equiv:Phi}).
    The claim follows using induction on $d$ and comparing the product of the semiorthogonal decompositions \eqref{SODprop3142} for $(d_j)_{j=1}^s$ and the semiorthogonal decomposition obtained using Proposition \ref{propsodpotential} from the product of the semiorthogonal decompositions \eqref{SODprop3141} for $(d_j)_{j=1}^s$.
\end{proof}

We finally give a proof of Theorem~\ref{MacMahonthm}: 

\begin{proof}[Proof of Theorem \ref{MacMahonthm}]
Using Corollary \ref{cor1} and Proposition \ref{magic}, there is a semiorthogonal decomposition
\begin{equation}\label{SODnhilb}
    D^b\left(\text{NHilb}(d)\right) =
\Big\langle \boxtimes_{i=1}^k \mathbb{M}(d_i)_{w_i}\Big\rangle,
\end{equation}
 where the right hand side consists of all tuples
 $A=(d_i, w_i)_{i=1}^k$ with $\sum_{i=1}^k d_i=d$ such that its associated $A'=(d_i, v_i)_{i=1}^k$ satisfies \eqref{i}.
By Proposition \ref{propsodpotential}, there is a semiorthogonal decomposition
\[\mathcal{DT}(d)=\Big\langle \mathrm{MF}\left(\boxtimes_{i=1}^k \mathbb{M}(d_i)_{w_i}, \oplus_{i=1}^k \mathrm{Tr}\,W_{d_i}\right)\Big\rangle.\]
The claim follows from Proposition \ref{ThomSebastiani}.
\end{proof}

\section{K-theoretic equivariant Donaldson-Thomas invariants of \texorpdfstring{$\mathbb{C}^3$}{C3}}\label{Section:KDT}

In this section, we 
prove Theorem~\ref{thm:intro2} and Corollary~\ref{intro:cor1}. 

\subsection{Moduli stacks of zero-dimensional sheaves}\label{subsection:Hallcomm}
Let $V$ be a $d$-dimensional vector space 
and let $\mathfrak{g}=\Hom(V, V)$ be the Lie algebra of $GL(d)$. 
We set 
\begin{align*}
	\mathcal{Y}(d):=\mathfrak{g}^{\oplus 2}/GL(d),
	\end{align*}
where $GL(d)$ acts on $\mathfrak{g}$ by conjugation. 
The stack $\mathcal{Y}(d)$ is the moduli stack of 
representations of a two-loop quiver with dimension vector $d$. 
Let $s$ be the morphism 
\begin{align}\label{mor:s}
s \colon \mathfrak{g}^{\oplus 2} \to \mathfrak{g}, \ (X, Y) \mapsto [X, Y].
\end{align}
The map $s$ induces a map of vector bundles $\partial: \mathfrak{g}^{\vee}\otimes\mathcal{O}_{\mathfrak{g}^{\oplus 2}}\to \mathcal{O}_{\mathfrak{g}^{\oplus 2}}$.
Let $s^{-1}(0)$ be the derived scheme with the dg-ring of regular functions
\begin{align}\label{diff:ds2}
\mathcal{O}_{s^{-1}(0)}=\mathcal{O}_{\mathfrak{g}^{\oplus 2}}\left[\mathfrak{g}^{\vee}\otimes\mathcal{O}_{\mathfrak{g}^{\oplus 2}}[1]; d_s\right],
\end{align}
where the differential $d_s$ is induced by the map $\partial$. Consider the (derived) stack 
\[\mathcal{C}(d):=s^{-1}(0)/GL(d) \hookrightarrow \mathcal{Y}(d).\]
For a smooth variety $X$, we denote by $\mathfrak{C}oh(X, d)$
the derived moduli stack of zero-dimensional sheaves on $X$
with length $d$, 
and by $\mathcal{C}oh(X, d)$ its classical truncation. 
Then 
$\mathcal{C}(d)$
is equivalent to $\mathfrak{C}oh(\mathbb{C}^2, d)$. 
 
For a decomposition $d=d_1+\cdots+d_k$, 
let $\mathcal{C}(d_1, \ldots, d_k)$ be the derived moduli stack 
of filtrations of coherent sheaves on $\mathbb{C}^2$
\begin{align}\label{filt:Q}
0=Q_0 \subset	Q_1 \subset Q_2 \subset \cdots \subset Q_k
	\end{align}
such that each subquotient $Q_i/Q_{i-1}$ is a zero-dimensional 
sheaf on $\mathbb{C}^2$ with length $d_i$. 

The stack $\mathcal{C}(d_1, \ldots, d_k)$ is a quotient stack described as follows. Let $P=GL(d_1,\ldots, d_k)$, let $\mathfrak{p}=\mathfrak{gl}(d_1,\ldots, d_k)$, and let $s_\mathfrak{p}\colon \mathfrak{p}^{\oplus 2}\to \mathfrak{p}$ be the commutator map. The stack $\mathcal{C}(d_1,\ldots, d_k)$ is the moduli stack of filtrations of finite dimensinal representations of $\mathbb{C}[x,y]$ such that the graded pieces have dimensions $d_i$. The moduli stack of such filtrations of finite dimensional representations of the free algebra $\mathbb{C}\langle x,y\rangle$ is $\mathfrak{p}^{\oplus 2}/P$. Indeed, the actions of $x$ and $y$ on a flag of representations correspond to a pair $(A,B)\in \mathfrak{p}^{\oplus 2}$ of matrices. The stack $\mathcal{C}(d_1,\ldots, d_k)$ is obtained from $\mathfrak{p}^{\oplus 2}/P$ by imposing that the two matrices commute, so $\mathcal{C}(d_1,\ldots, d_k)=s_\mathfrak{p}^{-1}(0)/P$.   

There exist evaluation morphisms 
\begin{align*}
	\mathcal{C}(d_1) \times \cdots \times \mathcal{C}(d_k) \stackrel{q}{\leftarrow} 
	\mathcal{C}(d_1, \ldots, d_k) \stackrel{p}{\to} \mathcal{C}(d). 
	\end{align*}
The morphism $p$ is proper and $q$ is quasi-smooth. 
The above diagram for $k=2$
defines the categorical Hall product
\begin{align}\label{hall:ast}
	\ast=p_{\ast}q^{\ast} \colon 
	D^b(\mathcal{C}(d_1)) \boxtimes D^b(\mathcal{C}(d_2)) \to 
	D^b(\mathcal{C}(d))
	\end{align}
which is a special case of the two dimensional 
categorical Hall product defined by Porta--Sala~\cite{PoSa}. 
	
Let $T$ be the two-dimensional torus in (\ref{torus:T})
which acts on $\mathbb{C}^2$ 
by $(t_1, t_2) \cdot (x, y)=(t_1 x, t_2 y)$. 
It naturally induces an action on $\mathcal{C}(d)$. 
There is also a $T$-equivariant Hall product:
\begin{align}\label{hall:ast2}
	\ast=p_{\ast}q^{\ast} \colon 
	D^b_T(\mathcal{C}(d_1)) \boxtimes D^b_T(\mathcal{C}(d_2)) \to 
	D^b_T(\mathcal{C}(d)).
	\end{align}
Here, the box product is taken over $BT$. 
In what follows, whenever we take a box-product in the $T$-equivariant setting, 
we take it over $BT$.

\subsection{Subcategories \texorpdfstring{$\mathbb{T}(d)_v$}{Tdv}}\label{subsection:Tdv}
Let $i \colon \mathcal{C}(d) \hookrightarrow \mathcal{Y}(d)$ be the natural closed immersion. 
Define the full triangulated subcategory \[\widetilde{\mathbb{T}}(d)_v\subset D^b(\mathcal{Y}(d))\] 
generated by the vector bundles $\mathcal{O}_{\mathcal{Y}(d)}\otimes \Gamma_{GL(d)}(\chi)$
for a dominant weight $\chi$ satisfying
$\chi+\rho \in \textbf{W}(d)_{v}$. 
Define the full triangulated subcategory 
\begin{align}\label{def:N}
	\mathbb{T}(d)_v \subset D^b(\mathcal{C}(d))
	\end{align}
with objects $\mathcal{E}$ such that 
 $i_{\ast}\mathcal{E}$ is in $\widetilde{\mathbb{T}}(d)_v$.

It is proved in~\cite{P2} that the
subcategories (\ref{def:N}) form the building blocks 
of the category $D^b(\mathcal{C}(d))$:
\begin{thm}\cite[Corollary~3.3]{P2}\label{thm:com:sod}
We have the semiorthogonal decomposition 
\begin{align}\label{sod:C}
	D^b(\mathcal{C}(d))=\left\langle \mathbb{T}(d_1)_{v_1} \boxtimes 
	\cdots \boxtimes \mathbb{T}(d_k)_{v_k} \relmiddle|
	\begin{array}{c}
	v_1/d_1 < \cdots < v_k/d_k, \\
	d_1+\cdots+d_k=d
	\end{array} \right\rangle. 
	\end{align}
	Here each fully-faithful functor 
	\begin{align*}
	 \mathbb{T}(d_1)_{v_1} \boxtimes 
	\cdots \boxtimes \mathbb{T}(d_k)_{v_k}
	\hookrightarrow D^b(\mathcal{C}(d))
	\end{align*}
	is given by the categorical Hall product (\ref{hall:ast}). 
	\end{thm}

	 The $T$-equivariant version $\mathbb{T}_T(d)_v \subset D^b_T(\mathcal{C}(d))$
	 is defined similarly to (\ref{def:N}). 
The argument in loc.~cit.~also shows the semiorthogonal decomposition
\begin{align}\label{sod:C2}
	D^b_T(\mathcal{C}(d))=\left\langle \mathbb{T}_T(d_1)_{v_1} \boxtimes 
	\cdots \boxtimes \mathbb{T}_T(d_k)_{v_k} \relmiddle|
	\begin{array}{c}
	v_1/d_1 < \cdots < v_k/d_k, \\
	d_1+\cdots+d_k=d
	\end{array} \right\rangle. 
	\end{align}
Here the box product $\mathbb{T}_T(d_1)_{v_1} \boxtimes 
	\cdots \boxtimes \mathbb{T}_T(d_k)_{v_k}$ is taken over $BT$.
\subsection{Constructions of objects in \texorpdfstring{$\mathbb{T}(d)_v$}{TD}}
For each $(d, v)$ we construct 
an object $\mathcal{E}_{d, v} \in \mathbb{T}(d)_v$, 
which also gives an object in $\mathbb{T}_T(d)_v$, 
using the closed substack $\mathcal{Z} \subset \mathcal{C}(1, 1, \ldots, 1)$
defined as follows. 
Let $\lambda$ be the cocharacter 
\begin{align}\label{lambda:cochar}
	\lambda \colon \mathbb{C}^{\ast} \to T(d), \ 
	t \mapsto (t^d, t^{d-1}, \ldots, t). 
	\end{align}
	The attracting stack of $\mathcal{Y}(d)$
	with respect to $\lambda$ is given by 
	\begin{align}\label{mor:slambda}
	   \mathcal{Y}(d)^{\lambda \geq 0}:=
	   \left(\mathfrak{g}^{\lambda \geq 0}\right)^{\oplus 2}\Big/GL(d)^{\lambda \geq 0} 
	\end{align}
	where $GL(d)^{\lambda \geq 0} \subset GL(d)$ is the subgroup of 
	upper triangular matrices. 
Then the morphism 
(\ref{mor:s})
restricts to the morphism 
\begin{align}\label{mor:slambda2}
	s^{\lambda \geq 0} \colon 
	\mathcal{Y}(d)^{\lambda \geq 0} \to \mathfrak{g}^{\lambda \geq 0}
	\end{align}
whose derived zero locus 
$\mathcal{C}(d)^{\lambda \geq 0}$ is equivalent to $\mathcal{C}(1, \ldots, 1)$. 
Let $X=(x_{i,j})$ and $Y=(y_{i,j})$ be elements of $\mathfrak{g}^{\lambda \geq 0}$
for $1\leq i, j\leq d$, 
where $x_{i,j}=y_{i,j}=0$ for $i>j$. 
Then the 
equation $s^{\lambda \geq 0}(X, Y)=0$ is 
\begin{align*}
\sum_{i\leq a \leq j} x_{i,a}y_{a,j}=\sum_{i\leq a \leq j} y_{i,a}x_{a,j}, 
	\end{align*}
for each $(i, j)$ with $i\leq j$. 
We call the above equation $\mathbb{E}_{i, j}$. 
The equation $\mathbb{E}_{i, i}$ is 
$x_{i, i} y_{i, i}-y_{i, i} x_{i, i}=0$, which always holds but 
imposes a non-trivial derived structure on 
$\mathcal{C}(1, \ldots, 1)$. 
The equation $\mathbb{E}_{i, i+1}$ is 
\begin{align*}
	(x_{i,i}-x_{i+1, i+1})y_{i, i+1}-(y_{i,i}-y_{i+1, i+1})x_{i, i+1}=0. 
	\end{align*}
The above equation is satisfied if the following 
equation $\mathbb{F}_{i, i+1}$ is satisfied: 
\begin{align*}
	\{x_{i, i}-x_{i+1, i+1}=0, y_{i, i}-y_{i+1, i+1}=0\}. 
	\end{align*}
We define the closed derived substack 
\begin{align}\label{defZ}
	\mathcal{Z}:=\mathcal{Z}(d) \subset \mathcal{Y}^{\lambda \geq 0}(d)
	\end{align}
to be the derived zero locus of 
the equations $\mathbb{F}_{i, i+1}$ for all $i$ and 
$\mathbb{E}_{i, j}$ for all $i+2 \leq j$. We usually drop $d$ from the notation if the dimension is clear from the context.
Then $\mathcal{Z}$ is a closed substack of 
$\mathcal{C}(d)^{\lambda \geq 0}=\mathcal{C}(1, \ldots, 1)$. 
Note that, set theoretically, the closed substack $\mathcal{Z}$
corresponds to filtrations (\ref{filt:Q}) 
such that each $Q_i/Q_{i-1}$ is isomorphic to $\mathcal{O}_x$ for some $x \in \mathbb{C}^2$
independent of $i$. 

We have the diagram of attracting loci 
\begin{align*}
	\mathcal{C}(1)^{\times d}=\mathcal{C}(d)^{\lambda} \stackrel{q}{\leftarrow}
	\mathcal{C}(d)^{\lambda \geq 0} \stackrel{p}{\to} \mathcal{C}(d),
	\end{align*}
where $p$ is a proper morphism. We set
\begin{align}\label{def:mi}
	m_i :=\left\lceil \frac{vi}{d} \right\rceil -\left\lceil \frac{v(i-1)}{d} \right\rceil
	+\delta_i^d -\delta_i^1 \in \mathbb{Z},
	\end{align}
	where $\delta^j_i$ is the Kronecker delta function: $\delta^j_i=1$ if $i=j$ and $\delta^j_i=0$ otherwise.
		 For a weight
	 $\chi=\sum_{i=1}^d n_i \beta_i$ with $n_i \in \mathbb{Z}$, 
		    we denote by $\mathbb{C}(\chi)$
		    the one dimensional $GL(d)^{\lambda \geq 0}$-representation given by 
		    \begin{align*}
		       GL(d)^{\lambda \geq 0} \to GL(d)^{\lambda}=T(d) \stackrel{\chi}{\to} \mathbb{C}^{\ast},
		    \end{align*}
		    where the first morphism is the projection.
		    \begin{defn}\label{definition:Edw}
We define the complex $\mathcal{E}_{d, v}$ by 
\begin{align}\label{def:Edw}
\mathcal{E}_{d, v}:=p_{\ast}\left(\mathcal{O}_{\mathcal{Z}} \otimes 
\mathbb{C}(m_1, \ldots, m_d)\right)
\in D^b(\mathcal{C}(d)). 
	\end{align}
The construction above is $T$-equivariant, so we also obtain an object $\mathcal{E}_{d, w}\in D^b_T(\mathcal{C}(d))$.
	\end{defn}
	
	In the following lemma, we show that $\mathcal{E}_{d, v}$
	is indeed an object in $\mathbb{T}(d)_v$:

\begin{lemma}\label{lem:E}
	We have 
	$\mathcal{E}_{d, v} \in \mathbb{T}(d)_{v}$. 
	\end{lemma}
\begin{proof}
Let $\mathcal{Y}(d)^{\lambda}$ be
the $\lambda$-fixed stack of $\mathcal{Y}(d)$
with respect to the cocharacter (\ref{lambda:cochar}). 
	As in Proposition~\ref{bbw}, we have the commutative diagram 
	\begin{align*}
		\xymatrix{
	\mathcal{C}(d)^{\lambda}
	\ar@<-0.3ex>@{^{(}->}[d]^-{i} & \mathcal{C}(d)^{\lambda\geq 0} \ar[r]^-{p} 
	\ar@<-0.3ex>@{^{(}->}[d]^-{i} \ar[l]_-{q}& \mathcal{C}(d) 
	\ar@<-0.3ex>@{^{(}->}[d]^-{i} \\
	\mathcal{Y}(d)^{\lambda}
	& \mathcal{Y}(d)^{\lambda\geq 0} \ar[r]^-{p} \ar[l]_-{q}& \mathcal{Y}(d)
	}
		\end{align*}
	where each vertical arrow is a closed immersion.
	We have 
	\begin{align*}
		i_{\ast}\mathcal{E}_{d, v}&=i_{\ast}p_{\ast}(\mathcal{O}_{\mathcal{Z}} \otimes \mathbb{C}(m_1, \ldots, m_d)) \\
		&\cong p_{\ast}(i_{\ast}\mathcal{O}_{\mathcal{Z}} \otimes q^{\ast}\mathcal{O}_{\mathcal{Y}(d)^{\lambda} }(m_1, \ldots, m_d)). 
			\end{align*}
		There is a sequence of closed substacks 
		\begin{align*}
			\mathcal{Z}=\mathcal{Z}_{d-1} \hookrightarrow 
			\mathcal{Z}_{d-2} \hookrightarrow \cdots \hookrightarrow 
			\mathcal{Z}_1 \hookrightarrow \mathcal{Y}(d)^{\lambda \geq 0}. 
			\end{align*}
		Here $\mathcal{Z}_1$ is the derived zero locus of the equations
		$\mathbb{F}_{i, i+1}$ for all $1\leq i \leq d-1$, 
		and $\mathcal{Z}_{l} \subset \mathcal{Z}_{l-1}$ for $2 \leq i \leq d-1$
		is the derived zero locus of 
		the equations $\mathbb{E}_{i, i+l}$ on $\mathcal{Z}_{l-1}$ for all $1\leq i\leq d-l$. 
		More precisely, let $\mathfrak{g}_l \subset \mathfrak{g}^{\lambda \geq 0}$ be the 
		sub $GL(d)^{\lambda \ge 0}$-representation
		defined by 
		\begin{align*}
		    \mathfrak{g}_l :=\{(x_{i,j}) \mid 1\leq i, j \leq d, x_{i,j}=0 \mbox{ for }
		    j\leq i+l\}.
		\end{align*}
		Then the morphism~\eqref{mor:slambda2} 		restricts to $\mathcal{Z}_1 \to \mathfrak{g}_1$. 
		The closed substack $\mathcal{Z}_l$ is defined by 
		the derived fiber product
		$\mathcal{Z}_l :=\mathcal{Z}_1 \times_{\mathfrak{g}_1} \mathfrak{g}_l$.
		So we have the sequence of squares 
		\begin{align*}
		    \xymatrix{
		    		    \cdots  \ar@<-0.3ex>@{^{(}->}[r] \ar@{}[rd]|\square&\mathcal{Z}_3 \ar[d]  \ar@<-0.3ex>@{^{(}->}[r] \ar@{}[rd]|\square &
		    \mathcal{Z}_2 \ar[d]  \ar@<-0.3ex>@{^{(}->}[r] \ar@{}[rd]|\square &
		    \mathcal{Z}_1 \ar[d]  \ar@<-0.3ex>@{^{(}->}[r] & \mathcal{Y}(d)^{\lambda \geq 0} \ar[d]^-{s^{\lambda \ge 0}} \\
		  \cdots  \ar@<-0.3ex>@{^{(}->}[r] & \mathfrak{g}_3   \ar@<-0.3ex>@{^{(}->}[r]&\mathfrak{g}_2   \ar@<-0.3ex>@{^{(}->}[r]&  \mathfrak{g}_1   \ar@<-0.3ex>@{^{(}->}[r]& \mathfrak{g}_0. 
		    }
		\end{align*}
		We note that the right square is not (derived or underived) Cartesian, 
		while the other squares are 
		derived Cartesians. 
		We also write $\mathcal{Z}_l=Z_l/GL(d)^{\lambda \geq 0}$
		for a closed subscheme $Z_l \subset (\mathfrak{g}^{\lambda \geq 0})^{\oplus 2}$. 
		
	Since the equations $\mathbb{F}_{i, i+1}$ are $GL(d)^{\lambda \geq 0}$-invariant, 
	they form a section
	$s_1$ of 
	$\mathcal{V}_1 :=\mathcal{O}^{\oplus 2(d-1)}$
	on $\mathcal{Y}(d)^{\lambda \geq 0}$. 
	Therefore 
	we have 
	\begin{align}\label{OZ1}
		\mathcal{Z}_1 \simeq \text{Spec}\left(\mathcal{O}_{\left(\mathfrak{g}^{\oplus 2}\right)^{\lambda\geq 0}}\left[\mathcal{V}_1^{\vee}[1]; d_{s_1}\right]\right)\Big/GL(d)^{\lambda\geq 0},
		\end{align}
		where the differential $d_{s_1}$ is induced by the section $s_1$.
		
		Suppose that $l\geq 2$. 
		We have the isomorphism of $GL(d)^{\lambda \geq 0}$-representations
		\begin{align*}
		    \mathfrak{g}_{l-1}/\mathfrak{g}_l =\bigoplus_{j=i+l}\mathbb{C}x_{i,j}
		    =\bigoplus_{j=i+l}\mathbb{C}(\beta_i-\beta_j).
		    \end{align*}
		   		    Thus the equation $\mathbb{E}_{i, i+l}$ on $\mathcal{Z}_{l-1}$ is an equation of a section of
	$\mathcal{O}_{\mathcal{Z}_{l-1}}(\beta_i-\beta_{i+l})$ for any $1\leq i\leq d-l$, so 
	these equation form a $GL(d)^{\lambda \geq 0}$-invariant section $s_{l}$ of the vector bundle  
	\begin{align*}
	    \mathcal{V}_{l}:=\bigoplus_{i=1}^{d-l}
	    \mathcal{O}_{Z_{l-1}}(\beta_i-\beta_{i+l})
	    \to Z_{l-1}. 
	\end{align*}
	Therefore
	we have 
	\begin{align}\label{OZ2}
	\mathcal{Z}_l \simeq 
	\Spec\left(\mathcal{O}_{Z_{l-1}}\left[\mathcal{V}_{l}^{\vee}[1]; d_{s_{l}} \right]  \right)
	/GL(d)^{\lambda \geq 0}. 
			\end{align}
	On the other hand, the closed substack
	$\mathcal{Y}^{\lambda \geq 0}(d) \hookrightarrow \left[\mathfrak{g}^{\oplus 2}/GL(d)^{\lambda \geq 0}\right]$
	is realized as a derived zero locus 
	with respect to the $GL(d)^{\lambda \geq 0}$-invariant section $s_0$ of the 
	vector bundle 
	\begin{align*}
	\mathcal{V}_0:=
	   \mathcal{O}_{\mathfrak{g}^{\oplus 2}}\otimes (\mathfrak{g}^{\lambda<0})^{\oplus 2}
	    =\bigoplus_{i>j}\mathcal{O}_{\mathfrak{g}^{\oplus 2}}(\beta_i-\beta_j)^{\oplus 2}
	    \to \mathfrak{g}^{\oplus 2}
	\end{align*}
	given by the projection 
	$\mathfrak{g} \twoheadrightarrow \mathfrak{g}^{\lambda<0}$. 
	Therefore 
	we have 
	\begin{align}\label{OY}
	\mathcal{Y}(d)^{\lambda \geq 0}
	\simeq \Spec \left(\mathcal{O}_{\mathfrak{g}^{\oplus 2}}\left[\mathcal{V}_{0}^{\vee}[1]; d_{s_{0}} \right]  \right)
	/GL(d)^{\lambda \geq 0}.
	\end{align}
		The map $p$ factors as follows:
	\begin{align}\label{fact:p}
		p \colon \mathcal{Y}(d)^{\lambda \geq 0} \hookrightarrow 
		\mathfrak{g}^{\oplus 2}/GL(d)^{\lambda \geq 0}  \stackrel{h}{\to} 
		\mathcal{Y}(d).
		\end{align}
		Consider a weight $\chi\in M$.
	Together with (\ref{OZ1}), (\ref{OZ2}) and (\ref{OY}), 
	the Borel-Weil-Bott Theorem, see Proposition \ref{bbw}, implies that 
	$p_{\ast}(i_{\ast}\mathcal{O}_{\mathcal{Z}} \otimes q^{\ast}\mathcal{O}(\chi))$ is resolved by 
	vector bundles of the form 
	\begin{align*}
		\mathcal{O}_{\mathcal{Y}(d)}\otimes\Gamma_{GL(d)}\left((\chi-\sigma)^+\right) 
		\end{align*}
	where $\sigma \in M$ is written as 
	\begin{align*}
		\sigma=\sum_{i>j}
	a_{i, j}(\beta_i-\beta_j), \
	-1 \leq a_{i, j} \leq 2, 
	0\leq a_{i+1, i} \leq 2. 
		\end{align*}
	We take $\chi$ to be
	\begin{equation}\label{defchi}
		\chi=\sum_{i=1}^d m_i \beta_i 
		=\sum_{i=1}^{d-1}\left(\frac{vi}{d}+1-\left\lceil \frac{vi}{d} \right\rceil   \right)(\beta_{i+1}-\beta_i)+\frac{v}{d}\sum_{i=1}^d \beta_i.
		\end{equation}
	Then each coefficient of $(\beta_{i+1}-\beta_i)$ lies in $(0, 1]$, 
	so 
	we have $\chi-\sigma+\rho \in \textbf{W}(d)_v$ and thus also $\left(\chi-\sigma\right)^++\rho \in \textbf{W}(d)_v$. 
	Therefore $i_{\ast}\mathcal{E}_{d, v}$ is an 
	object of the category $\widetilde{\mathbb{T}}(d)_v$, and so $\mathcal{E}_{d, v}$ is an object of $\mathbb{T}(d)_{v}$.
	\end{proof}

\subsection{Koszul duality}\label{subsection:Koszulduality}
Recall the stack $\mathcal{X}(d)$
and its regular function $\Tr W$ from (\ref{def:Wd}):
\begin{align}\label{funct:W}
	\Tr W \colon \mathcal{X}(d) \to \mathbb{C}, \ 
	(X, Y, Z) \mapsto \mathrm{Tr}(Z[X, Y]).
	\end{align}
	Recall the definition of graded matrix factorizations from Subsection \ref{gradingMF}. Consider the grading induced by the action of $\mathbb{C}^*$ on $\X(d)$ scaling the linear map corresponding to $Z$ with weight $2$.
The Koszul duality equivalence, also called dimensional reduction in the literature, gives the following:
\begin{thm}\cite{I, Hirano, T}
There is an equivalence 
\begin{align}\label{equiv:Phi}
	\Phi \colon D^b(\mathcal{C}(d)) \stackrel{\sim}{\to}
	\mathrm{MF}^{\mathrm{gr}}(\mathcal{X}(d), \Tr W). 
	\end{align}
	\end{thm}
The equivalence (\ref{equiv:Phi}) is given by 
taking the tensor product over $\mathcal{O}_{\mathcal{C}(d)}$
with 
the Koszul factorization 
\begin{align*}
	\mathcal{K}:=(\mathcal{O}_{\mathcal{C}(d)} \otimes_{\mathcal{O}_{\mathcal{Y}(d)}}
	\mathcal{O}_{\mathcal{X}(d)}, d_{\mathcal{K}}), \ 
	d_{\mathcal{K}}=d_{s} \otimes 1 + \kappa,
	\end{align*}
	where $d_s$
	is the differential in (\ref{diff:ds2})
	and $\kappa \in \mathfrak{g}^{\vee} \otimes \mathfrak{g}$
	corresponds to $\id \in \Hom(\mathfrak{g}, \mathfrak{g})$
	(see~\cite[Equation (2.3.2)]{T}). 
A quasi-inverse of (\ref{equiv:Phi}) is given by 
$\Hom_{\mathcal{X}(d)}(\mathcal{K}, -)$. 
More explicitly, for $\mathcal{E} \in D^b(\mathcal{C}(d))$
we have 
\begin{align*}
    \Phi(\mathcal{E})=(\mathcal{E}\otimes_{\mathbb{C}}\mathrm{Sym}(\mathfrak{g}), d_{\Phi(\mathcal{E})}), \ d_{\Phi(\mathcal{E})}(u\otimes v)=d_{\mathcal{E}}(u) \otimes v +\sum_{i=1}^k 
    (u \cdot e_i^{\vee})
\otimes (e_i \cdot v)
\end{align*}
where $\{e_1, \ldots, e_k\}$ is a basis of $\mathfrak{g}$
and $\{e_1^{\vee}, \ldots, e_k^{\vee}\}$
is its dual basis of $\mathfrak{g}^{\vee}$. 

Define 
\begin{align*}
	 \mathbb{S}^{\text{gr}}(d)_{w} \subset \text{MF}^{\text{gr}}(\mathcal{X}(d), \Tr W)
	\end{align*} as in Subsection \ref{gradedMFdef}.
We have the following lemma. 
\begin{lemma}\label{lem:PhiTS}
The equivalence $\Phi$
in (\ref{equiv:Phi}) restricts to the equivalence  
\begin{align}\label{equiv:NM}
\Phi \colon \mathbb{T}(d)_v	\stackrel{\sim}{\to} \mathbb{S}^{\text{gr}}(d)_{w}, \ 
v=w. 
	\end{align}
 \end{lemma}
 \begin{proof}
     For $A \in D^b(\mathcal{C}(d))$, 
     the object $\Phi(A)$ is of the form 
     $i_{\ast}A \otimes_{\mathcal{O}_{\mathcal{Y}}(d)}
     \mathcal{O}_{\mathcal{X}}(d)$ with 
     some differential. 
     Thus, by Lemma~\ref{lem:MF(M)gr}, 
     the functor $\Phi$ restricts to 
     the functor $\mathbb{T}(d)_v \to \mathbb{S}^{\text{gr}}(d)_v$.
     Conversely, for $B \in \mathbb{S}^{\rm{gr}}(d)_v$, the complex
     $i_{\ast}\Phi^{-1}(B)$
     is of the form $B \otimes_{\mathcal{O}_{\mathcal{X}(d)}}\mathcal{K}^{\vee} \cong 
     0^{\ast}B$
     where $0$ is the zero section of the projection
     $\mathcal{X}(d) \to \mathcal{Y}(d)$, see~\cite[Equation (2.3.6)]{T}. Therefore, by Lemma~\ref{lem:MF(M)gr}, 
     the functor $\Phi^{-1}$ sends $\mathbb{S}^{\rm{gr}}(d)_v$ to $\mathbb{T}(d)_v$. 
 \end{proof}
Applying 
the functor $\Phi$ to $\mathcal{E}_{d, v}$
and setting $v=w$, 
we obtain the following: 
\begin{defn}\label{def:objF}
For $(d, w) \in \mathbb{N} \times \mathbb{Z}$, we define 
the following object 
\begin{align*}
    \mathcal{F}_{d, w} :=\Phi(\mathcal{E}_{d, w}) \in \mathbb{S}^{\mathrm{gr}}(d)_w. 
\end{align*}
Its image under the forgetful functor 
$\mathbb{S}^{\mathrm{gr}}(d)_w \to \mathbb{S}(d)_w$
is also denoted by $\mathcal{F}_{d, w} \in \mathbb{S}(d)_w$. 
We also obtain an object in the $T$-equivariant category $\mathcal{F}_{d, w} \in \mathbb{S}_T(d)_w$.
\end{defn}

\begin{example}\label{exam:F21}
	Let $V=\mathbb{C}^2$. 
	A straightforward computation shows that
	$\mathcal{F}_{2, 1} \in \mathbb{S}(2)_1$
	is explicitly given by
	\begin{align*}
		\mathcal{F}_{2, 1}=
		\left(\xymatrix{ V^{\oplus 4} \otimes \mathcal{O}_{\mathcal{X}(2)}
		 \ar@/^10pt/[r]^-{\delta} & V^{\oplus 4} \otimes \mathcal{O}_{\mathcal{X}(2)}
		 \ar@/^10pt/[l]^-{\delta^{\vee}} } \right)
		\end{align*}
	where over $(X, Y, Z) \in \mathfrak{g}^{\oplus 3}$, the maps
	$\delta$ and $\delta^{\vee}$ are given by 
	\begin{align*}
		\delta=\begin{pmatrix}
			-\widetilde{Z} & \widetilde{X} & \widetilde{Y} & -4 \\
			[Y, Z] & [X, Y] & 0 & \widetilde{Y} \\
			[Z, X] & 0 & [X, Y] & -\widetilde{X} \\
			0 & [Z, X] & [Z, Y] & -\widetilde{Z}
			\end{pmatrix}, \ 
			\delta^{\vee}=\begin{pmatrix}
			[X, Y] & \widetilde{X} & \widetilde{Y} & 0 \\
			[Z, Y] & \widetilde{Z} & 0 & \widetilde{Y} \\
			[X, Z] & 0 & \widetilde{Z} & -\widetilde{X} \\
			-\Tr W \cdot I & [X, Z] & [Y, Z] & [X, Y]
		\end{pmatrix}.
		\end{align*}
	Here for $X \in \mathfrak{g}$, we set 
	$\widetilde{X}:=2X-\Tr X \cdot I$. 
		\end{example}
		\begin{remark}
				By its construction, the complex $\mathcal{E}_{d, v}$ is supported on 
	the pull-back of the small diagonal of 
	the support map $\mathcal{C}oh(\mathbb{C}^2, d) \to \mathrm{Sym}^d(\mathbb{C}^2)$. 
	On the other hand, by the expression in Example~\ref{exam:F21}, the complex $\mathcal{F}_{2, 1}$ is symmetric 
	with respect to the permutation of $(X, Y, Z)$. 
	It follows that $\mathcal{F}_{2, 1}$
	is supported on the pull-back of the small diagonal 
	of the support map $\mathcal{C}oh(\mathbb{C}^3, d) \to \mathrm{Sym}^d(\mathbb{C}^3)$. 
	This is an analogue of Davison's support lemma for BPS sheaves~\cite[Lemma~4.1]{Dav}
	and will be proved for an arbitrary object in $\mathbb{S}(d)_w$
	with $\gcd(d, w)=1$ in~\cite{PT1}. 
		\end{remark}
	From the functorial properties of Koszul 
	duality equivalences in~\cite[Section~2.4]{T}, 
	they are compatible with
	the Hall products by the following commutative diagram (see~\cite[Proposition~3.1]{P2})
	\begin{align}\label{com:hall}
		\xymatrix{
	D^b(\mathcal{C}(d_1)) \boxtimes D^b(\mathcal{C}(d_2)) \ar[r]^-{\ast} \ar[d]^-{\widetilde{\Phi}} &
	D^b(\mathcal{C}(d)) \ar[d]^-{\Phi} \\
	  \text{MF}^{\text{gr}}(\mathcal{X}(d_1), \Tr W_{d_1})
	  \boxtimes  \text{MF}^{\text{gr}}(\mathcal{X}(d_2), \Tr W_{d_2})
	  \ar[r]^-{\ast} &  \text{MF}^{\text{gr}}(\mathcal{X}(d), \Tr W_d).
	}
	\end{align}
Here, the left arrow $\widetilde{\Phi}$ is the composition of Koszul duality equivalences (\ref{equiv:Phi}) with the 
tensor product of 
\begin{equation}\label{def:factordimred}
	\det((\mathfrak{g}^{\nu>0})^{\vee}(2))[-\dim \mathfrak{g}^{\nu>0}]
	=(\det V_1)^{-d_2} \otimes (\det V_2)^{d_1}[d_1 d_2].
	\end{equation}
	Here the cocharacter 
	$\nu \colon \mathbb{C}^{\ast} \to T(d)$
is $\nu(t)=(\overbrace{t, \ldots, t}^{d_1}, \overbrace{1, \ldots, 1}^{d_2})$, the vector spaces $V_i$ have 
$\dim V_i=d_i$ for $i=1, 2$, and 
	$(1)$ is a twist by the weight one $\mathbb{C}^{\ast}$-character, 
	which is isomorphic to the shift functor $[1]$
	of the category of graded matrix factorizations. 
As an application of the diagram (\ref{com:hall}), 
we have the following: 
\begin{lemma}\label{lemma:N(d)}
    The categorical Hall product (\ref{hall:ast})
    restricts to the functor 
    \begin{align*}
        \ast \colon 
        \mathbb{T}(d_1)_{v_1} \boxtimes 
        \mathbb{T}(d_2)_{v_2} \to \mathbb{T}(d)_v
    \end{align*}
    if $v_1/d_1=v_2/d_2=v/d$. 
\end{lemma}
\begin{proof}
By the diagram (\ref{com:hall})
and the equivalence (\ref{equiv:NM}), 
it is enough to show that the bottom horizontal arrow in (\ref{com:hall})
	restricts to the functor 
	\begin{align}\label{com:hall2}
		\ast \colon \mathbb{S}^{\text{gr}}(d_1)_{w_1} \boxtimes 
		\mathbb{S}^{\text{gr}}(d_2)_{w_2} \to \mathbb{S}^{\text{gr}}(d)_{w}
		\end{align}
	for $w=v$, $w_1=v_1-d_1 d_2$ and $w_2=v_2+d_1 d_2$. 
	Let us take 
	\begin{align*}
		\chi_1+\rho_1 \in \textbf{W}(d_1)_{w_1}, \ 
		\chi_2+\rho_2 \in \textbf{W}(d_2)_{w_2}.
	\end{align*}
Then 
we have 
\begin{align}\label{chi:W}
    \chi_1+\chi_2+\rho &= (\chi_1+\rho_1)+(\chi_2+\rho_2)+(\rho-\rho_1-\rho_2) \\
 \notag&\in \textbf{W}(d_1)_{w_1}+\textbf{W}(d_2)_{w_2}
 +\frac{1}{2}\mathfrak{g}^{\nu<0} \\
 \notag&=\textbf{W}(d_1)_{v_1}+\textbf{W}(d_2)_{v_2}
 +\frac{3}{2}\mathfrak{g}^{\nu<0}.
    \end{align}
By Proposition~\ref{bbw}, 
the product $\left(\Gamma_{GL(d_1)}(\chi_1) \otimes \mathcal{O}_{\X(d_1)}\right) \ast 
\left(\Gamma_{GL(d_2)}(\chi_2) \otimes \mathcal{O}_{\X(d_2)}\right)$ 
for the 
categorical Hall product 
\begin{align*}
\ast \colon D^b(\mathcal{X}(d_1)) \boxtimes 
D^b(\mathcal{X}(d_2)) \to D^b(\mathcal{X}(d))
\end{align*}
is generated by the vector bundles
$\Gamma_{GL(d)}\left((\chi_1+\chi_2-\sigma_I)^+\right) \otimes \mathcal{O}_{\X(d)}$ for 
$\sigma_I$ partial sums of weights 
in $3(\mathfrak{g}^{\nu<0})$. 
More precisely, 
$\sigma_I=\sum_{\beta\in I}\beta$ for $I$ a submultiset of the multiset $\mathcal{W}$ which contains the weights $\beta_j-\beta_i$ for $1\leq i\leq d_1<j\leq d$ with multiplicity three.
By (\ref{chi:W}), we conclude that 
\begin{align*}
	\chi_1+\chi_2-\sigma +\rho \in \textbf{W}(d)_v=\textbf{W}(d)_w.
	\end{align*}
Therefore the bottom horizontal arrow in 
(\ref{com:hall}) restricts to the functor (\ref{com:hall2}).
\end{proof}

\subsection{Relation with shuffle algebra}\label{subsection:shuffle}
The $T$-equivariant Hall product \eqref{hall:ast2} induces an associative algebra structure 
\begin{align}\label{G:prod}
\ast \colon G_T(\mathcal{C}(d_1)) \otimes_{\mathbb{K}}G_T(\mathcal{C}(d_2)) \to G_T(\mathcal{C}(d)),
	\end{align}
	where $\mathbb{K}=\mathbb{Z}\left[q_1^{\pm 1}, q_2^{\pm 1}\right]$ is the Grothendieck 
group of $BT$ and $q_i$ are $T$-weights. 
This Hall product has been studied in~\cite{SV, N2, Zhao0}.
Let $i \colon \mathcal{C}(d) \hookrightarrow \mathcal{Y}(d)$ be the closed immersion. 
The pull-back 
via $\mathcal{Y}(d) \to BGL(d)$ 
gives the isomorphism 
\begin{align*}
    \bigoplus_{d\geq 0}K_T(BGL(d))=
\bigoplus_{d\geq 0}\mathbb{K}\left[z_1^{\pm 1}, \ldots, z_d^{\pm 1}\right]^{\mathfrak{S}_d}
\stackrel{\cong}{\to}
\bigoplus_{d\geq 0} K_T(\mathcal{Y}(d)). 
\end{align*}
Therefore the push-forward by $i$ induces the morphism 
\begin{align}\label{i:ast}
	i_{\ast} \colon 
	\bigoplus_{d\geq 0} G_T(\mathcal{C}(d)) \to 
\bigoplus_{d\geq 0}\mathbb{K}\left[z_1^{\pm 1}, \ldots, z_d^{\pm 1}\right]^{\mathfrak{S}_d}. 
	\end{align}
The product (\ref{G:prod}) is compatible with
the shuffle algebra on the right hand side of \eqref{i:ast}, defined as follows. 
Let $\xi(x)$ be defined by 
\begin{align*}
	\xi(x):=\frac{(1-q_1^{-1}x)(1-q_2^{-1}x)(1-q^{-1}x^{-1})}{1-x},
	\end{align*}
where $q:=q_1 q_2$. 
For $f \in \mathbb{K}\left[z_1^{\pm 1}, \ldots, z_a^{\pm 1}\right]$
and $g \in \mathbb{K}\left[z_{a+1}^{\pm 1}, \ldots, z_{a+b}^{\pm 1}\right]$, we set 
\begin{align}\label{def:shuffle}
	f \ast g:=\frac{1}{a! b!}
	\mathrm{Sym}\left(fg \cdot \prod_{\substack{1\leq i\le a,\\ a<j\leq a+b}}
	\xi(z_i z_j^{-1})\right),
	\end{align}
where we denote by $\mathrm{Sym}(h(z_1, \ldots, z_d))$ the sum of 
$h(z_{\sigma(1)}, \ldots, z_{\sigma(d)})$ for $\sigma \in \mathfrak{S}_d$. 
The following lemma should be well-known (cf.~\cite[Proposition~2.7]{N2}). 
We reprove it via Koszul duality: 
\begin{lemma}\label{lem:shuffle}
The morphism (\ref{i:ast}) is compatible with the above $\ast$-products. 
\end{lemma}
\begin{proof}
By the construction of the Koszul duality equivalence,
we have the commutative diagram 
\begin{align*}
	\xymatrix{
	G_T(\mathcal{C}(d)) \ar[r]^-{i_{\ast}} \ar[d]_-{\Phi}^-{\cong}
	& K_T(\mathcal{Y}(d)) \ar[d]^-{\cong} & \ar[l]_-{\cong} \mathbb{K}[z_1^{\pm 1}, \ldots, 
	z_d^{\pm 1}]^{\mathfrak{S}_d} \ar@{=}[d] \\
	K_T(\text{MF}^{\text{gr}}(\mathcal{X}(d), \Tr W)) \ar[r] & 
		K_T(\text{MF}^{\text{gr}}(\mathcal{X}(d), 0)) & \ar[l]_-{\cong} \mathbb{K}[z_1^{\pm 1}, \ldots, 
		z_d^{\pm 1}]^{\mathfrak{S}_d}.
	}
	\end{align*}	
The middle vertical arrow is given by the pull-back of 
the projection $\mathcal{X}(d) \to \mathcal{Y}(d)$, 
and the bottom left arrow is given by 
\begin{align}\notag
(\alpha \colon \mathcal{F}\rightleftarrows \mathcal{G} \colon \beta)
\mapsto [\mathcal{F}]-[\mathcal{G}], 
	\end{align}
	see \cite[Proposition~3.6]{P0} that this assignment induces a map on Grothendieck groups and that it preserves the Hall product.
On the other hand, under the bottom right isomorphism, 
the Hall product on $K_T(\text{MF}^{\text{gr}}(\mathcal{X}(d), 0))$ is 
compatible with the shuffle product
(\ref{def:shuffle}) 
after replacing $\xi(x)$ with $\xi'(x)$ (see~\cite[Proposition~3.4]{P0}):
\begin{align*}
	\xi'(x):=\frac{(1-q_1^{-1}x)(1-q_2^{-1}x)(1-qx)}{1-x}. 
\end{align*}	
Then the lemma follows from the commutative diagram (\ref{com:hall})
together with the explicit computation
\begin{align*}
	(-1)^{d_1 d_2}q^{-d_1 d_2}\prod_{\begin{subarray}{c}
			1\leq i\leq d_1 \\
			d_1<j\leq d
			\end{subarray}}
		z_i^{-1} z_j \cdot 
		\prod_{\begin{subarray}{c}
				1\leq i\leq d_1 \\
				d_1<j\leq d
		\end{subarray}}\xi'(z_i z_j^{-1})=
	\prod_{\begin{subarray}{c}
			1\leq i\leq d_1 \\
			d_1<j\leq d
	\end{subarray}}\xi(z_i z_j^{-1}). 
	\end{align*}
	\end{proof}
\begin{remark}
	In~\cite{N2},
	two kinds of shuffle products (with respect to $\xi'$, $\xi$)
	are considered depending on a choice of some 
	line bundle. 
	They correspond to Hall products on 
	$\mathrm{MF}^{\mathrm{gr}}_T(\mathcal{X}(d), \Tr W)$, 
	$D^b_T(\mathcal{C}(d))$ respectively, which are related by Koszul duality.  
	\end{remark}

\subsection{Generators}\label{subsection:generators}

The object $\mathcal{E}_{d, v} \in \mathbb{T}_T(d)_v$ defined in (\ref{def:Edw})  determines a class 
$[\mathcal{E}_{d, v}] \in G_T(\mathcal{C}(d))$. 
\begin{lemma}\label{lem:iE}
	We have the equality 
	\begin{multline}\label{elem:E}
		i_{\ast}[\mathcal{E}_{d, v}] 
		=(1-q_1^{-1})^{d-1}(1-q_2^{-1})^{d-1}\cdot
	\\	
	\mathrm{Sym}\left(\frac{z_1^{m_1} \cdots z_d^{m_d}}{(1-q^{-1}z_1^{-1}z_2)\cdots 
			(1-q^{-1}z_{d-1}^{-1}z_d)} 
		\cdot \prod_{j>i}\xi(z_i z_j^{-1})  \right).
		\end{multline}
	Here the exponents $m_i$ for $1\leq i\leq d$ are given by \eqref{def:mi}. 
	\end{lemma}
\begin{proof}
Consider the weight $\chi=\sum_{i=1}^d m_i\beta_i$ for $m_i$ 
given in (\ref{def:mi}). Denote by $\mathcal{O}$ the structure sheaf of 
		$\mathfrak{g}^{\oplus 2}/GL(d)^{\lambda \geq 0}$ and recall the morphism $h$ given in (\ref{fact:p}). 
	By taking account of the $T$-weights from 
	the argument used to prove Lemma~\ref{lem:E}, 
	the element $i_{\ast}[\mathcal{E}_{d, v}]$
	is given by 
	\begin{align*}
		&	i_{\ast}[\mathcal{E}_{d, v}] 
		=\left(1-q_1^{-1}\right)^{d-1}\left(1-q_2^{-1}\right)^{d-1}\cdot \\
		& h_{\ast}\left(\mathcal{O}\left(\chi\right)\prod_{j\geq i+2}(1-q^{-1}\mathcal{O}(\beta_j-\beta_i))
		\prod_{j>i}(1-q_1^{-1}\mathcal{O}(\beta_i-\beta_j))
		(1-q_2^{-1}\mathcal{O}(\beta_i-\beta_j))   \right),
		\end{align*}
		where the weights $\left(1-q_1^{-1}\right)^{d-1}\left(1-q_2^{-1}\right)^{d-1}$ appear by \eqref{OZ1}, the weights of the first product 
		\begin{align*}
		  \prod_{j\geq i+2}(1-q^{-1}\mathcal{O}(\beta_j-\beta_i))  
		\end{align*}
				appear by \eqref{OZ2}, and the weights of the second product 
				\begin{align*}
			\prod_{j>i}(1-q_1^{-1}\mathcal{O}(\beta_i-\beta_j))
		(1-q_2^{-1}\mathcal{O}(\beta_i-\beta_j))  	
		\end{align*}
				appear by \eqref{OY}.
				
	The push-forward $h_{\ast}$ 
	is given by symmetrization of the product
	with 
	$\prod_{j>i}(1-z_i z_j^{-1})^{-1}$ (see the proof of~\cite[Proposition~3.4]{P0}), 
	so we obtain the desired equality. 
	\end{proof}

Note that by Lemma~\ref{lemma:N(d)}, 
the product $\mathcal{E}_{d_1, v_1}
\ast \cdots \ast \mathcal{E}_{d_k, v_k}$
for $v/d=v_i/d_i$ for $1\leq i\leq k$
is an object of $\mathbb{T}(d)_{v}$. The result of Lemma~\ref{lemma:N(d)} also holds in the $T$-equivariant case, so the above product is also an object of $\mathbb{T}_T(d)_{v}$.
Let $\mathbb{F}=\mathbb{Q}(q_1, q_2)$ be the fraction 
field of $\mathbb{K}$. 
We show that the above products generate 
$K_T(\mathbb{T}(d)_v)$ over $\mathbb{F}$: 
\begin{thm}\label{prop:gen}
	We have 
	\begin{align*}
		K_T(\mathbb{T}(d)_v) \otimes_{\mathbb{K}}\mathbb{F}=\bigoplus_{\begin{subarray}{c}
			(d, v)=(d_1, v_1)+\cdots+(d_k, v_k), \\
			v/d=v_i/d_i
		\end{subarray}}
	\mathbb{F} \cdot [\mathcal{E}_{d_1, v_1}] \ast \cdots \ast
	 [\mathcal{E}_{d_k, v_k}]. 
		\end{align*}
	Here $(d, v)=(d_1, v_1)+\cdots+(d_k, v_k)$ is an unordered
	partition of $(d, v)$. 
	In particular, we have that
	\[\dim_{\mathbb{F}} K_T(\mathbb{T}(d)_v) \otimes_{\mathbb{K}}\mathbb{F}=p_2(\gcd(d, v)).
	\]
	\end{thm}
\begin{proof}
Let 
	\begin{align*}
		\mathcal{S} \subset \bigoplus_{d\geq 0} \mathbb{F}[z_1^{\pm 1}, \ldots, z_d^{\pm 1}]^{\mathfrak{S}_d}
		\end{align*}
	be the subalgebra for the shuffle product (\ref{def:shuffle})
	generated by $z_1^l$ for $l\in \mathbb{Z}$. 
	It is proved in~\cite[Theorem~4.6]{N2}
	that the morphism (\ref{i:ast}) induces an isomorphism of algebras
	\begin{align}\label{isom:S}
		i_{\ast} \colon \bigoplus_{d\geq 0}
		G_T(\mathcal{C}(d))\otimes_{\mathbb{K}}\mathbb{F}
		 \stackrel{\cong}{\to} \mathcal{S}. 
		\end{align}
		
	Consider the morphism 
	\begin{align}\label{mor:F}
	\bigoplus_{d\geq 0} \mathbb{F}[z_1^{\pm 1}, \ldots, z_d^{\pm 1}]^{\mathfrak{S}_d}
	\to \bigoplus_{d\geq 0} \mathbb{F}(z_1, \ldots, z_d)^{\mathfrak{S}_d}	
		\end{align}
	defined by 
	\begin{align*}
		f(z_1, \ldots, z_d) \mapsto f(z_1, \ldots, z_d) \cdot 
		\prod_{i\neq j}(1-q^{-1} z_i z_j^{-1})^{-1}. 
		\end{align*}
	Then (\ref{mor:F}) is an algebra homomorphism, where 
	the product on the right hand side is the shuffle product 
	(\ref{def:shuffle}) where we replace $\xi(x)$ with $w(x)$ defined by
	\begin{align*}
		w(x):=\frac{(1-q_1^{-1}x)(1-q_2^{-1}x)}{(1-x)(1-q^{-1}x)}. 
		\end{align*}
	Let $\mathcal{S}'$ be the subalgebra of the right hand side of (\ref{mor:F}), 
	generated by the elements of the form 
	\begin{equation}\label{Ambullet}
		A_{k_{\bullet}}:=
		\mathrm{Sym}\left(\frac{z_1^{k_1} \cdots z_d^{k_d}}{(1-q^{-1}z_1^{-1}z_2)\cdots 
			(1-q^{-1}z_{d-1}^{-1}z_d)} 
		\cdot \prod_{j>i}w(z_i z_j^{-1})  \right)
		\end{equation}
	for various $(k_1, \ldots, k_d) \in \mathbb{Z}^d$ and $d\geq 1$. 
	For $(d, v)\in\mathbb{N}\times\mathbb{Z}$, we set $A_{d, v}$ to be $A_{m_{\bullet}}$ 
	for the choice of
	$m_{\bullet}$ in (\ref{def:mi}). 
	It is proved in~\cite[Equation~(2.11)]{N}
	that we have 
	the following basis of $\mathcal{S}'$ as 
	a $\mathbb{F}$-vector space
	\begin{align}\label{basis:S'}
	\mathcal{S}'=\bigoplus_{v_1/d_1 \leq \cdots \leq v_k/d_k}
		\mathbb{F} \cdot A_{d_1, v_1} \ast \cdots \ast A_{d_k, v_k},
		\end{align}
		where the tuples $(d_i, v_i)_{i=1}^k$ appearing above
		are unordered for subtuples $(d_i, v_i)_{i=a}^b$
		with $v_a/d_a=\cdots=v_b/d_b$. 
	Since $z_1^l \mapsto A_{(l)}=A_{1,l}$ under the morphism (\ref{mor:F}), 
	it restricts to the morphism 
	$\mathcal{S} \to \mathcal{S}'$ which is obviously injective. 
	By Lemma~\ref{lem:iE}, we have 
	\begin{align*}
		(1-q_1^{-1})^{1-d}(1-q_2^{-1})^{1-d}
		[\mathcal{E}_{d, v}] \mapsto A_{d, v}
		\end{align*}
	under (\ref{mor:F}). 
	Therefore $\mathcal{S} \to \mathcal{S}'$ is also surjective, 
	hence it is an isomorphism \begin{equation}\label{SS'}
	    \mathcal{S} \stackrel{\cong}{\to}
	\mathcal{S}'.
	\end{equation}
	
	From (\ref{basis:S'}), 
	we have the following basis of $\mathcal{S}$
	as a $\mathbb{F}$-vector space
	\begin{align}\label{basis:S}
		\mathcal{S}=\bigoplus_{v_1/d_1 \leq \cdots \leq v_k/d_k}
		\mathbb{F} \cdot [\mathcal{E}_{d_1, v_1}] \ast \cdots \ast 
		[\mathcal{E}_{d_k, v_k}]. 
	\end{align}
On the other hand, the semiorthogonal decomposition \eqref{sod:C2}
gives the direct sum 
decomposition 
\begin{align}\label{decompose:G}
	\bigoplus_{d\geq 0} G_T(\mathcal{C}(d))
	=\bigoplus_{v_1/d_1<\cdots<v_k/d_k}
	K(\mathbb{T}_T(d_1)_{v_1} \boxtimes 
	\cdots \boxtimes \mathbb{T}_T(d_k)_{v_k}). 
	\end{align}
By comparing with (\ref{basis:S}) under the isomorphism (\ref{isom:S}), we obtain the desired claim. 
\end{proof}

\subsection{K-theoretic dimension formula}

Recall the categorical 
DT invariants from Definition~\ref{def:catDT2}. 
We consider the $K$-theoretic 
Donaldson-Thomas invariants, 
defined as the Grothendieck groups of the corresponding categories of matrix factorizations:
\begin{align*}
K_{\ast}(\mathcal{DT}^{\bullet}(d)), \ 
\ast \in \{\emptyset, T\}, \ \bullet \in \{\emptyset, \mathrm{gr}\}. 
	\end{align*}
	By~\cite[Corollary~3.13]{T4}, 
	we have the natural isomorphisms
	(which hold for all graded matrix factorizations): 
\begin{align}\label{isom:DTgrade}
K(\mathcal{DT}^{\mathrm{gr}}(d)) \stackrel{\cong}{\to} K(\mathcal{DT}(d)), \ 
K_T(\mathcal{DT}^{\mathrm{gr}}(d)) \stackrel{\cong}{\to} K_T(\mathcal{DT}(d)). 
\end{align}

Recall from Subsection \ref{partitions} that $p_3(d)$ is the number of plane partitions and that it satisfies MacMahon's formula (\ref{macmahon:formula}). 
Recall also that the DT invariants of $\mathbb{C}^3$ of degree $d$ are equal to $(-1)^dp_3(d)$, see \cite[Corollary 4.3]{BF}. 
We show that an analogous equality holds for localized $T$-equivariant Donaldson-Thomas invariants of $\mathbb{C}^3$.
Recall the object $\mathcal{F}_{d, w} \in \mathbb{S}(d)_w$ and $\mathbb{S}_T(d)_w$
defined in Definition~\ref{def:objF}. As a corollary of Theorem~\ref{prop:gen}, 
we have the following: 
\begin{cor}\label{thm:dim}
	We have 
	\begin{align}\label{basis:DT}
	    K_T(\mathcal{DT}(d)) \otimes_{\mathbb{K}}\mathbb{F}=
	    \bigoplus_{\begin{subarray}{c}d_1+\cdots+d_k=d \\
	    0\leq v_1/d_1 \leq \cdots \leq v_k/d_k <1
	    \end{subarray}} \mathbb{F} \cdot [\mathcal{F}_{d_1, w_1}] \ast \cdots \ast
	    [\mathcal{F}_{d_k, w_k}]
	\end{align} 	where the tuples 
	$(d_i, w_i)_{i=1}^k$
	appearing in the right hand side are 
unordered for 
   subtuples $(d_i, w_i)_{i=a}^b$ with 
   $v_a/d_a=\cdots =v_b/d_b$. In particular, we have 
	\begin{align}\label{dim:formula}
		\dim_{\mathbb{F}} K_T(\mathcal{DT}(d)) \otimes_{\mathbb{K}} \mathbb{F}
		=p_3(d). 
		\end{align}
	\end{cor}
\begin{proof}
We consider the category $\mathcal{DT}_T^{\text{gr}}(d)$ for the grading induced by the action of $\mathbb{C}^*$ on $\X(d)$ scaling the linear map corresponding to $Z$ with weight $2$.
	By Theorem~\ref{MacMahonthm} for $0<\mu<\frac{1}{d}$, 
	there is a semiorthogonal decomposition 
	\begin{align*}
		\mathcal{DT}_T^{\text{gr}}(d)=
			\left\langle \mathbb{S}_T^{\text{gr}}(d_1)_{w_1} \boxtimes 
			\cdots \boxtimes \mathbb{S}_T^{\text{gr}}(d_k)_{w_k} \relmiddle|
			\begin{array}{c}
				0\leq v_1/d_1 < \cdots < v_k/d_k<1, \\
				d_1+\cdots+d_k=d
			\end{array} \right\rangle.
		\end{align*}
		Here, the box product $\mathbb{S}_T^{\text{gr}}(d_1)_{w_1} \boxtimes 
			\cdots \boxtimes \mathbb{S}_T^{\text{gr}}(d_k)_{w_k}$ is taken over $BT$.
	Note that $\mathbb{S}^{\text{gr}}_T(d)_w$ is equivalent to $\mathbb{T}_T(d)_v$ for $v=w$ by the equivalence 
	(\ref{equiv:NM}). 
	Together with the fact that $v_i-w_i$ is divisible by $d_i$
	and the equivalence (\ref{equiv:periodic2}), 
	there is an equivalence 
	\begin{align*}
	    \mathbb{T}_T(d_1)_{v_1} \boxtimes \cdots \boxtimes 
		\mathbb{T}_T(d_k)_{v_k}
		\stackrel{\sim}{\to}	
		\mathbb{S}_T^{\text{gr}}(d_1)_{w_1} \boxtimes 
			\cdots \boxtimes \mathbb{S}_T^{\text{gr}}(d_k)_{w_k}. 
		\end{align*}
		By comparing (\ref{decompose:G}) with (\ref{sod:C}), 
	for $(d_1, v_1), \ldots, (d_k, v_k)$ 
	with $v_1/d_1<\cdots<v_k/d_k$, 
	we have 
	\begin{align*}
		K\left(\mathbb{T}_T(d_1)_{v_1} \boxtimes \cdots \boxtimes 
		\mathbb{T}_T(d_k)_{v_k}\right)\otimes_{\mathbb{K}}\mathbb{F}
		=\bigotimes_{i=1}^k K_T(\mathbb{T}(d_i)_{v_i}) \otimes_{\mathbb{K}}\mathbb{F},
		\end{align*}
		where the tensor product on the right hand side for $1\leq i \leq k$
		is over $\mathbb{K}$.
		Together with Theorem~\ref{prop:gen} and (\ref{isom:DTgrade}), we obtain (\ref{basis:DT}). 
		
	Recall that $p_2(m)$ is the number of partitions of $m$
	and that it satisfies the identity (\ref{macmahon:formula}). 
	Let $a_d:=\dim_{\mathbb{F}} K_T(\mathcal{DT}(d)) \otimes_{\mathbb{K}} \mathbb{F}$. 
	By the above arguments and by Theorem~\ref{prop:gen}, we have that
	\begin{align}\notag
	\sum_{d\geq 0} a_d q^d &=
		\sum_{0\leq v_1/d_1<\cdots<v_k/d_k<1}
		\prod_{i=1}^k p_2\left(\gcd(d_i, v_i)\right)q^{d_i} \\
		\label{sum:A}&=\prod_{\begin{subarray}{c}
				0\leq \mu<1 \\
		\mu=a/b, \ \gcd(a, b)=1
		\end{subarray}	
		}
	\prod_{k\geq 1}\frac{1}{1-q^{bk}}. 
		\end{align}
	For each $d \in \mathbb{Z}_{\geq 1}$, there is a bijection 
	\begin{align*}
		\{(k, a, b) \in \mathbb{Z}_{\geq 0}^3 : d=bk, \gcd(a, b)=1, 0\leq a<b\}
		\stackrel{\cong}{\to}
		\{0, 1, \ldots, d-1\}
		\end{align*}
	given by $(k, a, b) \mapsto ka$. In particular, the number of 
	the elements of the left hand side equals to $d$. 
	Therefore (\ref{sum:A})
	equals $M(q)$, hence the identity (\ref{dim:formula}) holds. 
	\end{proof}

	\bibliographystyle{amsalpha}
\bibliography{main}

\textsc{\small Tudor P\u adurariu: Columbia University, 
Mathematics Hall, 2990 Broadway, New York, NY 10027, USA.}\\
\textit{\small E-mail address:} \texttt{\small tgp2109@columbia.edu}\\

\textsc{\small Yukinobu Toda: Kavli Institute for the Physics and Mathematics of the Universe (WPI), University of Tokyo, 5-1-5 Kashiwanoha, Kashiwa, 277-8583, Japan.}\\
\textit{\small E-mail address:} \texttt{\small yukinobu.toda@ipmu.jp}\\

\end{document}